\documentclass[10pt]{amsart}

\usepackage{amssymb,amsmath,amsthm,amscd}

\advance\oddsidemargin by -1cm
\advance\evensidemargin by -1cm
\newtheorem{theorem}{Theorem}
\newtheorem{proposition}{Proposition}
\newtheorem{corollary}{Corollary}
\newtheorem{lemma}{Lemma}
\newtheorem{definition}{Definition}

\newcounter{anhang} 
  
\newtheorem{theorem_anh}[anhang]{Theorem} 
 
\newtheorem*{theoremA}{Theorem A}

\theoremstyle{definition}
\newtheorem{remark}{Remark}
\newtheorem{example}{Example}

\newtheorem{rem_anh}[anhang]{Remark}

\newcommand{\bdm}{\begin{displaymath}}
\newcommand{\edm}{\end{displaymath}}
\newcommand{\bq}{\begin{equation}}
\newcommand{\eq}{\end{equation}}
\newcommand{\bqn}{\begin{equation*}}
\newcommand{\eqn}{\end{equation*}}

\newcommand{\rn}{\mathbb{R}^n}

\newcommand{\eps}{\varepsilon}

\newcommand{\phw}{\tilde \psi^{wk}}
\newcommand{\Sing}{\mathrm{Sing}\,}
\newcommand{\Reg}{\mathrm{Reg}\,}

\newcommand{\norm}[1]{\left\| #1 \right\|}
\newcommand{\mklm}[1]{\left\{ #1 \right\}}
\newcommand{\eklm}[1]{\left\langle #1 \right\rangle}

\renewcommand{\d}{\,d}
\newcommand{\N}{{\mathbb N}}

\newcommand{\C}{{\mathbb C}}
\newcommand{\Ccal}{{\mathcal C}}
\newcommand{\R}{{\mathbb R}}
\newcommand{\A}{{\mathcal A}}

\newcommand{\D}{{\mathcal D}}

\newcommand{\F}{{\mathcal F}}

\newcommand{\J}{{\mathbb J }}
\newcommand{\M}{{\mathcal M}}
\newcommand{\T}{{\rm T}}

\newcommand{\X}{{\mathfrak X}}

\newcommand{\1}{{\bf 1}}
\renewcommand{\epsilon}{\varepsilon}
\renewcommand{\phi}{\varphi}
\renewcommand{\rho}{\varrho}

\newcommand{\Cinft}{{\rm C^{\infty}}}
\newcommand{\CT}{{\rm C^{\infty}_c}}

\newcommand{\Lcal}{{\mathcal L}}

\renewcommand{\S}{{\mathcal S}}

\newcommand{\GL}{\mathrm{GL}}

\newcommand{\g}{{\bf \mathfrak g}}

\renewcommand{\t}{{\bf \mathfrak t}}

\newcommand{\Ad}{\mathrm{Ad}\,}

\newcommand{\sgn}{\mathrm{sgn}\,}
\newcommand{\id}{\mathrm{id}\,}

\renewcommand{\det}{\mathrm{det}\,}
\renewcommand{\Im}{\mathrm{Im}\,}

\newcommand{\vol}{\text{vol}\,}

\newcommand{\Crit}{\mathrm{Crit}}

\DeclareMathOperator{\supp}{supp}
\DeclareMathOperator{\Res}{Res}

\DeclareMathOperator{\gd}{\partial}

\newcommand{\e}[1]{\,{\mathrm e}^{#1}\,}

\begin{document}

\author{Pablo Ramacher}
\title[Singular equivariant asymptotics and the momentum map]{Singular equivariant asymptotics and the momentum map. Residue formulae in equivariant cohomology} 
\address{Pablo Ramacher, Philipps-Universit\"at Marburg, Fachbereich Mathematik und Informatik, Hans-Meer\-wein-Str., 35032 Marburg, Germany}
\email{ramacher@mathematik.uni-marburg.de}
\date{October 7, 2013}
\begin{abstract}
Let $M$ be a smooth manifold and $G$  a compact connected Lie group acting on $M$ by isometries. In this paper, we study the equivariant cohomology of ${\bf X}=T^\ast M$, and relate it to the cohomology of the Marsden-Weinstein reduced space via certain residue formulae. In case that  $\bf X$ is a compact symplectic manifold with a Hamiltonian $G$-action, similar residue formulae were derived by Jeffrey, Kirwan et al. \cite{jeffrey-kirwan95, JKKW03}.

\end{abstract}

\maketitle

\setcounter{tocdepth}{1}
\tableofcontents

\section{Introduction}

Let $\bf X$ be a symplectic manifold carrying a Hamiltonian action of a compact, connected Lie group $G$ with Lie algebra $\g$, and denote the corresponding  momentum map by $\J: {\bf X} \rightarrow \g^\ast$.  In case that $\bf X$ is compact and $0$  a regular value of the momentum map,   the cohomology of the Marsden-Weinstein reduced space ${\bf{X}}_{red}=\J^{-1}(0)/G$ was expressed by Jeffrey and Kirwan \cite{jeffrey-kirwan95}  in terms of  the equivariant cohomology of $\bf X$ via  certain residue formulae. If $0$ is not a regular value, similar residue formulae were derived by them and their collaborators  \cite{JKKW03} for nonsingular, connected, complex projective varieties $\bf X$. These formulae rely on  the localization theorem  for compact group actions of Berline-Vergne \cite{berline-vergne82,berline-getzler-vergne}, and are related to the non-Abelian localization theorem of Witten \cite{witten92}. 
 The intention of  this paper is to extend their results to non-compact situations, and derive similar residue formulae  in  case that   $\bf X$ is given by the cotangent bundle of a $G$-manifold.

Let $\bf X$ be a smooth manifold carrying a smooth action of a connected Lie group $G$. According to Cartan \cite{cartan50}, its equivariant cohomology  can be defined by replacing the algebra  $\Lambda({\bf X})$ of smooth differential forms on $\bf X$ by the algebra $(S(\g^\ast) \otimes \Lambda({\bf X}))^G$ of $G$-equivariant polynomial mappings
\bqn 
 \rho: \g \ni X \longmapsto \rho(X) \in \Lambda(\bf{X}),
\eqn
where $\g$ denotes the Lie algebra of $G$. Let $\widetilde X$ denote the fundamental vector field on $\bf X$ generated by an element $X \in \g$. Defining  \emph{equivariant exterior differentiation} by
\bqn 
D\rho(X)=d(\rho(X)) - \iota_{\widetilde X} (\rho(X)), \qquad X \in \g, \,\rho \in (S(\g^\ast) \otimes \Lambda({\bf X}))^G,
\eqn
where $d$ and $\iota$ denote the usual exterior differentiation and contraction, 
the \emph{equivariant cohomology} of the $G$-action on $\bf{X}$ is  given by the quotient
\bqn 
H^\ast_G({\bf X})= \mathrm{Ker}\,  D\slash\,  \Im D,
\eqn
which is canonically isomorphic to the topological equivariant cohomology introduced in \cite{atiyah-bott84} in case that $G$ is compact, an assumption that we will make from now on. The main difference between  ordinary and equivariant cohomology is that the latter has a larger coefficient ring, namely $S(\g^\ast)$, and that it depends on the orbit structure of the underlying $G$-action. Let us now assume that $\bf{X}$ admits a symplectic structure $\omega$ which is left  invariant by $G$. By Cartan's homotopy formula,
\bqn 
0=\Lcal_{\widetilde X} \omega=d \circ \iota_{\widetilde X} \omega  + \iota_{\widetilde X}\circ  d\omega=d \circ \iota_{\widetilde X} \omega,
\eqn
where $\Lcal$ denotes the Lie derivative with respect to a vector field, 
implying that $\iota_{\widetilde X} \omega$ is closed for each $X \in \g$. $G$ is  said to act on $\bf{X}$ in a \emph{Hamiltonian fashion}, if this form is even exact, meaning that  there exists a linear function $J: \g \rightarrow \Cinft(\bf{X})$ such that for each $X \in \g$, the fundamental vector field $\widetilde X$ is equal to the Hamiltonian vector field of $J(X)$, so that
\bqn 
d(J(X))+ \iota_{\widetilde X} \omega=0.
\eqn 
An immediate consequence of this is that for any equivariantly closed form $\rho$ the form given by $e^{i(J(X)-\omega)} \rho(X)$ is equivariantly closed, too. Following Souriau and Kostant, one  defines the momentum map of a Hamiltonian action as the equivariant map
 \bqn 
 \J:{\bf{X}} \longrightarrow \g^\ast, \quad  \J(\eta)(X)=J(X)(\eta).
 \eqn
 Assume next that $0 \in \g^\ast$ is a regular value of $\J$, which is equivalent to the assumption that the stabilizer of each point of $\J^{-1}(0)$ is finite. In this case, $\J^{-1}(0)$ is a smooth manifold, and the corresponding \emph{Marsden-Weinstein reduced space}, or \emph{symplectic quotient}
 \bqn 
 {\bf X}_{red}= \J^{-1}(0)/G
 \eqn
 is an orbifold   with a unique symplectic form $\omega_{red}$ determined by the identity $\iota ^\ast \, \omega= \pi^\ast  \, \omega_{red}$, where $\pi: \J^{-1}(0) \rightarrow {\bf X}_{red}$ and $\iota:\J^{-1}(0)\hookrightarrow  {\bf X}$ denote the canoncial projection and inclusion, respectively. Furthermore, $\pi^\ast$ induces an isomorphism between $H^\ast({\bf X}_{red})$ and $H_G^\ast(\J^{-1}(0))$. Consider now the map 
 \bqn 
 {\mathcal{K}}: H^\ast_G({\bf X}) \stackrel{\iota^\ast}{\longrightarrow} H^\ast_G(\J^{-1}(0)) \stackrel{(\pi^\ast)^{-1}}{\longrightarrow} H^\ast ({\bf X}_{red}),
 \eqn
and assume that  ${\bf X}$ is compact and oriented. In this case, Kirwan \cite{kirwan84} showed that ${\mathcal{K}}$
 defines a surjective homomorphism, so that the cohomology of ${\bf X}_{red}$ should be computable from the equivariant cohomology of $\bf X$. This is the content of the  residue formula of Jeffrey and Kirwan \cite{jeffrey-kirwan95}, which for any $\rho \in H^\ast_G({\bf X})$ expresses the integral
 \bq
 \label{eq:1}
 \int_{{\bf X}_{red}} e^{-i\omega_{red}} {\mathcal{K}}(\rho)=\int_{{\bf X}_{red}} \sum_{k=0}^{\dim {\bf X}_{red}/2} \frac{(-i\omega_{red})^k}{k!} {\mathcal{K}}(\rho)_{[\dim {\bf X}_{red}-2k]}
 \eq
 in terms of data of ${\bf X}$. More precisely, let $T\subset G$ be a maximal torus, and ${\bf X}^T$ its fixed point set. Then 
 \eqref{eq:1} is given by a sum over the components $F$ of ${\bf X}^T$ of certain residues involving the restriction of $\rho$ to the $G$-orbit $G\cdot F$ and the equivariant Euler form $\chi_{NF}$ of the normal bundle $NF$ of $F$. The departing point of their work  is  the observation that the integral  \eqref{eq:1} should be given by the $\g$-Fourier transform of the tempered distribution
 \bqn 
\g \ni X \, \mapsto \, \int_{\bf X} e^{i(J(X)-\omega)}  \rho(X)
 \eqn
evaluated at $0\in \g^\ast$.  The mentioned formula of Jeffrey and Kirwan is then essentially a consequence of  the localization formula of Berline and Vergne \cite{berline-vergne82}. In   case that  $0\in \g^\ast$ is not a regular value,  analogous residue formulae were derived in \cite{JKKW03} for nonsingular, connected, complex projective varieties $\bf{X}$ within the framework of  geometric invariant theoretic quotients,  under some weak assumptions about the group action. In this situation, there is no longer a surjection from equivariant cohomology onto the cohomology of the corresponding quotient, whose singularities are worse than in the orbifold case. Nevertheless, their is still a surjection onto its intersection cohomology, which is a direct summand of the ordinary cohomology of any resolution of singularities of the quotient. Using a canonical desingularization procedure for such quotients  developed by Kirwan \cite{kirwan85}   in combination  with certain residue operations established by Guillemin and Kalkman \cite{guillemin-kalkman}, residue formulae for intersection pairings can then be derived.
 
Historically, the Berline-Vergne localization formula emerged as a generalization of a result of Duistermaat and Heckman \cite{duistermaat-heckman82} concerning the pushforward of the Liouville measure  of a compact, symplectic manifold carrying a Hamiltonian torus action along the momentum map. As it turns out,  this pushforward is a piecewise polynomial measure, or equivalently, its inverse Fourier transform is exactly given by the leading term in the stationary phase approximation. The study of the pushforward of the Liouville measure was motivated by attempts of finding an asymptotic approximation to the Kostant multiplicity formula \cite{kostant59} in order to examine the partition function occuring in that formula, which otherwise is very difficult to evaluate  \cite{guillemin-lerman-sternberg}. On the other side, the origin of the Berline-Vergne localization formula can be traced back to a residue formula for holomorphic vector fields derived by Bott \cite{bott67}, which was inspired by the generalized Lefschetz formula of Atiyah and Bott \cite{atiyah-bottI}.
 
 In this paper, we shall prove a residue formula in  case that ${\bf X}=T^\ast M$ is given by the cotangent bundle of a smooth manifold $M$ on which a compact, connected Lie group $G$ acts by general isometries. For this, we shall determine the asymptotic behavior of integrals of the form
\bqn
\label{eq:2}
I_\varsigma(\mu)=\int_{\g} \left [ \int_{\bf{X}} e^{i(\J(\eta)-\varsigma)(X)/\mu} a(\eta,X) \d \eta \right ] \d X,  \qquad \mu \to 0^+,
\eqn
via the stationary phase principle, where $\varsigma \in \g^\ast$,  $a \in \CT( \bf{X} \times \g)$ is an amplitude, $d\eta$  the Liouville measure on $\bf{X}$, and $dX$ denotes an Euclidean measure on $\g$ given by an $\Ad(G)$-invariant inner product on $\g$. While asymptotics for $I_\varsigma(\mu)$ can be easily obtained for free group actions, one meets with serious difficulties when singular orbits are present. The reason is that, when trying to examine these integrals in case that  $\varsigma \in \g^\ast$ is not a regular value of the momentum map, the critical set of  $(\J(\eta)-\varsigma)(X)$  is no longer  smooth, so that, a priori, the stationary phase  principle can not be applied in this case. Instead, we shall  circumvent this obstacle in the case  $\varsigma=0$  by partially resolving the singularities of the critical set of the momentum map,  and then apply the stationary phase theorem in a suitable resolution space. By this we are able to   obtain asymptotics for $I_0(\mu)$   with  remainder estimates  in the case of singular group actions.  
This  approach was developed first  in \cite{cassanas-ramacher09,ramacher10}  to describe the spectrum of an invariant elliptic operator on a compact $G$-manifold, where similar integrals occur, and used in the derivation of equivariant heat asymptotics in \cite{paniagua-ramacher12}. The asymptotic description of $I_\varsigma(\mu)$ in a neighborhood of $\varsigma =0$ then allows us to derive the following residue formula.   
Let $\rho\in H^\ast_G(T^\ast M)$ be of the form $\rho(X)=\alpha+D\beta(X)$, where $\alpha$ is a closed, basic differential form on $T^\ast M$ of compact support, and $\beta$ is an equivariant differential form of compact support. Fix a maximal torus $T\subset G$, and denote the corresponding root system by $\Delta(\g,\t)$. Assume that the dimension $\kappa$ of a principal $G$-orbit is equal to $d=\dim \g$, and 
denote  the product of the positive roots by $\Phi$. Let further $W$ be the Weyl group  and $H$ a principal isotropy group of the $G$-action. Denote the principal stratum of $\J^{-1}(0)$ by $\Reg \J^{-1}(0)$, and put ${\Reg {\bf X}_{red}}=\Reg \J^{-1}(0)/G$. Also, let  $r:\Lambda^\ast({\bf X}) \rightarrow \Lambda^{\ast-\kappa}(\Reg \J^{-1}(0))$ be the natural restriction map,  and write $\widetilde {\mathcal{K}}= (\pi^\ast)^{-1} \circ r$. Then, by Theorem \ref{thm:res},
\begin{align*}
(2\pi)^{d}  \int_{\Reg {\bf X}_{red}} \widetilde {\mathcal{K}}(e^{-i\omega} \alpha )= \frac{|H|}{|W| \,\vol T} \mathrm{Res} \Big (  \Phi^2 \sum_{F \in \F}  u_F\Big ),
\end{align*}
where   $\mathcal{F}$   denotes the set of components of the fixed point set  of the $T$-action on ${\bf X}=T^\ast M$, and the  $u_F$ are rational functions on $\t$  given by
\bqn 
u_F: \t \ni Y \, \longmapsto  \,  (-2\pi)^{\mathrm{rk}\, F/2} e^{iJ_Y(F)}  \int_F\frac{e^{-i\omega}\rho(Y)}{\chi_{NF}(Y)},
\eqn
$J_Y(F)$ being the constant value of $J(Y)$ on $F$. The definition of the residue operation, given in Section \ref{sec:2}, relies on the fact that the Fourier transform of $u_F $ is a piecewise polynomial measure.  Our approach is in many respects similar to the one of Jeffrey, Kirwan et al., but differs from their's in that it is based on the  stationary phase principle, which suggests that it should be possible to find a new proof of  their results, and  extend them to general symplectic manifolds.

 \medskip

{\bf Acknowledgements.} The author wishes to thank Mich\`ele Vergne for pointing out to him that the results in \cite{ramacher10} could be related to equivariant cohomology, and teaching him many things about the field. This research was financed in its beginnings by the grant RA 1370/2-1 of the German Research Foundation (DFG).

\section{Localization in equivariant cohomology}
\label{sec:2}

Let $\bf X$ be a $2n$-dimensional, paracompact, symplectic manifold with symplectic form $\omega$ and Riemannian metric $g$.  Since $\omega$ is non-degenerate, $\omega^n/n!$ yields a volume form on $\bf X$ called the \emph{Liouville form}, whose existence is equivalent to the fact that  $\bf X$ is orientable.
Define a bundle morphism $\mathcal{J}:T{\bf X} \rightarrow T{\bf X}$ by setting
\bqn 
g_\eta(\mathcal{J}  \X,  {\mathfrak{Y}}) =\omega_\eta( \X,  {\mathfrak{Y}}), \qquad  \X, \mathfrak{Y} \in T_\eta{\bf X},
\eqn
and assume that $\mathcal{J}$ is normed in such a way that $\mathcal{J}^2=-1$, which defines $\mathcal{J}$ uniquely.  $\mathcal{J}$ constitutes an almost-complex structure that is compatible with $\omega$, meaning that 
\bqn 
\omega_\eta (\mathcal{J} \X, \mathcal{J} \mathfrak{Y}) = 
\omega_\eta (\X, \mathfrak{Y}), \qquad \omega_\eta(\X, \mathcal{J} \X) >0.
\eqn
Furthermore, $g_\eta (\mathcal{J} \X, \mathcal{J} \mathfrak{Y}) =g_\eta (\X, \mathfrak{Y})$. $({\bf X},\mathcal{J},g)$ is consequently an almost-Hermitian manifold.  Next, assume that $\bf X$ carries a Hamiltonian action of a compact, connected Lie group $G$ of dimension $d$, and denote the corresponding Kostant-Souriau momentum map  by 
\bqn
\J: {\bf X} \to \g^\ast,  \quad \J(\eta)(X)=J_X(\eta)=J(X)(\eta).
\eqn
By definition, $dJ_X+ \iota_{\widetilde X} \omega =0$ for all $X \in \g$,
where $\widetilde X$ denotes the vector field on $  {\bf X}$ given by 
\bqn 
(\widetilde X f)(\eta)= \frac d {dt} f(e^{-tX} \cdot \eta)_{|t=0}, \qquad X \in \g, \quad f \in \Cinft( {\bf X}).
\eqn
 By this choice, the mapping $X\mapsto \widetilde X$ becomes a Lie-algebra homomorphism, so that in particular $\widetilde{[X,Y]}=[\widetilde X, \widetilde Y]$.  Also note that $\J$ is $G$-equivariant in the sense that $\J(g^{-1} \eta) = \Ad^\ast(g) \J(\eta)$. 

In what follows, we assume that $\g$ is endowed with an $\Ad(G)$-invariant inner product, which allows us to identify $\g^\ast$ with $\g$. Let further $dX$ and $d\xi$ be corresponding measures on $\g$ and  $\g^\ast$, respectively, and denote by 
$$\F_\g: \S(\g^\ast) \rightarrow \S(\g), \qquad \F_\g:\S'(\g) \rightarrow \S'(g^\ast)
$$ the $\g$-Fourier transform on the Schwartz space and the space of tempered distributions, respectively. In this paper, we intend  to relate  the equivariant cohomology $H_G^\ast ({\bf X})$ of $\bf X$ to the  cohomology of the symplectic quotient
\bqn 
{ {\bf X}}_{red}= \Omega_0/G, \qquad \Omega_\varsigma=\J^{-1}(\varsigma).
\eqn 
Following \cite{witten92} and \cite{jeffrey-kirwan95},  we consider for this  the map
\bqn 
X \quad \mapsto \quad L_\alpha(X)= \int_{{\bf X}} e^{iJ_X} \alpha, \qquad X \in \g, \qquad  \alpha \in \Lambda_c({\bf X}), 
\eqn
regarded as a tempered distribution in $\S'(\g)$, where $\Lambda_c({\bf X})$ denotes the algebra of differential forms on $\bf X$ of compact support. 
 If $({\bf X}, \omega)$ is a compact symplectic manifold, $G$ a torus, and $\alpha=\sigma^n/n!$ the Liouville measure, $L_\alpha$ is  the Duistermaat-Heckman integral, and corresponds to the inverse $\g$-Fourier transform of the pushforward $\J_\ast(\sigma^n/n!)$ of the Liouville form along the momentum map. In this case, the $\g$-Fourier transform of $L_\alpha$ is exactly $\J_\ast(\sigma^n/n!)$ and a piecewise polynomial measure on $\g^\ast$ \cite{duistermaat-heckman82}. 

We are therefore interested in  the $\g$-Fourier transform $\F_\g L_\alpha$ of $L_\alpha$ in general, and particularly, in its description near $0 \in \g^\ast$.  Take an $\Ad^\ast(G)$-invariant function $\phi \in \CT(\g^\ast)$ with total integral equal to one and $\g$-Fourier transform $\hat \phi(X) = (\F_\g \phi) (X) = \int_{\g^\ast} e^{-i\eklm{\xi,X}} \phi(\xi) \d \xi$, where we wrote $\xi(X)=\eklm{\xi,X}$. Then $\phi_\eps(\xi)= \phi(\eps^{-1} \xi  ) /\eps^d$, $\eps>0$, constitutes an approximation of the $\delta$-distribution  in $\g^\ast$ at $0$ as $\eps \to 0$,  and we consider the limit
\begin{align}
\label{eq:50}
\lim_{\eps \to 0} \eklm{\F_{\g} L_\alpha,\phi_\eps}&=\lim_{\eps \to 0}\int_{\g} L_\alpha(X) \hat \phi(\eps X) \d X=\lim_{\eps\to 0} \int_{\g} \int_{{\bf X}} e^{iJ_X/\eps} \alpha \,  \hat \phi(X) \frac {dX}{\eps^d},
\end{align}
where we took into account that $\hat \phi_\eps(X)= \hat \phi(\eps X)$. 
Next,   fix a maximal torus $T\subset G$  of dimension $d_T$ with Lie algebra $\t$, and consider the root space decomposition
\bqn 
\g^\C= \t^\C\oplus \bigoplus_{\gamma \in \Delta} \g_\gamma,
\eqn
where $\Delta=\Delta(\g,\t)$ denotes the set of roots of $\g$ with respect to $\t$, and $\g_\gamma$ are the corresponding root spaces. Since $\dim_\C \g_\gamma=1$, the decomposition  implies $d-d_T=\dim_\R \g-\dim_\R \t=|\Delta|$. Assume that $\alpha$ is such that $L_\alpha $ is $\Ad(G)$-invariant. Using  Weyl's integration formula \cite[Lemma 3.1]{jeffrey-kirwan95}, \eqref{eq:50} can be rewritten as
\bq
\label{eq:43}
\lim_{\eps \to 0} \eklm{\F_{\g} L_\alpha,\phi_\eps}=\frac{\vol G}{|W| \vol T} \lim_{\eps \to 0} \int_{\mathfrak{t}}\left [ \int_{{\bf X}}  e^{iJ_Y} \alpha \right ]\hat \phi(\eps Y) \Phi^2(Y) dY,
\eq
where $\Phi(Y)=\prod_{\gamma\in \Delta_+} \gamma(Y)$ and  $\Delta_+$ is the set of   positive roots, while $W=W(\g,\t)$ denotes the  Weyl group. Here $\vol G$ and $\vol T$ stand for the volumes of $G$ and $T$ with respect to the corresponding volume forms on $G$ and $T$ induced by the invariant inner product on $\g$ and its restriction to $\t$, respectively. In what follows, we shall express this limit in terms of  the set  
\bqn 
F^T=\mklm{\eta \in {\bf X}: t \cdot \eta =\eta \quad \forall \,  t \in T}
\eqn
of  fixed points of the underlying $T$-action. The  connected components  of $F^T$ are smooth submanifolds of possibly different dimensions, and we denote the set of these components by  $\F$. Let  $F \in \mathcal{F}$ be fixed, and consider the normal bundle $NF$ of $F$. As can be shown, the real vector bundle $NF$ can be given a complex structure, and splits into a direct sum of two-dimensional real bundles $P^F_q$, which can be regarded as complex line bundles over $F$. For each  $\eta \in F$, the fibers  $(P_q^F)_\eta$ are $T$-invariant, and endowing them with the standard complex structure, the action of  $\t$ can be written as 
\bqn 
(P_q^F)_\eta \ni v \mapsto i\lambda^F_q(Y) v \in (P_q^F)_\eta, \qquad Y \in \t,
\eqn
where the $\lambda^F_q\in \t^\ast$  are the weights of the torus action  \cite{duistermaat94}. They do not depend on $\eta$. Now, if $\rho$ is an  equivariantly closed form,  $L_{e^{-i\omega}\rho(Y)}(Y)$ can be computed using 

\begin{theorem}[Localization formula of Berline-Vergne] Let $\bf X$ be  a smooth $n$-dimensional manifold acted on by a compact Lie group $G$, and $\rho$ an equivariantly closed form on $\bf X$ with compact support. For $Y \in \g$, let ${\bf X}_0$ denote the zero set of $Y$. Then $\rho(Y)_{[n]}$ is exact outside ${\bf X}_0$, and 
\bqn 
\int_{\bf X} \rho(Y)= \int_{{\bf X}_0} (-2\pi)^{\mathrm{rk}\, N{\bf X}_0/2}\frac{\rho(Y)}{\chi_{N{\bf X}_0}(Y)},
\eqn 
where $N{\bf X}_0$ denotes the normal bundle of ${\bf X}_0$, which has been endowed with an orientation compatible with the one of ${\bf X}_0$, and $\chi_{N{\bf X}_0}$ is the equivariant Euler form of the normal bundle. 
\end{theorem}
\begin{proof}
The proof is the same as the proof of \cite[Theorem 7.13]{berline-getzler-vergne}, which consists essentially in a local computation, except for \cite[Lemma 7.14]{berline-getzler-vergne} which, nevertheless, can be easily generalized to equivariantly closed forms with compact support on non-compact manifolds. 
\end{proof}

To apply this theorem in our context, recall that an element $Y\in \t$ is called \emph{regular}, if the set $\mklm{\exp(sY): \, s \in \R}$ is dense in $T$.  The set of regular elements, in the following denoted by  $\t'$, is dense in $\t$, and 
\bq
 \label{eq:reg}
 \mklm{\eta \in {\bf X}: \widetilde Y_\eta =0}=F^T, \qquad Y \in \t'.
 \eq
We then have the following
\begin{corollary}
\label{cor:3}
Let $\rho\in H_G^\ast({\bf X})$ be an equivariantly closed form on $\bf X$ of compact support, and $Y\in \t'$. 
Then
\bqn
\label{eq:54}
L_{e^{-i\omega}\rho(Y)}(Y) = \int_{{\bf X}} e^{i(J_Y-\omega)} \rho(Y)=\sum_{F \in \F} u_F(Y),
\eqn
where the $u_F$ are rational functions on $\t$ given by 
\bq
\label{eq:47}
u_F: \t \ni Y \, \longmapsto  \,  (-2\pi)^{\mathrm{rk}\, NF/2} e^{iJ_Y(F)}  \int_F\frac{e^{-i\omega}\rho(Y)}{\chi_{NF}(Y)},
\eq
$J_Y(F)$ being  the constant value of $J_Y$ on $F$. 
\end{corollary}
\begin{proof}
Since $Y \mapsto e^{i(J_Y-\omega)}\rho(Y)$ defines an equivariantly closed form, the assertion follows immediately from the previous theorem and \eqref{eq:reg}.
\end{proof}
In the last corollary, the equivariant Euler class  is given by 
\bqn 
\chi_{NF}(Y)=\prod_{q} (c_1(P^F_q)+\lambda^F_q(Y)),
\eqn
where $c_1(P^F_q) \in H^2(F)$ denotes the first Chern class of the complex line bundle $P^F_q$. Thus,
\bqn 
\frac 1 {\chi_{NF}(Y)}= \frac 1{\prod_q \lambda^F_q(Y)}\prod_q \Big (1 + \frac{c_1(P^F_q)}{\lambda^F_q(Y)}\Big )^{-1}=\frac 1{\prod_q \lambda^F_q(Y)}\prod_q \sum_{0 \leq r_q} (-1)^{r_q} \Big ( \frac{c_1(P^F_q)}{\lambda^F_q(Y)}\Big )^{r_q}.
\eqn
Note that the sum in the last expression is finite, since ${c_1(P^F_q)}/{\lambda^F_q(Y)}$ is nilpotent. Consequently,  the inverse makes sense. Let us also note that the set of critical points of $J_X$ is given by 
\bqn 
\Crit \, J_X=\mklm{\eta \in {\bf X}: \widetilde X_\eta =0}, \qquad X \in \g,
\eqn
and is clean in the sense of Bott. Indeed, $\Crit\, J_X$ is a smooth submanifold consisting of possibly several components of different dimension. On the other hand, the Hessian of $J_X$ is given by the symmetric bilinear form
\bqn 
\mathrm{Hess} \, J_X: T_\eta({\bf X}) \times T_\eta({\bf X}) \longrightarrow \R, \quad (\X_1,\X_2) \mapsto (\tilde \X_1)_\eta(\tilde \X_2(J_X)), \qquad \eta \in \Crit \, J_X, 
\eqn
where $\tilde \X_2(J_X)=dJ_X(\tilde \X_2)=-\iota_{\widetilde X} \omega(\tilde \X_2)$, and  $\tilde \X$ denotes the extension of a vector $\X\in T_\eta({\bf X})$ to a vector field. Now,    
\begin{align}
\begin{split}
\label{eq:40cis}
\widetilde \X_i(\omega(\widetilde X,\widetilde \X_j))&=\mathcal{L}_{\widetilde \X_i }(\iota_{\widetilde X} \iota_{\widetilde \X_j} \omega)=\iota_{\mathcal{L}_{\widetilde \X_i}\widetilde X}\,  \iota_{\tilde \X_j} \omega+\iota_{\widetilde X} \, \mathcal{L}_{\widetilde \X_i} ( \iota_{\widetilde \X_j} \omega)\\
&=\iota_{\mathcal{L}_{\widetilde \X_i}\widetilde X}\,  \iota_{\widetilde \X_j} \omega+\iota_{\widetilde X} \, \iota_{\mathcal{L}_{\widetilde \X_i} \widetilde \X_j} \omega + \iota_{\widetilde X} \, \iota_{\widetilde \X_j} \mathcal{L}_{\widetilde \X_i} (\omega),
\end{split}
\end{align}
so that   at a point $\eta \in \Crit \, J_X$ one computes
\begin{align}
\begin{split}
\label{eq:40}
-\mathrm{Hess} \, J_X(\X_1,\X_2)&=\widetilde \X_1(\omega(\widetilde X,\tilde \X_2))=-\omega([\widetilde X, \widetilde \X_1], \widetilde \X_2),
\end{split}
\end{align}
since $\widetilde X$ vanishes on $\Crit \, J_X$. But the Lie derivative $ \X \mapsto (\mathcal{L}_{\widetilde X}  \widetilde \X)_\eta=[\widetilde X, \widetilde \X]_\eta$ defines an invertible endomorphism of $N_\eta \Crit \, J_X$. Consequently,  the Hessian of $J_X$ is transversally non-degenerate and $\Crit \, J_X$ is clean. 

We would like to compute \eqref{eq:43} using Corollary \ref{cor:3}, but since the rational functions  \eqref{eq:47} are not locally integrable on $\t$, we cannot proceed directly. Instead  note that, since $\Phi^2$ and $\hat \phi$ have analytic continuations to $\t^\C=\t\otimes \C$,  Cauchy's integral theorem yields  for arbitrary $Z\in \t$
\begin{align*}
\begin{split}
 \int_{\mathfrak{t}}\left [ \int_{{\bf X}}  e^{i(J_Y-\omega)} \rho(Y) \right ](\hat \phi_\eps \Phi^2)(Y) dY &= \int_{\mathfrak{t}}\left [ \int_{{\bf X}}  e^{i(J_{Y+iZ}-\omega)} \rho(Y+iZ) \right ](\hat \phi_\eps \Phi^2)(Y+iZ) dY.
  \end{split}
\end{align*}
Here we took into account that by the Theorem of Paley-Wiener-Schwartz \cite[Theorem 7.3.1]{hoermanderI} $\hat \phi_\eps(Y+iZ)$ is rapidly falling in $Y$. 
Let now $\Lambda$ be a proper cone in the complement of all the hyperplanes $\mklm{Y \in \t: \lambda_q^F(Y)=0}$, 
so that $Y \in \Lambda$ necessarily implies $\lambda^F_q(Y)\not=0$ for alle $q$  and $F$.  By the foregoing considerations,  $u_F$  defines a holomorphic function on $\t+i\Lambda$, and for arbitrary compacta $M\subset \mathrm{Int} \, \Lambda $, there is an estimate of the form
\bqn 
|u_F(\zeta)| \leq C(1+|\zeta|)^N, \qquad \zeta= Y+iZ, \quad \Im \zeta \in M,
\eqn
for some $N \in \N$.  The functions $u_F \Phi^k$, $k=0,1,2,\dots$, are holomorphic on $\t+i \Lambda$, too, and satisfy similar bounds. Then, by \cite[Theorem 7.4.2]{hoermanderI}, there exists for each $k$ a distribution $U_F^{\Phi^k} \in \D'(\t^\ast)$ such that 
\bq 
\label{eq:uFphi}
e^{-\eklm{\cdot, Z}} U_F^{\Phi^k} \in \S'(\t^\ast), \qquad \F^{-1}_{\t}(e^{-\eklm{\cdot, Z}} U_F^{\Phi^k})= (u_F \Phi^k)(\cdot + i Z), \qquad Z \in \Lambda.
\eq
We therefore obtain with Corollary \ref{cor:3} for arbitrary $Z \in \Lambda$ and $\varsigma \in \t^\ast$ the equality
\begin{align}
\label{eq:51a}
\begin{split}
 \int_{\mathfrak{t}}\left [ \int_{{\bf X}}  e^{i(J_Y-\omega)} \rho(Y) \right ]& (e^{-i\eklm{\varsigma, \cdot }} \hat \phi_\eps \Phi^2)(Y) dY =  \sum _{F \in \mathcal{F}}  \eklm{(u_F\Phi^2)( \cdot +iZ), (e^{-i\eklm{\varsigma, \cdot }} \hat \phi_\eps)(\cdot + i Z) }\\ &= \sum _{F \in \mathcal{F}}   \eklm{e^{-\eklm{\cdot, Z}} U_F^{\Phi^2},\F^{-1}_{\t} \big ( (e^{-i\eklm{\varsigma, \cdot }} \hat \phi_\eps)(\cdot + i Z)\big ) }\\&=  \sum _{F \in \mathcal{F}}  \eklm{ U_F^{\Phi^2},\F_\t^{-1}\big (e^{-i\eklm{\varsigma, \cdot }} \hat \phi_\eps \big )}.
 \end{split}
\end{align}

\begin{remark}
\label{rem:1}
Let us mention    that for arbitrary $\varsigma \in \t^\ast$
\bqn 
\F_\t^{-1}(e^{-i\eklm{\varsigma, \cdot }} \hat \phi_\eps)(\xi)=\frac 1 {\eps^{d_T}}(\F_\t^{-1} \hat \phi)\Big (\frac{\xi- \varsigma}\eps\Big ), \qquad \xi \in \t^\ast,
\eqn
constitutes an approximation of the $\delta$-distribution in $\t^\ast$ at $\varsigma$, since for arbitrary $v \in \CT(\t^\ast)$ 
\bqn 
\eklm{\F_\t^{-1}(e^{-i\eklm{\varsigma, \cdot }} \hat \phi_\eps), v}=\int_{\t^\ast} (\F_\t^{-1} \hat \phi)(\xi) v(\eps \xi + \varsigma) \d \xi \to v(\varsigma) \hat \phi(0) = v(\varsigma), \qquad \eps \to 0.
\eqn
\end{remark}
\begin{remark}
Alternatively, each of the summands in  \eqref{eq:51a} can be expressed as
\begin{gather*}
  \eklm{(u_F\Phi)( \cdot +iZ), (e^{-i\eklm{\varsigma, \cdot }}\Phi \hat \phi_\eps)(\cdot + i Z) }\\= (2\pi)^{|\Delta_+|}  \eklm{\F^{-1}_{\t} (e^{-\eklm{\cdot, Z}} U_F^\Phi), (e^{-i\eklm{\varsigma, \cdot }} \F_{\t} (\Phi \phi_\eps ))(\cdot + i Z)}= (2\pi)^{|\Delta_+|}  \eklm{  U_F^\Phi,(\Phi  \phi_\eps)( \cdot-\varsigma)},
\end{gather*}
where we used the equality $\Phi \hat \phi_\eps= \Phi \F_\g(\phi_\eps) = (2\pi)^{|\Delta_+|}  \F_{\t}(\Phi \phi_\eps )$, see \cite[Lemma 3.4]{jeffrey-kirwan95}, and the fact that  $ (e^{-i\eklm{\varsigma, \cdot }} \F_{\t} (\phi_\eps \Phi))(\cdot + i Z)=  \F_\t (e^{\eklm{\cdot,Z}} (\phi_\eps \Phi)(\cdot -\varsigma))$, or as
\begin{gather*}
 \eklm{u_F( \cdot +iZ), (e^{-i\eklm{\varsigma, \cdot }}\Phi^2 \hat \phi_\eps)(\cdot + i Z) }\\= (2\pi)^{|\Delta_+|}    \eklm{\F^{-1}_{\t} (e^{-\eklm{\cdot, Z}} U_F), (e^{-i\eklm{\varsigma, \cdot }}  \F_{\t} (D_\Phi (\Phi \phi_\eps )))(\cdot + i Z)}= (2\pi)^{|\Delta_+|}  \eklm{  U_F,D_\Phi (\Phi  \phi_\eps)( \cdot-\varsigma)},
\end{gather*}
where $D_\Phi$ denotes the differential operator such that $\F_\t(D_\Phi ( \Phi\phi_\eps))= \Phi \F_\t( \Phi\phi_\eps)$.
\end{remark}

 As a consequence of equations \eqref{eq:50}, \eqref{eq:43}, and     \eqref{eq:51a} we arrive at  
 
 \begin{proposition}
 \label{prop:A}
Let $\rho$ be an equivariantly closed differential form. Then
\begin{gather*}
\lim_{\eps \to 0} \eklm{\F_{\g} \Big ( L_{e^{-i\omega}\rho(\cdot)}(\cdot ) \Big ) ,\phi_\eps}=\lim_{\eps\to 0} \int_{\g} \int_{{\bf X}} e^{i(J_X/\eps-\omega)} \rho(X/\eps) \,  \hat \phi(X) \frac {dX}{\eps^d}\\ = \frac{\vol \, G}{|W| \vol \, T} \lim_{\eps \to 0}  \sum _{F \in \mathcal{F}}   \eklm{ U_F^{\Phi^2},\F_\t^{-1}\big ( \hat \phi_\eps \big )}.
\end{gather*}
\end{proposition}
\qed

In order to further investigate the distributions $U_F^{\Phi^k}$,  note that the functions $u_F\Phi^k$ are  given by a linear combination of terms of the form
\bqn 
\frac{e^{iJ_Y(F)}}{\Pi_q \lambda_q^F(Y)^{r_q}} P(Y), \qquad P \in \C[\t^\ast ]. 
\eqn
 The crucial observation is now that, due to this fact, the $u_F \Phi^k$ are tempered distributions whose $\t$-Fourier transforms  are \emph{piecewise polynomial} measures  \cite[Proposition 3.6]{jeffrey-kirwan95}. By the continuity of the Fourier transform in $\S'$ we therefore have
 \bqn 
 \F_\t (u_F\Phi^k)=\F_\t\Big(\lim_{t\to 0} u_F\Phi^k(\cdot +itZ)\Big )=\lim_{t\to 0} \F_\t( u_F\Phi^k(\cdot +itZ) )=\lim_{t \to 0} e^{-\eklm{\cdot, tZ}} U_F^{\Phi^k} =U_F^{\Phi^k}.
 \eqn
 Thus, $U_F^{\Phi^k}\in \S'(\t^\ast)$ is the $\t$-Fourier transform of $u_F\Phi^k$, and, in  particular, a piecewise polynomial measure. Motivated by Proposition \ref{prop:A}, we are interested in the behavior of $U_F^{\Phi^k}$ near the orgin, which leads us to the following
 
 \begin{definition}
 Let $\varsigma\in \t^\ast$ be such that for all $F \in \mathcal{F}$ the Fourier transforms $U_F^{\Phi^k}$ are smooth on the segment $t\varsigma$, $t \in (0, \delta)$. We then define the so-called \emph{residues}
 \bqn 
 \Res^{\Lambda, \varsigma}(u_F \Phi^k)= \lim_{t \to 0} U_F^{\Phi^k}(t\varsigma). 
 \eqn
 \end{definition}
 
 Note that the limit defining $ \Res^{\Lambda, \varsigma}(u_F \Phi^k)$ certainly exists, but does depend on $\varsigma$ (and $\Lambda$) as $U_F^{\Phi^k}$ is not continuous at the origin.  Furthermore, for arbitrary  $Z \in \Lambda$, 
  \begin{align*}
  \Res^{\Lambda, \varsigma}(u_F \Phi^k)&= \lim_{t \to 0} \lim_{\eps \to 0} \int_{\t^\ast} U_F^{\Phi^k}(\xi) \F_\t^{-1}(e^{-i\eklm{t\varsigma, \cdot }} \hat \phi_\eps)(\xi) \d \xi \\ 
  &= \lim_{t \to 0} \lim_{\eps \to 0} \eklm{ \F_\t^{-1}\big ( U_F^{\Phi^k} e^{-\eklm{\cdot, Z}}\big ) ,  \big (e^{-i\eklm{t\varsigma, \cdot }} \hat \phi_\eps\big )  (\cdot + iZ) }\\ 
    &=
  \lim_{t \to 0} \lim_{\eps \to 0}  \int_{\t} (u_F\Phi^k)(Y+iZ) e^{-i \eklm{t\varsigma,Y+iZ}} \hat \phi_\eps (Y+iZ) d Y,
 \end{align*}
 in concordance with the definition of the residues in \cite[Section 8]{jeffrey-kirwan95}. In particular, this implies
\begin{align}
\label{eq:sum1}
\begin{split}
  \sum_{F\in \F}  \Res^{\Lambda, \varsigma}(u_F \Phi^k)&= \lim_{t \to 0} \lim_{\eps \to 0} \int_{\t} \left [\int_{{\bf X}} e^{i(J - t\varsigma)(Y)} e^{-i\omega} \rho(Y) \right ] \Phi^k(Y) \hat \phi(\eps Y) \d Y. 
\end{split}
\end{align}
Similarly,
\begin{align*}
\label{eq:sum2}
\begin{split}
  \sum_{F\in \F} U_F^{\Phi^k}(\varsigma)&= \lim_{\eps \to 0} \int_{\t} \left [\int_{{\bf X}} e^{i(J - \varsigma)(Y)} e^{-i\omega} \rho(Y) \right ] \Phi^k(Y) \hat \phi(\eps Y) \d Y. 
\end{split}
\end{align*}
For a deeper understanding of  the residues and the limits in Proposition \ref{prop:A},   we are  therefore led to a  systematic study  of the asymptotic behavior of integrals of the form 
\bq
\label{int}
I_\varsigma(\mu)=   \int_{\g} \left [ \int _{{\bf X}}  e^{i \psi_\varsigma(\eta,X)/\mu }   a(\eta,X)   \, d\eta \right ] \d X, \qquad \mu \to 0^+,  
\eq 
where $\g$ is the Lie algebra of an arbitrary connected, compact Lie group $G$, $a \in \CT({\bf X} \times \g)$ is an amplitude, $\d \eta=\omega^n/n!$ the Liouville measure on ${\bf X}$,  and $dX$ an Euclidean measure on $\g$ given by an $\Ad(G)$-invariant inner product on $\g$, while
\bq
\label{eq:phase}
\psi_\varsigma(\eta,X) = \J(\eta)(X)-\varsigma(X), \qquad \varsigma \in \g^\ast.
\eq
This will occupy us in the next sections. 

\section{The stationary phase theorem and resolution of singularities}
\label{sec:6}

In what follows, we shall describe the  asymptotic behavior of the integrals $I_\varsigma(\mu)$ defined in \eqref{int} by means of the stationary phase principle. As we shall see,  the critical set of the corresponding phase function is in general not smooth. We shall therefore first partially resolve its singularities, and then apply the stationary phase principle in a suitable resolution space.
We begin by recalling 
\begin{theorem_anh}[Stationary phase theorem for vector bundles]
\label{thm:SP}
 Let $M$ be an $n$-dimensional, oriented manifold, and  $\pi:E \rightarrow M$ an oriented vector bundle of rang $l$.  Let further $\alpha \in \Lambda_{cv}^{q}(E)$ be a differential form on $E$ with compact support along the fibers, $\tau \in \Lambda^{n+l-q}_c(M)$ a differential form on $M$ of compact support,  $\psi \in \Cinft(E)$, and consider the integral
\bq
\label{eq:SPT}
I(\mu)=\int_E e^{i\psi/\mu} (\pi^\ast \tau) \wedge \alpha, \qquad \mu >0.
\eq
Let $\iota:M \hookrightarrow E$ denote the zero section. Assume that the critical set of $\psi$ coincides with $\iota(M)$, and that the transversal Hessian of $\psi$ is non-degenerate along $\iota(M)$. Then,  for each $N \in \N$,  $I(\mu)$ possesses an asymptotic expansion of the form
\bqn
I(\mu) = e^{i\psi_0/\mu}e^{i\frac{\pi}4\sigma_{\psi}}(2\pi \mu)^{\frac {l}{2}}\sum_{j=0} ^{N-1} \mu^j Q_j (\psi;\alpha,\tau)+R_N(\mu),
\eqn
where $\psi_0$ and $\sigma_{\psi}$ denote the value of $\psi$ and the signature of the transversal Hessian along $\iota(M)$, respectively. The coefficients $Q_j$ are given by measures supported on $M$,  and can be computed explicitly, as well as the remainder term $R_N(\mu)=O(\mu^{l/2+N})$. 
\end{theorem_anh}
\begin{proof}
See Appendix A. 
\end{proof}

If the critical set of the phase function is not smooth, the stationary phase principle can not be applied a  priori, and one faces serious difficulties in describing the asymptotic behavior of oscillatory integrals. We shall  therefore  first partially resolve the singularities of the critical set, and then apply the stationary phase principle in a suitable resolution space.  To explain our approach, let $\M$ be a smooth variety, $\mathcal{O}_\M$ the structure sheaf of rings of $\M $, and $I \subset \mathcal{O}_\M $ an ideal sheaf. The aim in the theory of resolution of singularities is  to construct a birational morphism $\Pi: \widetilde \M  \rightarrow \M $ such that $\widetilde \M $ is smooth, and the inverse image ideal sheaf $\Pi^\ast I$ is locally principal. This is called the \emph{principalization} of $I$, and implies resolution of singularities. That is, for every quasi-projective variety $\mathcal X$, there is a smooth variety $\widetilde {\mathcal{X}}$, and a birational and projective morphism $\pi:\widetilde {\mathcal{X}} \rightarrow \mathcal{X}$. Vice versa, resolution of singularities implies principalization. If $\Pi^\ast (I)$ is monomial, that is, if for every $\tilde x \in \widetilde \M $ there are local coordinates $\sigma_i$ and natural numbers $c_i$ such that 
 \bqn 
 \Pi^\ast (I) \cdot  \mathcal{O}_{\tilde x,\widetilde \M } = \prod_i \sigma_i^{c_i} \cdot \mathcal{O}_{\tilde x,\widetilde \M },
 \eqn
 one obtains strong resolution of singularities, which means that, in addition to the properties stated above, $\pi$ is an isomorphism over the smooth locus of $\mathcal{X}$, and $\pi^{-1} (\mathrm{Sing}\,  \mathcal{X})$ a divisor with simple normal crossings. 
Consider next the derivative  $D(I)$ of $I$, which is the sheaf ideal that is generated by all derivatives of elements of $I$. Let further $Z \subset \M $ be a smooth subvariety, and $\pi: B_Z \M  \rightarrow \M $ the corresponding monoidal transformation with center $Z$ and exceptional divisor $F \subset B_Z\M $. Assume that $(I,m)$ is a  marked ideal sheaf with $m \leq \mathrm{ord}_Z I$. The \emph{total transform} $\pi^\ast I$ vanishes along $F$ with multiplicity $\mathrm{ord}_Z I$, and by removing the ideal sheaf $\mathcal{O}_{B_Z\M }(-\mathrm{ord}_Z I \cdot F)$ from $\pi^\ast I$  we obtain the \emph{birational, or weak transform} $\pi_\ast^{-1} I$ of $I$. Take local coordinates $(x_1,\dots, x_n)$ on $\M $ such that $Z=(x_1= \dots =x_r=0)$. As a consequence,
\bqn
y_1=\frac {x_1}{x_r}, \dots, y_{r-1}=\frac{x_{r-1}}{x_r}, y_r=x_r, \dots,  y_n=x_n
\eqn
define local coordinates on $B_Z \M $, and for $(f,m) \in (I,m)$ one has
\bqn
\pi_\ast^{-1} (f(x_1,\dots,x_n),m)= (y_r^{-m} f (y_1y_r, \dots y_{r-1}y_r, y_r, \dots, y_n),m).
\eqn
By the work of Hironaka \cite{hironaka}, resolutions are known to exist, and we refer the reader to  \cite{kollar} for a detailed exposition.

Consider now an oscillatory integral of the form \eqref{eq:SPT}  in case that the critical set $\Ccal=\iota(M)\subset E=\M$ of the phase function $\psi$ is not clean. Let $I_{\Ccal}$  be the ideal sheaf of $\Ccal$, and $I_\psi=(\psi)$ the ideal sheaf generated by the phase function $\psi$. Then $D(I_\psi)=D_{\Ccal}$. 
The essential idea behind our approach  to singular  asymptotics is  to construct a  partial monomialization 
\bqn 
\Pi^\ast( I_\psi)  \cdot  \mathcal{O}_{\tilde x,\widetilde \M }= {\sigma}^{c_1}_{1}\cdots \sigma^{c_k}_{k}  \, \Pi^{-1}_\ast (I_\psi) \cdot  \mathcal{O}_{\tilde x,\widetilde \M }, \qquad \tilde x \in \widetilde \M , 
\eqn
 of the ideal sheaf $I_\psi=(\psi)$ via a suitable resolution  $\Pi:\widetilde \M  \rightarrow \M $  in such a way that $D(\Pi^{-1}_\ast (I_\psi))$ is a resolved ideal sheaf. As a consequence, the phase function factorizes locally according to $\psi \circ \Pi \equiv 
 {\sigma}^{c_1}_{1}\cdots \sigma^{c_k}_{k} \cdot \tilde \psi^{wk}$, and we show that the corresponding  weak transforms $\tilde \psi^ {wk}= \Pi_\ast^{-1}(\psi)$ have clean critical sets in the sense of Bott \cite{bott56}. Here   $\sigma_1, \dots, \sigma_k$ are local variables near each $\tilde x \in \widetilde \M $ and  $c_i$ are natural numbers. This  enables  one  to apply the stationary phase theorem in the resolution space $\widetilde \M $ to the weak transforms $ \tilde \psi^ {wk}$  with the variables  $\sigma_1, \dots,\sigma_k$ as parameters.  Note that by Hironaka's theorem,  $I_\psi$  can always be monomialized.  But in general, this monomialization would not be explicit enough  to allow an application of the stationary phase theorem. 

\section{Equivariant asymptotics and the momentum map}
\label{sec:4}

We commence now with our study of the asymptotic behavior of the integrals \eqref{int}  by means of the generalized stationary phase theorem. To determine the critical set of the phase function $\psi_\varsigma(\eta,X)$, let $\mklm{X_1, \dots, X_d}$ be a basis of $\g$, and write $X=\sum_{i=1}^d s_i X_i$. Due to the linear dependence of $J_X$ in $X$, 
\bqn 
\gd_{s_i} \psi_\varsigma (\eta,X) =J_{X_i}(\eta)-\varsigma(X_i),
\eqn
and because of  the non-degeneracy of $\omega$, 
\bqn 
 J_{X, \ast}=0 \quad \Longleftrightarrow \quad dJ_X=-\iota_{\widetilde X} \omega=0 \quad \Longleftrightarrow \quad \widetilde X =0.
\eqn
Hence,  
 \begin{align}
 \label{eq:4}
 \Crit(\psi_\varsigma)&=\mklm{ (\eta,X) \in {{\bf X}} \times \g: \psi_{\varsigma,\ast}  (\eta,X) =0}=\mklm{(\eta,X)\in \Omega_\varsigma \times \g: \widetilde X_\eta=0 },
\end{align}
where $  \Omega_\varsigma=\J^{-1}(\varsigma)$ is the $\varsigma$-level of the momentum map. 
Now, the major difficulty in applying the generalized stationary phase theorem in our setting  stems from the fact that, due to the singular orbit structure of the underlying group action,   $\Omega_\varsigma$ and, consequently, the considered critical set $\Crit(\psi_\varsigma)$, are in general singular varieties. In fact,  if the $G$-action on ${\bf X}$ is not free, $\Omega_\varsigma$ and the symplectic quotients $\Omega_\varsigma/G_\varsigma$ are no longer smooth for general $\varsigma \in \g^\ast$, where $G_\varsigma$ denotes the stabilizer of $\varsigma$ under the co-adjoint action.  Nevertheless,  both $\Omega_\varsigma$ and $\Omega_\varsigma/G_\varsigma$  have Whitney stratifications into smooth submanifolds, see Lerman-Sjamaar \cite{lerman-sjamaar}, and 
Ortega-Ratiu \cite[Theorems 8.3.1 and 8.3.2]{ortega-ratiu},  which correspond to the stratification of ${\bf X}$ into orbit types, see Duistermaat-Kolk \cite{duistermaat-kolk}. In particular, one has the following

\begin{lemma}
\label{lemma:0}
 $\Omega_\varsigma$ has a principal stratum 
 $\mathrm{Reg} \,\Omega_\varsigma$, which is an open and dense subset of $\Omega_\varsigma$,  and a smooth submanifold in ${\bf X} $ of codimension equal to the dimension $\kappa$ of a principal $G$-orbit in ${\bf X}$. Furthermore, 
\bq
\label{eq:7}
T_\eta(\Reg \, \Omega_\varsigma)=[T_\eta(G\cdot \eta)]^\omega=(\g \cdot \eta)^\omega, \qquad \eta \in \Reg \, \Omega_\varsigma,
\eq
where we denoted the symplectic complement of a subspace $V\subset T_\eta{\bf X}$ by $V^\omega$, and   wrote $\g \cdot \eta= \{\widetilde X_\eta: X \in \g\}$. 
\end{lemma}
\begin{proof}
Let $\Reg {\bf X}$ denote the union of all orbits of principal type in $\bf X$, so that $\Reg \Omega_\varsigma=\Omega_\varsigma \cap \Reg {\bf X}$. By the principal orbit theorem, $\Reg {\bf X}$ is open and dense, and the  assertion follows with \cite[Corollary 4.6.2 and (5.5.7)]{ortega-ratiu}. 
\end{proof}

Let us consider  first the case when $\varsigma\in \g^\ast$ is a regular value of the momentum map, which is equivalent to the fact  that $G$ acts \emph{locally freely} on $\Omega_\varsigma$, meaning that 
\bq
\label{eq:nonvan}
\widetilde X_\eta \not=0 \qquad \text{for all } \eta \in \Omega_\varsigma, \, 0\not= X \in \g.
\eq
Consequently, all stabilizers $G_\eta$ of points $\eta \in \Omega_\varsigma$ are finite, and therefore  either of principal or exceptional type. In this case, both $\Omega_\varsigma$ and $ \Crit(\psi_\varsigma)=\Omega_\varsigma \times \mklm{0}$ are smooth, and  $\dim \g \cdot \eta= \kappa$ for all $\eta \in \Omega_\varsigma$, where  $\kappa$ is the dimension of a principal $G$-orbit.  Furthermore, \eqref{eq:nonvan} implies that $\kappa=\dim \g$. We then  have the following

\begin{proposition}
\label{prop:regasymp}
Let $\bf X$ be a paracompact, symplectic manifold of dimension $2n$ with a Hamiltonian action of a compact Lie group $G$ of dimension $d$. 
Assume that $\varsigma \in \g^\ast$ is a regular value of the momentum map $\J:{\bf X} \rightarrow \g^\ast$, and let $I_\varsigma(\mu)$ be defined as in \eqref{int}.  Then, for each $N\in \N$, there exists a constant $C_{N,\psi_\varsigma,a}$ such that 
\bqn 
\Big |I_\varsigma(\mu) - (2\pi \mu)^\kappa \sum_{j=0}^{N-1} \mu^j Q_j(\psi_\varsigma,a) \Big | \leq C_{N,\psi_\varsigma, a}\,  \mu^N,
\eqn
where the coefficients $Q_j$ are given explicitly in terms of measures on $\Omega_\varsigma$. 
\end{proposition}
\begin{proof}
As already noted, $\Ccal_\varsigma=\Crit (\psi_\varsigma)=\Omega_\varsigma \times \mklm{0}$ is a smooth manifold of dimension $2\kappa$, and due to \eqref{eq:7} we have
\bqn 
T_{(\eta,0)} \Ccal_\varsigma \simeq T_\eta \Omega_\varsigma = (\g \cdot \eta)^\omega, \qquad N_{(\eta,0)} \Ccal_\varsigma= \mathcal{J}( \g \cdot \eta )\times \R^d,
\eqn 
 where $\mathcal{J}:T{\bf X} \rightarrow T{\bf X}$ denotes the bundle homomorphism introduced in Section \ref{sec:2}.
 By definition, the Hessian of $\psi_\varsigma$ at $(\eta,0)\in  \Ccal_\varsigma$ is given by the symmetric bilinear form 
\bqn 
\mathrm{Hess} \, \psi_\varsigma: T_{(\eta,0)}({\bf X}\times \g) \times T_{(\eta,0)}({\bf X}\times \g) \rightarrow \C, \qquad (v_1,v_2) \mapsto \tilde v_1(\tilde v_2(\psi_\varsigma))(\eta,0).
\eqn
Let $\{\widetilde \X_1, \dots, \widetilde \X_{2n}\}$ be a local orthonormal frame  in $T{\bf X}$ and $\mklm{e_1,\dots, e_d}$ the standard basis in $\R^d$ corresponding to an orthonormal basis $\mklm{A_1,\dots, A_d}$ of $\g$.  In the  basis
\bqn
((\widetilde \X_i)_\eta; 0), \qquad (0; e_j), \qquad i=1,\dots, 2n, \quad j=1,\dots,d,
\eqn
 of $T_{(\eta,X)}({\bf X} \times \g) = T_\eta{\bf X} \times \R^d$,  $\mathrm{Hess}\,  \psi_\varsigma$ is then given by the matrix
 \bqn 
 \mathcal{A}=-\left ( \begin{matrix}  0  &  \omega_\eta(\widetilde A_j, \widetilde \X_i) \\   \omega_\eta(\widetilde A_i,\widetilde \X_j) & 0 \end{matrix} \right )=-\left ( \begin{matrix}  0& g_\eta(\mathcal{J} \widetilde A_j, \widetilde \X_i) \\   g_\eta(\mathcal{J} \widetilde A_i,\widetilde \X_j) & 0 \end{matrix} \right ).
 \eqn
 Indeed, for arbitrary $X \in \g$ one has  $\widetilde \X_i(J_X)=dJ_X(\widetilde \X_i)=-\iota_{\widetilde X} \omega(\widetilde \X_i)$, and with \eqref{eq:40cis} we obtain $(\widetilde \X_i)_\eta(\omega(\widetilde 0,\widetilde \X_j))=0$.  
 In order to compute the transversal Hessian of $\psi_\varsigma$, we have to exhibit a basis for $N_{(\eta,0)} \Ccal_\varsigma$. Let therefore $\mklm{B_1, \dots, B_\kappa}$ be another basis of $\g=\g_\eta^\perp$ such that $\{(\widetilde B_1)_\eta, \dots, (\widetilde B_\kappa)_\eta\}$ is an orthonormal basis of $\g \cdot \eta$, where we remind the reader that $\kappa=d$.  It is then easy to see that 
\bqn
\mathcal{B}_k=(\mathcal{J} (\widetilde B_k)_\eta;0), \qquad \mathcal{B}'_k=( 0 ; g_\eta(\widetilde A_1,\widetilde B_k), \dots, g_\eta(\widetilde A_\kappa,\widetilde B_k)), \qquad k=1,\dots, \kappa,
\eqn
constitutes a basis of $N_{(\eta,0)} \Ccal_\varsigma$  with $\eklm{\mathcal{B}_k,\mathcal{B}_l}=\delta_{kl}$, $\mathcal{B}_k\perp \mathcal{B}_l'$,  and $\eklm{\mathcal{B}_k',\mathcal{B}_l'}=(\Xi)_{kl}$, where $\Xi$ is given by  the linear transformation
\bq
\label{eq:Xi}
\Xi: \g \cdot \eta \longrightarrow \g \cdot \eta: \X \mapsto \sum_{j=1}^\kappa g_\eta(\X, \widetilde A_j) (\widetilde A_j)_\eta.
\eq 
With these definitions one computes
\begin{align*}
\mathcal{A}(\mathcal{B}_k)=&\Big (0; -\sum_{j=1}^{2n} g_\eta(\mathcal{J} \widetilde A_1, \widetilde \X_j) g_\eta( \mathcal{J} \widetilde B_k, \widetilde \X_j), \dots\Big )\\=&(0; -g_\eta(\mathcal{J} \widetilde A_1,  \mathcal{J} \widetilde B_k), \dots,-g_\eta(\mathcal{J} \widetilde A_\kappa,  \mathcal{J} \widetilde B_k))=-\mathcal{B}_k',\\
\mathcal{A}(\mathcal{B}_k')=&\Big (- \Big (\sum\limits_{j=1}^{\kappa} g_\eta(\mathcal{J} \widetilde A_j, \widetilde \X_1) g_\eta(  \widetilde A_j, \widetilde B_k), \dots \Big ); 0 \Big )=((g_\eta(\Xi (\widetilde B_k)_\eta, \mathcal{J} \widetilde \X_1), \dots);0).
\end{align*}
 Since the $\{\mathcal{J} (\widetilde B_1)_\eta,\dots, \mathcal{J} (\widetilde B_\kappa)_\eta\}$  form an orthonormal basis of $\mathcal{J}( \g \cdot \eta)$, we obtain
 \bqn 
 \mathcal{A}(\mathcal{B}_k')=-(\mathcal{J} \Xi(\widetilde B_k)_\eta;0)=-\sum_{j=1}^\kappa g_\eta(\mathcal{J} \Xi(\widetilde B_k)_\eta, \mathcal{J} (\widetilde B_j)_\eta) \, \mathcal{B}_j.
 \eqn
 Thus, the transversal Hessian $\mathrm{Hess}  \, \psi_\varsigma(\eta,0)_{|N_{(\eta,0)} \Ccal_\varsigma}$ is given by the non-degenerate matrix
\bq
\label{eq:20}
\mathcal{A}_{\text{trans}}=
 \left ( \begin{matrix}  0 &-\1_\kappa \\ -\Xi_{|\g\cdot \eta}  & 0 \end{matrix} \right ).
\eq
By the non-stationary principle, we can choose  the support of the amplitude $a$ in the integral $I_\varsigma(\mu)$  close to $\Ccal_\varsigma$. Identifying a tubular neighborhood of $\Ccal_\varsigma$ with a neighborhood of the zero section in $N\Ccal_\varsigma$, the assertion now follows with Theorem \ref{thm:SP} by integrating along the fibers of $\nu: N\Ccal_\varsigma \to \Ccal_\varsigma$.
The exact form of the coefficients can  be read off from \eqref{eq:67}, in which $\psi''$ corresponds to  $\A_{\text{trans}}$. Note that  the submersion $P_\varsigma: \Ccal_\varsigma \rightarrow \Omega_\varsigma, (\eta,0) \mapsto \eta$ is simply the identity, so that 
measures on $\Ccal_\varsigma$ are identical with measures on $\Omega_\varsigma$. 
\end{proof}

Let us resume the considerations in Section \ref{sec:2}, the notation being the one introduced previously, and consider  the following, more specific oscillatory integrals. 

\begin{lemma}
\label{prop:exact}
Let $\rho=D\beta$ be an equivariantly exact form on ${\bf X}$ of compact support, $\varsigma \in \g^\ast$, and $\eps>0$. Then 
\bqn 
\int_\g  \left [  \int_{{\bf X}} e^{i(J-\varsigma)(X)}e^{-i\omega} \rho(X)\right ] \hat \phi_\eps(X) dX=0.
\eqn
\end{lemma}
\begin{proof}
The proof is essentially an elaboration of an argument given  in \cite[Equation (8.20)]{jeffrey-kirwan95}. In what follows, write $\bar \omega(X)=\omega-J_X$ for the extension of the symplectic form to an equivariantly closed form, and assume that $\beta=\sum \theta_j \beta_j$, $\theta_j \in S^j(\g^\ast)$, where the $\beta_j$ are  differential forms of compact support.  Let further $\phi \in \CT(\g^\ast)$  and $\delta=\delta(\eps)>0$ be such that $\supp \phi_\eps \subset B(0,\delta)$.  Define $\Delta_\delta= \mklm{\eta \in {\bf X}:  |\J(\eta)-\varsigma| <\delta}$, and let $\Delta_\delta \subset\Delta_\delta'$ be a smooth domain with smooth boundary $\gd \Delta_\delta'$.    Since $D\sigma(X)_{[2n]}=d(\sigma(X)_{[2n-1]})$ for any equivariant differential form $\sigma$, one computes
\begin{gather*}
\int_\g  \left [  \int_{{\bf X}} e^{-i\bar \omega(X)}\rho(X)\right ] e^{-i\varsigma(X)} \hat \phi_\eps(X) dX=\int_\g  \left [ \int_{{\bf X}} D\Big (  e^{-i\bar \omega}\beta\Big )(X)\right ] e^{-i\varsigma(X)} \hat \phi_\eps(X)  \d X\\
=\int_\g  \left [ \int_{{\bf X}} d\Big (  (e^{-i\bar \omega}\beta )(X)\Big ) \right ]e^{-i\varsigma(X)} \hat \phi_\eps(X)  \d X= \int_{{\bf X}} d\left ( \int_{\g}  e^{-i\varsigma(X)}\hat \phi_\eps(X)  (e^{-i\bar \omega}\beta )(X) \d X \right )\\
=\sum_j \int_{{\bf X}} d \left ( \int_{\g} e^{i(J-\varsigma)(X)} \hat \phi_\eps(X) \theta_j(X) \d X e^{-i\omega} \beta_j \right )\\
=\sum_j \int_{{\bf X}} d \left ( \int_{\g} e^{i(J-\varsigma)(X)} \F_\g( \theta_j(-i\gd_\xi) \phi_\eps)(X) \d X e^{-i\omega} \beta_j \right )\\
=(2\pi)^d\sum_j \int_{\Delta_\delta'} d \Big (  [(\theta_j(-i\gd_\xi) \phi_\eps) \circ (\J-\varsigma)] e^{-i\omega} \beta_j \Big )
\\=(2\pi)^d\sum_j \int_{\gd \Delta_\delta'}   [(-i\theta_j(\gd_\xi) \phi_\eps) \circ (\J-\varsigma)] e^{-i\omega} \beta_j =0
\end{gather*}
since $\phi_\eps \circ (\J-\varsigma)$ vanishes on $\gd \Delta_\delta'$. Hereby we used the Theorem of Stokes for differential forms with compact support, see \cite[page 119]{sternberg}.  
\end{proof}

\begin{proposition}
\label{prop:4}
Let $\varsigma\in \g ^\ast$ be a regular value of $\J:{\bf X} \rightarrow \g ^\ast$,   $\alpha \in \Lambda_c({\bf X})$, and $\theta \in S^r(\g ^\ast)$. 
Then  
\bqn 
\lim_{\eps \to 0} \int_{\g } \left [\int_{{\bf X}} e^{i(J - \varsigma)(X)}  \alpha \right ] \theta(X) \hat \phi (\eps X) \d X= \frac{(2\pi)^{d} \vol \,G}{|H_G|}  \int_{\J^{-1}(\varsigma)} \frac{\iota_\varsigma^\ast(F)}{\vol \, \mathcal{O}_G}
\eqn
for some form $F\in \Lambda_c({\bf X})$ explicitly given in terms of $\J$, $\alpha$ and $\theta$, where $H_G$ denotes a principal isotropy group of the $G$-action, and $\mathcal{O}_G(\eta)=G\cdot \eta$ the $G$-orbit through a point $\eta\in {\bf X}$, while $\iota_\varsigma: \J^{-1}(\varsigma) \hookrightarrow {\bf X}$ is the inclusion. 
\end{proposition}
\begin{proof}
Let $\psi_\varsigma(\eta,X)=(\J(\eta)-\varsigma)(X)$, so that the limit in question reads
\begin{align*}
\lim_{\eps \to 0} \frac 1{\eps^{d+r}}
\int_{\g} \left [\int_{{\bf X}} e^{i\psi_\varsigma/\eps}  \alpha \right ]  \theta \,  \hat \phi \d X.
\end{align*}
 Proposition \ref{prop:regasymp} yields for  the integral above  an asymptotic expansion with leading power $\eps^{d}$ and coefficients $Q_{r,j}$ given by measures on $\Ccal_\varsigma=\Crit (\psi_\varsigma)=\Omega_\varsigma\times \mklm{0}\equiv \Omega_\varsigma$.  In order to compute them, let $\mklm{\mathcal{B}_k, \mathcal{B}_l'}$ be the basis of $N_{(\eta,0)} \mathcal{C}_\varsigma$ introduced in the proof of Proposition \ref{prop:regasymp}, and let $\mklm{s_k,s_l'}$ be corresponding coordinates in  $N_{(\eta,0)} \mathcal{C}_\varsigma$. 
The transversal Hessian of $\psi_\varsigma$ is given by the matrix \eqref{eq:20}. By the non-stationary principle, we can choose  the support of $\alpha$ close to $\Omega_\varsigma$. Identify a tubular neighborhood of $\Omega_\varsigma$ with a neighborhood of the zero section in $N\Omega_\varsigma$. Integrating along the fibers of $\nu: N\Ccal_\varsigma\simeq N\Omega_\varsigma\times \g\to \Ccal_\varsigma$  then yields 
\begin{gather*}
\int_{\g} \left [ \int_{{\bf X}} e^{i\psi_\varsigma/\eps} \alpha \right ]  \theta  \,  \hat \phi \d X=  \int_{N\Ccal_\varsigma} e^{i\psi_\varsigma/\eps}    \theta \,   \hat \phi \,  \alpha \, dX 
= \int_{\Ccal_\varsigma} \nu_\ast \Big ( e^{i\psi_\varsigma/\eps}   \theta \,   \hat \phi \, \alpha  \, dX \Big ).
\end{gather*}
Assume now that with respect to the trivialization of $\nu$ given by the frame $\mklm{\mathcal{B}_k, \mathcal{B}_l'}$ we have
\bqn 
 \alpha \d X\equiv f \, \nu^\ast(\beta) \wedge ds \wedge ds', \qquad \beta \in \Lambda_c(\Omega_\varsigma),
\eqn
for some smooth function $f$. Applying \eqref{eq:67} we obtain for arbitrary large  $N\in \N$ an expansion of the form 
\begin{align} 
\begin{split}
\label{eq:21}
&\nu_\ast \Big ( e^{i\psi_\varsigma/\eps}   \theta \,  \hat \phi \,  \alpha  \, dX \Big ) \\=\frac \beta{\det (\A_{\text{trans}}(\eta,0)/2\pi i \eps)^{1/2}}
&\sum_{p-q<N}  \sum_{2p \geq 3 q} \frac{\eps^{p-q}}{p! \,  q!\,  i^j \,  2^p}\eklm{\mathcal{A}_{\text{trans}}^{-1} D, D } ^p (  \theta \,   \hat \phi  \, f \, H^q) ( \eta, 0)+ R_{N},
\end{split}
\end{align}
where $\eta \in \Omega_\varsigma$, $D=-i(\gd_{s_1}, \dots, \gd_{s_{\kappa}},\gd_{s_1'}, \dots, \gd_{s_{\kappa}'} )$, $(\theta \,  \hat \phi)(\eta,s,s')=(\theta \, \hat \phi)(X(s'))$, and  
\begin{align*}
H(\eta,s,s')&=\psi_{\varsigma}(\eta,s,s')- \eklm{\A_{\text{trans}} \Big ( \begin{array}{c} s\\ s' \end{array}\Big ),\Big ( \begin{array}{c} s\\ s' \end{array}\Big )}\Big / 2, \qquad 
\psi_{\varsigma}(\eta,s,s')= J_{X(s')} (\eta,s)-\varsigma(X(s')), 
\end{align*}
 is a smooth function vanishing at $(\eta,0)$ of  order $3$. The inner sum with $p-q=j$   therefore corresponds to a differential operator of order $2j$ acting on $\theta \,  \hat \phi \, f$, since in this case $2p - 3q=2j-q$, the maximal order being attained for $p=j$ and  $q=0$. Now,  since $\psi_\varsigma(\eta,X)$ depends linearly on $X$,  derivatives at $s'=0$ of $\psi_{\varsigma}(\eta, s,s')$, and consequently of $H(\eta, s,s')$, of order greater or equal $3$  vanish,  unless exactly one $s'$-derivative occurs. On the other hand, $\theta $ vanishes at $X(s')=0$ of order $r$. Furthermore, due to the particular form of $\mathcal{A}_{\text{trans}}$ in \eqref{eq:20},
\bqn 
\eklm{\mathcal{A}_{\text{trans}}^{-1} D, D } \equiv\sum c_{kl} \gd_{s_k} \gd_{s_l'}
\eqn
is a differential operator of first order in the $s'$-variables. Consequently, the inner sums in \eqref{eq:21} with  $p<r+q$ must vanish, and for $N=p-q=r$, only terms proportional to $\hat \phi(0)$ occur. 
Summing up we have shown that  
\bqn 
Q_{r,j}=0, \qquad \text{for all } \, j=0, \dots, r-1,
\eqn
 the leading term being  of order $\eps^{d+r}$, and we  obtain
\begin{gather*}
\lim_{\eps \to 0} \frac 1{\eps^{d+r}}
\int_{\g} \left [\int_{{\bf X}} e^{i(J - \varsigma)(X)/\eps} \alpha  \right ]  \theta (X)  \hat \phi(X) \d X\\
=  (2\pi)^{d} \hat \phi(0)\int_{\J^{-1}(\varsigma)}  \frac{ i^\ast_\varsigma(F)}{|\det  \Xi |^{1/2}} =  \frac{ (2\pi)^{d} \hat \phi(0)\, \vol G}{|H_G|} \int_{\J^{-1}(\varsigma)}   \frac{ i^\ast_\varsigma(F)}{\vol \mathcal{O}_G},
\end{gather*}
where $F\in \Lambda_c({\bf X})$ is explicitly given  in terms of $ \alpha$,  $\J$ and $\theta$. Here we took into account that  $|\det  \Xi_{|\g\cdot \eta} |^{1/2}=\vol (G\cdot \eta) \, |G_\eta|/ \vol G$ for $\eta \in \Omega_\varsigma$,   \cite[Lemma 3.6]{cassanas}.  Since $\hat \phi(0)=1$, the assertion follows.  
\end{proof}

Let $T\subset G$ be a maximal torus, and consider next the composition $\J_T: {\bf X} \rightarrow \t^\ast$ of the momentum map $\J$ with the restriction map from $\g^\ast$ to $\t^\ast$, which yields a momentum map for the $T$-action on $\bf X$. Then $\J_T^{-1}(\varsigma)/T_\varsigma\simeq \J_T^{-1}(\varsigma)/T$. Also, define
$$
 {\mathcal{K}}^T_\varsigma: H^\ast_T({\bf X}) \stackrel{\iota^\ast_{\varsigma,T}}{\longrightarrow} H^\ast_T(\J^{-1}_T(\varsigma)) \stackrel{(\pi^\ast_{\varsigma,T})^{-1}}{\longrightarrow} H^\ast (\J^{-1}_T(\varsigma)/T),
 $$ 
$\iota _{\varsigma,T}:\J^{-1}_T(\varsigma)\hookrightarrow {\bf X}$ being the inclusion, and $\pi_{\varsigma,T}:\J^{-1}_T(\varsigma) \rightarrow\J^{-1}_T(\varsigma)/T$ the canonical  projection.
In what follows, we shall also write $\Omega_\varsigma^T=\J_T^{-1}(\varsigma)$. We then have the following

\begin{proposition}
\label{prop:res}
Consider the segment $\mklm{t \varsigma: 0< t <1, \, \varsigma\in \t^\ast}$, and assume that it consists of regular values of $\J_T:{\bf X} \rightarrow \t^\ast$ and that all $U_F^{\Phi^2}$ are smooth on the segment. Then, if $\rho\in H^\ast_G({\bf X})$ is an equivariantly closed form of compact support, 
\bqn 
\sum_{F \in \F} \Res^{\varsigma,\Lambda}(u_F \Phi^2)= \frac{ (2\pi)^{d_T} \vol \,T}{|H_T|} \int_{\Reg \Omega_0^T/T}  {\mathcal{K}}^T_0 (F),
\eqn
where ${\mathcal{K}}_0^T= (\pi_{0,T} ^\ast)^{-1} \circ i_{0,T}^\ast$ is defined over $\Reg \Omega_0^T/T$, and $F$ is explicitly given in terms of $e^{-i\omega} \rho$, $\Phi$, and $\J$. 
In particular, the sum of the residues is independent of $\varsigma$ and $\Lambda$, and will be denoted by $$\Res \Big ( \Phi^2\sum_{F \in \F} u_F \Big ).$$
\end{proposition}
\begin{proof}
By \eqref{eq:sum1} and the previous proposition, 
\bq
\label{eq:2208}
\sum_{F \in \F} \Res^{\varsigma,\Lambda}(u_F \Phi^2)= \frac{ (2\pi)^{d_T} \vol T}{|H_T|}  \lim_{t \to 0}   \,  \int_{\Omega_{t\varsigma}^T/T} {\mathcal{K}}^T_{t \varsigma}(F),
\eq
where $d=\dim \g=\dim \t +|\Delta|=d_T+2 |\Delta_+|$.
We now assert that for sufficiently small $t>0$  there exists a birational map 
\bqn
\Xi_{t\varsigma}: \Omega^T_{t\varsigma}/T  \, \, \longrightarrow \, \,  \Omega^T_0/T
\eqn
which is a diffeomorphism over  $\Reg \Omega^T_0/T$. To see this,  consider an embedded resolution $\Pi: \widetilde {\bf X} \rightarrow {\bf X}$ of $\Omega^T_0$ \cite{bierstone-milman97}. By the functoriality of the resolution, the strict transform $\widetilde{\Omega}^T_0$  is a  $T$-invariant submanifold of the resolution space $\widetilde {\bf X}$, and there exists an invariant tubular neighborhood  $\widetilde W$ of $\widetilde{\Omega}^T_0$. Let $\tilde p: \widetilde W \rightarrow \widetilde \Omega^T_0$ be the canonical projection. For sufficiently small $t>0$, $\Pi^{-1}(\Omega^T_{t\varsigma})$ is contained in $  \widetilde W$. Since $\Omega^T_{t\varsigma}$ is diffeomorphic to $\Pi^{-1}(\Omega^T_{t\varsigma})$, which by Lemma \ref{lemma:0} is diffeomorphic to $\widetilde \Omega^T_0$,  we  obtain the birational map
\bqn 
\Omega^T_{t\varsigma} \stackrel{\Pi^{-1}}{\longrightarrow} \Pi^{-1}(\Omega^T_{t\varsigma})\, \stackrel{\tilde p}{\longrightarrow} \,\widetilde \Omega^T_0 \, \stackrel{\Pi}{\longrightarrow} \, \Omega^T_0.
\eqn
Dividing by $T$ then yields  the desired map $\Xi_{t\varsigma}$. As a consequence, we obtain with Lebesgue's theorem as $t \to 0$
\begin{align*}
\int_{\Omega^T_{t\varsigma}/T}{\mathcal{K}}^T_{t \varsigma}(F) &= \int_{\Xi_{\t\varsigma}^{-1}(\Reg \Omega^T_{0}/T)}{\mathcal{K}}^T_{t \varsigma}(F) = \int_{\Reg \Omega^T_{0}/T} (\Xi_{\t\varsigma}^{-1})^\ast ({\mathcal{K}}^T_{t \varsigma}(F))\, \to  \int_{\Reg \Omega^T_{0}/T} {\mathcal{K}}_0^T (F ),
\end{align*}
and the assertion follows with \eqref{eq:2208}.
\end{proof}

 \begin{corollary}
 \label{cor:A}
Let the notation be as in Section \ref{sec:2}, and  $\rho\in H^\ast_G({\bf X})$  an equivariantly closed differential form. Then
\begin{gather*}
\lim_{\eps \to 0} \eklm{\F_{\g} \Big ( L_{e^{-i\omega}\rho(\cdot)}(\cdot )\Big ) ,\phi_\eps}= \frac{\vol \, G}{|W| \vol \, T}\Res \Big (\Phi^2 \sum_{F \in \F} u_F \Big ).
\end{gather*}
\end{corollary}
\begin{proof}
Since $U_F^{\Phi^2}$ is a piecewise polynomial measure, and $\F^{-1}_\t (\hat \phi_\eps) \in \S(\t^\ast)$,
\bqn 
\eklm{ U_F^{\Phi^2},  \F_\t^{-1}(\hat\phi_\eps)}=\int_{\t^\ast}  U_F^{\Phi^2}(\eps\varsigma)  (\F_\t^{-1}\hat\phi)(\varsigma) \d \varsigma.
\eqn
Furthermore, for  $0<\eps \leq 1$ and almost every $\varsigma \in \t^\ast$ we have the estimate $| U_F^{\Phi^2}(\eps\varsigma)  (\F_\t^{-1}\hat\phi)(\varsigma)|\leq C (1+|\varsigma|)^N |(\F_\t^{-1}\hat\phi)(\varsigma)|$ for some $C,N>0$. Taking into account Remark \ref{rem:1}  and the  previous proposition, 
an application of  Lebesgue's theorem on bounded convergence then yields 
\bqn 
  \lim_{\eps \to 0}  \sum _{F \in \mathcal{F}}  \eklm{ U_F^{\Phi^2},  \F_\t^{-1}(\hat\phi_\eps)}=  \lim_{\eps \to 0} \int_{\t^\ast} \sum _{F \in \mathcal{F}} U_F^{\Phi^2}(\eps\varsigma)  (\F_\t^{-1}\hat\phi)(\varsigma) \d \varsigma=\hat \phi(0) \Res \Big ( \Phi^2\sum_{F \in \mathcal{F}} u_F  \Big ),
\eqn
and the assertion follows with Proposition \ref{prop:A}.
 \end{proof}

Thus, in order to derive the residue formula mentioned in the introduction, we are left with the task of evaluating the limit $\lim_{\eps \to 0} \eklm{\F_{\g} L_{e^{-i\omega}\rho(\cdot)}(\cdot ) ,\phi_\eps}$ in terms of the reduced space  ${\bf X}_{red}$. 
This amounts to an examination of the asymptotic behavior of the integrals \eqref{int} in case that  $\varsigma \in \g^\ast$, and in particular $\varsigma=0$,  is a singular value of the momentum map, in which case $\Crit (\psi_\varsigma)$ is a singular variety. From now on, we will only be considering the case $\varsigma=0$, and simply write  $\psi$ for $\psi_0$,  $I(\mu)$ for $I_0(\mu)$, and so on.   As explained in the previous section,  we shall partially resolve the singularities of the critical set $\Crit (\psi)$ first, and then make use of the stationary phase principle in a suitable resolution space.  Partial desingularizations  of the zero level set $\Omega=\J^{-1}(0)$ of the momentum map and the symplectic quotient $\Omega /G$ have been obtained   by Meinrenken-Sjamaar \cite{meinrenken-sjamaar} for compact symplectic manifolds with a Hamiltonian compact Lie group action by performing blowing-ups along minimal symplectic suborbifolds containing the strata of maximal depth  in $\Omega $.  In the context of geometric invariant-theoretic quotients, partial desingularizations were studied in \cite{kirwan85} and \cite{JKKW03}.

From now on,  we will restrict ourselves to the case where  ${\bf X}$ is given by the cotangent bundle of a   Riemannian manifold. For a general symplectic manifold, the  desingularization process should be similar, but more involved, and we intend  to deal with this case at some other occasion. 
 Thus, let  $M$ be  a Riemannian manifold of dimension $n$,    $\gamma:T^\ast M\rightarrow M$ its cotangent bundle,  and $\tau: T(T^\ast M)\rightarrow T^\ast M$ the tangent bundle, endowed with corresponding Riemannian structures \cite{mok75}. Define on $T^\ast M$ the Liouville form 
\bqn 
\Theta_\eta(\X)=\tau(\X)[\gamma_\ast(\X)], \qquad \X \in T_\eta(T^\ast M).
\eqn
We then regard $T^\ast M$ as a  symplectic manifold with symplectic form $\omega= d\Theta$ and Riemannian metric $g$. Assume now that $M$ carries an isometric action of   a compact, connected Lie group $G$ with Lie algebra $\g$, and define for every $X \in \g$ the function
\bqn
J_X: T^\ast M \longrightarrow \R, \quad \eta \mapsto \Theta(\widetilde{X})(\eta).
\eqn
Note that  $ \Theta(\widetilde {X})(\eta)=\eta(\widetilde{X}_{\pi(\eta)})$.
The function $J_X$ is linear in $X$, and due to the invariance of the Liouville form \cite{cannas-da-silva} one has 
\bqn
\mathcal{L}_{\widetilde X} \Theta = dJ_X+ \iota_{\widetilde X} \omega =0, \qquad \forall X \in \g,
\eqn
where $\mathcal{L}$ denotes the Lie derivative. Hence, the infinitesimal action of $X\in \g$ on $ T^\ast M$  is given by the Hamiltonian vector field defined by $J_X$, which means that  $G$ acts on $ T^\ast M$ in a Hamiltonian way. The corresponding symplectic momentum map  is  then given by 
\bqn
\J: T^\ast M \to \g^\ast,  \quad \J(\eta)(X)=J_X(\eta).
\eqn
Note that 
\bq
\label{eq:Ann}
\eta \in \Omega \quad \Longleftrightarrow \quad \eta _m \in \mathrm{Ann}(T_m (G\cdot m)) \quad \forall m \in M,
\eq
where $\mathrm{Ann} \, (V_m) \subset T_m^\ast M$ denotes the annihilator of a vector subspace $V_m \subset T_mM$. 

\begin{example}
In case that $M=\rn$, let $(q_1,\dots, q_n, p_1,\dots p_n)$ denote the canonical coordinates on $ T^\ast \rn\simeq \R^{2n}$. Let further $G \subset \GL(n,\R)$ be a closed subgroup acting  on $\T^\ast \rn$ by $g \cdot  (q,p)=(g \, q, \, ^Tg^{-1}\, p)$. The symplectic form  reads 
$ \omega= d\theta =\sum _{i=1}^n \d p_i  \wedge \, \d q_i$, 
where $\theta=\sum p_i \d q_i$ is the Liouville form,  and the corresponding momentum map is  given by 
\bqn
\J:T^\ast\rn\simeq {\rn}\times \rn\to \g^\ast,  \quad \J(q,p)(X)=\theta(\widetilde X)(q,p)=\eklm{ X q, p},
\eqn
where $\eklm{\cdot,\cdot}$  denotes the Euclidean inner product in $\rn$. In this case, for $\varsigma \in \g^\ast$, 
 \begin{align*}
 \Crit(\psi_\varsigma)&=\mklm{(q,p,X)\in \Omega_\varsigma \times \g: X \in \g_{(q,p)}},
\end{align*}
    where  $  \Omega_\varsigma= \mklm{(q,p) \in {T^\ast \rn}: \eklm {Aq,p}-\varsigma(A)=0 \text{ for all } A \in \g}$ and $\g_{(q,p)}$ is given by the set of all $X \in \g$ such that $ Xq=0, \, Xp=0$.
\end{example}

By Lemma \ref{lemma:0},  $\Omega $ has a principal stratum 
 $\mathrm{Reg} \,\Omega $, which is an open and dense subset of $\Omega $,  and a smooth submanifold in $T^\ast M $ of codimension equal to the dimension $\kappa$ of a principal $G$-orbit in $T^\ast M$. Furthermore, $
T_\eta(\Reg \, \Omega )=[T_\eta(G\cdot \eta)]^\omega=(\g \cdot \eta)^\omega, \, \eta \in \Reg \, \Omega$. 
We describe next the smooth part of the critical set \eqref{eq:4} for the phase function $\psi(\eta)(X)=\J(\eta)(X)$. 

\begin{lemma}
The smooth part of $\Crit(\psi)$ corresponds to 
\bq
\label{eq:z}
\mathrm{Reg}\,  \Crit(\psi)=\mklm{(\eta,X)\in \mathrm{Reg}\, \Omega  \times \g:  X\in \g_\eta},
\eq
and constitutes a submanifold of codimension $2\kappa$. Furthermore, 
\bq
\label{eq:10}
T_{(\eta,X)}\Reg \Crit (\psi )=\mklm{(\X,w) \in (\g \cdot \eta)^\omega \times \R^d: \sum_{i=1}^d w_i (\widetilde X_i)_\eta= [ \widetilde X, \widetilde \X]_\eta},
\eq
where $\widetilde \X$ denotes an extension of $\X$ to a vector field \footnote{In the proposition below, we shall actually see that  $[ \widetilde X, \widetilde \X]_\eta \in \g \cdot \eta$ for $X \in \g_\eta$ and $\X \in (\g \cdot \eta)^\omega$.}.
\end{lemma}
\begin{proof}
 Since the Lie algebra of $G_\eta$ is given by $\g_\eta=\{X\in \g: \widetilde X_\eta=0\}$, the first assertion is clear from \eqref{eq:4}. To see the second, let $(\eta(t),X(t))$ be a smooth curve in $\Reg \, \Omega  \times \g$. Writing $X(t)=\sum s_j(t)X_j$ with respect to a basis $\mklm{X_1, \dots, X_d}$ of $\g$, one computes for any $f \in \Cinft(\Reg \, \Omega )$
 \begin{align*}
 \frac d{dt} \widetilde{X(t)}_{\eta(t)}f_{|t=t_0}&=\sum_{j=1}^d \frac d{dt} \Big (s_j(t) (\widetilde{X_j})_{\eta(t)}f\Big ) _{|t=t_0}\\&=\sum_{j=1}^d \dot s_j(t_0) (\widetilde{X_j}f)(\eta(t_0))+ \sum_{j=1}^d  s_j(t_0) \frac d{dt} (\widetilde{X_j} f)(\eta(t))_{|t=t_0}. 
 \end{align*}
 Writing $\X=\dot \eta(t_0)\in T_{\eta(t_0)}\Reg \, \Omega $, one has $\frac d{dt} (\widetilde{X_j} f)(\eta(t))_{|t=t_0}=\widetilde \X_{\eta(t_0)}(\widetilde{X_j}f)$, so that if $(\eta(t),X(t))$ is a curve in $\Reg \Crit(\psi)$ one obtains
 \bqn
 \sum_{j=1}^d \dot s_j(t_0) (\widetilde{X_j})_{\eta(t_0)}f+ \sum_{j=1}^d  s_j(t_0) [\widetilde \X, \widetilde{X_j}]_{\eta(t_0)} f=0,
 \eqn
 since $\widetilde X(t_0)_{\eta(t_0)} (\widetilde \X f)=0$, and the assertion follows with \eqref{eq:7}.
\end{proof}

Before we start with the actual desingularization process of the phase function $\psi$,  let us mention the following 

\begin{proposition}
\label{prop:submersion}
The mapping $P:\Reg \, \Crit(\psi) \rightarrow \Reg \, \Omega , (\eta,X) \mapsto \eta$ is a submersion. 
\end{proposition}
\begin{proof}
Let $\eta \in \Reg \, \Omega $ and $X\in \g_\eta$. We show that $[\widetilde \X,\widetilde X]_\eta \in \g\cdot \eta$ for all $\X\in T_\eta \Reg \, \Omega $. To begin, note that $\pi_G:\Reg \, \Omega  \rightarrow \Reg \, \Omega /G$ is a submersion and a principal fiber bundle with $\ker (\pi_G)_{\ast,\eta}=\g \cdot \eta$ \cite[Theorem 8.1.1]{ortega-ratiu}.
If therefore $\eta(t) \in \Reg \, \Omega $ denotes a curve  with $\eta(0)=\eta$, $\dot \eta(0)=\X$, and $g \in G_\eta$, differentiation of $\pi_G(g \cdot \eta(t))= \pi_G(\eta(t))$ yields $\X-g_{\ast, \eta}(\X) \in  \ker (\pi_G)_{\ast,\eta}=\g \cdot \eta$. Consequently,
\bq
\label{eq:25}
\frac d{dt} (e^{-tX})_{\ast, \eta} \X_{|t=0}=\lim_{t\to 0} t^{-1}\big [(e^{-tX})_{\ast, \eta} \X-\X\big ] \in \g \cdot \eta,
\eq
 where we made the identification $T_\X(T_\eta \Reg \, \Omega )\simeq T_\eta \Reg \, \Omega $. Now, for arbitrary $Y \in \g$ \cite[Proposition 4.2.2]{ortega-ratiu}, 
 \bqn 
 \omega_\eta([\widetilde \X, \widetilde X],\widetilde Y)=-\omega_\eta([\widetilde X, \widetilde Y],\widetilde \X)-\omega_\eta([\widetilde Y, \widetilde \X],\widetilde X)=0,
 \eqn
since $\widetilde X_\eta=0$, and $\widetilde \X_\eta=\X \in (\g \cdot \eta)^\omega$. Hence, $[\widetilde \X, \widetilde X]_\eta \in (\g \cdot \eta)^\omega$. Furthermore, for arbitrary $f \in \Cinft(T^\ast M)$,
\begin{align*}
[\widetilde \X, \widetilde X]_\eta f=\widetilde \X_\eta( \widetilde X f)=\frac d {ds} (\widetilde X f)(\eta(s))_{|s=0}=\frac d {dt} \left ( \frac d{ds} f(e^{-tX} \cdot \eta (s))_{|s=0} \right )_{|t=0}=\frac d{dt} \big (\widetilde{(e^{-tX})_{\ast, \eta} \X_{|t=0}}\big )_\eta f, 
\end{align*}
so that with \eqref{eq:25}
\bq
\label{eq:11}
[\widetilde \X, \widetilde X]_\eta=\frac d{dt} (e^{-tX})_{\ast, \eta} \X_{|t=0}\in \g\cdot \eta.
\eq
 The previous lemma then implies that $P_{\ast, (\eta,X)}: T_{(\eta,X)} \Reg \Crit (\psi) \rightarrow T_\eta \Reg \, \Omega , (\X,w) \mapsto \X$ is a surjection, and the assertion follows.
\end{proof}

\begin{remark}
\label{rem:2}
Note that for $\eta \in \Reg \, \Omega$, and $X \in \g_\eta$, the previous proposition implies  that the Lie derivative  defines a homomorphism
\bq
\label{eq:LieX}
L_X: \g \cdot \eta \ni \X \longmapsto \Lcal_{\widetilde X}(\widetilde \X)_\eta=[\widetilde X, \widetilde \X]_\eta \in \g \cdot \eta.
\eq
\end{remark}

\section{The desingularization process   in the case  ${\bf X}=T^\ast M$, $\varsigma=0$}

We shall now proceed to a partial desingularization of  the critical set  of the phase function \eqref{eq:phase} for ${\bf X}=T^\ast M$, $\varsigma=0$, and  derive an asymptotic description of the integral \eqref{int} in this case. An analogous desingularization process was already implemented in \cite{ramacher10} to describe the asymptotic distribution of eigenvalues of an invariant elliptic operator. The desingularization employed here constitutes a local version of the latter, and for this reason is slightly simpler. Indeed, the phase function considered in \cite{ramacher10} is a global analogue of $\psi(\eta,X) = {\mathbb J}(\eta)(X)$. It should be noticed, however, that  these phase functions are not equivalent in the sense of Duistermaat \cite{duistermaat74}, so that  the corresponding desingularization processes can not be reduced to  each other \footnote{Observe that a similar phenomenon occurs in \cite{DKV2}.}.  To begin,  we shall need a suitable $G$-invariant covering of $M$. In its construction, we shall follow Kawakubo \cite{kawakubo}, Theorem 4.20.  For a more detailed survey on compact group actions, we refer the reader to \cite{ramacher10}, Section 3. Thus, let $(H_1), \dots,  (H_L)$ denote the isotropy types of $M$, and  arrange them in such a way that 
\bqn
H_j \text{ is conjugate to a subgroup of }H_i  \quad \Rightarrow \quad i \leq j.
\eqn
Let $H\subset G$ be a closed subgroup, and $M(H)$ the union of all orbits of type $G/H$. Then $M$ has a stratification into orbit types according to  
\bqn
M=M(H_1) \cup \dots \cup M(H_L).
\eqn
By the principal orbit theorem, the set $M(H_L)$ is open and dense in $M$, while $M(H_1)$ is a $G$-invariant submanifold. Denote by $\nu_1$ the normal $G$-vector bundle of $M(H_1)$, and by $f_1: \nu_1 \rightarrow M$ a $G$-invariant tubular neighbourhood of $M(H_1)$ in $M$. Take a $G$-invariant metric on $\nu_1$, and put
\bqn
{D}_t(\nu_1)=\mklm {v \in \nu_1: \norm{v} \leq t }, \qquad t >0.
\eqn
We then define the  $G$-invariant submanifold with boundary
\bqn
M_2=M - f_1(\stackrel{\circ}{D}_{1/2}(\nu_1)), 
\eqn
on which the isotropy type $(H_1)$ no longer occurs, and endow it with a $G$-invariant Riemannian metric with product form in a $G$-invariant collar neighborhood of $\gd M_2$ in $M_2$. Consider now the union $M_2(H_2)$ of orbits in $M_2$ of type $G/H_2$, a $G$-invariant submanifold of $M_2$ with boundary, and let $f_2:\nu_2 \rightarrow M_2$ be a $G$-invariant tubular neighborhood  of $M_2(H_2)$ in $M_2$, which exists due to the particular form of the metric on $M_2$. Taking a $G$-invariant metric on $\nu_2$, we define
\bqn
M_3=M_2 - f_2(\stackrel{\circ}{D}_{1/2}(\nu_2)), 
\eqn
which constitutes a  $G$-invariant submanifold with corners and isotropy types $(H_3), \dots (H_L)$. Continuing this way, one finally obtains for $M$ the decomposition 
\bqn  
M= f_1({D}_{1/2}(\nu_1)) \cup \dots  \cup f_L({D}_{1/2}(\nu_L)),
\eqn
where we identified  $f_L({D}_{1/2}(\nu_L))$ with $M_L$. This leads to the covering 
\bqn
M= f_1(\stackrel{\circ}{D}_{1}(\nu_1)) \cup \dots \cup f_L(\stackrel{\circ}{D}_{1}(\nu_L)),\qquad  f_L(\stackrel{\circ}{D}_{1}(\nu_L))=\stackrel{\circ} M_L.
\eqn

Let us now start resolving the singularities of the critical set $\Crit(\psi)$.  For this, we will  set up an iterative desingularization process along the strata of the underlying $G$-action, where each step in our iteration will consist of a decomposition, a monoidal transformation, and a reduction. For simplicity, we shall assume that at each iteration step the set of maximally singular orbits is connected. Otherwise each of the connected components, which might even have different dimensions,  has to be treated separately. 

\subsection*{First decomposition} Take $1\leq k \leq L-1$. As before,  let $f_k:\nu_k\rightarrow M_k$ be an invariant tubular neighborhood of $M_k(H_k)$ in 
\bdm
M_k=M-\bigcup_{i=1}^{k-1} f_i(\stackrel{\circ}{D}_{1/2}(\nu_i)),
\edm 
a manifold with corners on which $G$ acts with the isotropy types $(H_k), (H_{k+1}), \dots, (H_L)$, and put  $W_k=f_k(\stackrel{\circ}{D_1}(\nu_k))$, $W_L=\stackrel{\circ}{M_L}$, so that $M= W_1 \cup \dots \cup W_L$. Write further $S_k= \mklm{v \in \nu_k: \norm{v}=1}$. 
Introduce a partion of unity $\mklm{\chi_k}_{k=1,\dots,L}$ subordinate to the covering $\mklm{W_k}$, and with the notation of \eqref{int} define 
\bqn
I_k(\mu)=   \int _{T^\ast W_k}  \int_{\g} e^{i \psi(\eta,X)/\mu }   (a\chi_k)(\eta,X)  \d X \d \eta,   
\eqn
so that $I(\mu)=I_1(\mu)+\dots +I_L(\mu)$. 
As will be explained in Lemma \ref{lemma:Reg}, the critical set of $\psi$ is clean on the support of $a\chi_L$, so that we can apply  directly the stationary phase theorem to compute the integral $I_L(\mu)$. But if $k \in \mklm{1, \dots, L-1}$, the sets 
\begin{align*}
  \Omega_k&=\Omega \, \cap \, T^\ast W_k, \\
   \Crit_k(\psi) &=\mklm{(\eta,X)\in \Omega_k \times \g: \widetilde X_\eta=0}
\end{align*}
are no longer smooth manifolds, so that the stationary phase theorem can not  a priori be applied  in this situation. Instead, we shall resolve the singularities of $\Crit_k(\psi)$, and after this  apply the principle of the stationary phase in a suitable resolution space. For this, introduce for each $x^{(k)}\in M_k(H_k)$ the decomposition
\bqn
\g=\g_{x^{(k)}}\oplus \g_{x^{(k)}}^\perp, 
\eqn
where $\g_{x^{(k)}}$ denotes the Lie algebra of the stabilizer $G_{x^{(k)}}$ of $x^{(k)}$, and $\g_{x^{(k)}}^\perp$ its orthogonal complement with respect to some $\Ad(G)$-invariant inner product in $\g$. Let further $A_1(x^{(k)}), \dots, A_{d^{(k)}}(x^{(k)})$ be an orthonormal basis of $\g_{x^{(k)}}^\perp$, and $B_1(x^{(k)}),\dots, B_{e^{(k)}}(x^{(k)})$ an orthonormal basis of $\g_{x^{(k)}}$. Consider  the isotropy algebra bundle over $M_k(H_k)$
\bqn
\mathfrak{iso} \,M_k(H_k) \rightarrow M_k(H_k),
\eqn
as well as the canonical projection
\bqn 
\pi_k: W_k \rightarrow M_k(H_k), \qquad m=f_k(x^{(k)},v^{(k)}) \mapsto x^{(k)}, \qquad x^{(k)} \in M_k(H_k), \, v^{(k)} \in (\nu_k)_{x^{(k)}},
\eqn
where $f_k(x^{(k)},v^{(k)})=(\exp_{x^{(k)}} \circ \gamma ^{(k)})( v^{(k)})$, and 
\bqn 
\gamma^{(k)}(v^{(k)}) = \frac {F_k(x^{(k)})}{(1+ \norm {v^{(k)}})^{1/2}} v^{(k)}
\eqn is an equivariant diffeomorphism from $(\nu_k)_{x^{(k)}}$ onto its image, $F_k:M_k(H_k) \rightarrow \R$ being a smooth, $G$-invariant positive function, see Bredon \cite[pages 306-307]{bredon}. We consider then the induced bundle
\bqn
\pi_k^\ast \frak{iso}\,  M_k(H_k)=\mklm {(f_k(x^{(k)},v^{(k)}),X)\in W_k \times \g: X \in \g_{x^{(k)}}},
\eqn
and denote by  
$$\Pi_k: W_k \times \g \rightarrow  \pi_k^\ast \frak{iso} \, M_k(H_k)$$
the canonical projection which is obtained by considering geodesic normal coordinates around $\pi_k^\ast \, \frak{iso} M_k(H_k)$, and  identifying  $W_k\times \g$ with a neighborhood of the zero section in  the normal bundle $N\, \pi_k^\ast \,\frak{iso} \, M_k(H_k)$. Note  that 
 the fiber of the normal bundle to $\pi^\ast \frak{iso}\,  M_k(H_k)$ at a point $(f_k(x^{(k)},v^{(k)}),X)$ can be identified with $\g_{x^{(k)}}^\perp$.  Integrating along the fibers of the normal bundle to  $\pi_k^\ast \, \frak{iso} M_k(H_k)$ we therefore obtain  for $I_k(\mu)$ the expression
\begin{align}
\label{eq:*}
\begin{split}
&\int_{\pi_k^\ast \, \frak{iso} M_k(H_k)} \left [\int_{\Pi_k^{-1}(m,B^{(k)})\times T^\ast_m W_k
} e^{i\psi/\mu} a\chi_k \, \Phi_k \, \d (T^\ast_m W_k) \, dA^{(k)}   \right ]  dB^{(k)} \d m \\
 =\int_{ M_k(H_k) }&\Big [\int_{ \g \times \pi_k^{-1}(x^{(k)})}\Big [\int_{T^\ast_{\exp_{x^{(k)}} v^{(k)}} W_k
} e^{i\psi/\mu} a\chi_k \, \Phi_k \, \d (T^\ast_{\exp_{x^{(k)}} v^{(k)}}W_k) \Big ]  dA^{(k)} \, dB^{(k)} \, dv^{(k)}  \Big ]      dx^{(k)},
\end{split}
\end{align}
where 
\bqn
\gamma^{(k)}\big ( \stackrel \circ D_1(\nu_k)_{x^{(k)}} \big ) \times  \g_{x^{(k)}}^\perp \times   \g_{x^{(k)}} \ni (v^{(k)},  A^{(k)},B^{(k)})\mapsto (\exp_{x^{(k)}} v^{(k)},A^{(k)}+B^{(k)})=(m,X)
\eqn
are coordinates on $\g \times  \pi_k^{-1}(x^{(k)})$, while $dm$, $dx^{(k)}$, $dA^{(k)}, dB^{(k)} ,  d v^{(k)}$, and $ \d (T^\ast_m W_k)$ are suitable measures on $W_k$,   $M_k(H_k)$, $\g_{x^{(k)}}^\perp$, $\g_{x^{(k)}}$, $\gamma^{(k)}(\stackrel \circ D_1(\nu_k)_{x^{(k)}})$, and $T^\ast_m W_k$, respectively, such that 
\bqn \d X \d \eta \equiv\Phi_k \d (T^\ast_{\exp_{x^{(k)}} v^{(k)}}W_k)(\eta)  dA^{(k)} \, dB^{(k)} \, dv^{(k)} \, dx^{(k)},
\eqn
where $\Phi_k$ is a Jacobian.

\subsection*{First monoidal transformation}  Let now $k \in \mklm{1, \dots, L-1} $ be fixed. For the further analysis of the integral $I_k(\mu)$, we shall sucessively resolve the singularities of $\Crit_k(\psi)$, until we are in position to apply the principle of the stationary phase in a suitable resolution space. To begin with, we perform a monoidal transformation 
\bqn 
\zeta_k: B_{Z_k}( W_k \times \g) \longrightarrow W_k \times \g
\eqn
 in $W_k \times \g$ with center $Z_k= \frak{iso} \, M_k(H_k)$. For this, let us write  $ A^{(k)}(x^{(k)},\alpha^{(k)})=\sum  \alpha_i^{(k)} A_i^{(k)}(x^{(k)})\in \g^\perp_{x^{(k)}}$,  $ B^{(k)}(x^{(k)},\beta^{(k)})=\sum  \beta_i^{(k)} B_i^{(k)}(x^{(k)})\in \g_{x^{(k)}}$, and 
\bqn
\gamma^{(k)}( v^{(k)}) = \sum _{i=1}^{c^{(k)}} \theta_i^{(k)} v_i^{(k)}(x^{(k)})\in \,  \gamma^{(k)} \Big ( \stackrel \circ D_1(\nu_k)_{x^{(k)}}\Big ),
\eqn
where $\{v_1^{(k)},\dots  ,v_{c^{(k)}}^{(k)}\}$ denotes an orthonormal frame in $\nu_k$. With respect to these coordinates we have $Z_k=\mklm{T^{(k)}=(\alpha^{(k)},\theta^{(k)})=0}$,  so that 
\begin{gather*}
B_{Z_k}( W_k \times \g)=\mklm{ (m,X,[t]) \in W_k \times \g \times \mathbb{RP}^{c^{(k)} + d^{(k)}-1}:  T^{(k)}_i t_j = T^{(k)}_j t_i,   },\\
\zeta_k: (m,X,[t]) \longmapsto (m,X).
\end{gather*}
Let us  now cover $B_{Z_k}( W_k \times \g)$ with  charts $\mklm{(\phi_k^\rho, U_k^\rho)}$,  where $U_k^\rho=B_{Z_k}( W_k \times \g)\cap (W_k \times \g \times V_\rho)$,  $V_\rho=\mklm{[t] \in  \mathbb{RP}^{c^{(k)} + d^{(k)}-1}: t_\rho\not=0}$, and  $\phi_k^\rho$ is given by the canonical coordinates on $V_\rho$. As a consequence, we obtain for $\zeta_k$ in each of the $\theta^{(k)}$-charts $\mklm{U_k^\rho}_{1\leq \rho\leq c^{(k)}}$ the expressions
\begin{align}
\label{eq:**}
\begin{split}
 \zeta_k ^\rho&=\zeta_k \circ \phi_k^\rho: ( x^{(k)},\tau_k,\, ^\rho \tilde v^{(k)},  A^{(k)}, B^{(k)})\stackrel{{' \zeta_k^ \rho}}{\longmapsto}  ( x^{(k)},\tau_k \, ^\rho \tilde v^{(k)}, \tau_k A^{(k)}, B^{(k)}) \\
& \longmapsto (\exp_{x^{(k)}} \tau_k  \,^\rho\tilde v^{(k)}, \tau_k A^{(k)}+ B^{(k)})\equiv(m,X),
\end{split}
\end{align}
where $\tau_k \in (-1,1)$, 
\bqn
 \,^\rho \tilde v^{(k)}(x^{(k)},\theta^{(k)})= \gamma^{(k)} \Big ( \big (v_\rho^{(k)}(x^{(k)})+ \sum _{i\not=\rho}^{c^{(k)}} \theta_i^{(k)} v_i^{(k)}(x^{(k)})\big )  \Big / \sqrt{1 + \sum_{i\not=\rho} (\theta_i^{(k)} )^2} \Big ) \in \gamma^{(k)} (\,^\rho S_k^+)_{x^{(k)}},
\eqn
and 
$$\,^\rho S_k^+=\mklm{v \in \nu_k: v = \sum s_i v_i, s_\rho>0,  \norm{v}=1}.$$
  Note that for each $1 \leq \rho\leq c^{(k)}$, $$W_k \simeq f_k(\,^\rho S_k^+ \times (-1,1))$$ up to a set of measure zero. 
Now, for given $m \in M$, let $Z_m \subset T_mM$ be a neighborhood of zero such that $\exp_m: Z_m \longrightarrow M$ is a diffeomorphism onto its image. Then
\bqn
\qquad (\exp_m)_{\ast,v}: T_v Z_m \longrightarrow T_{\exp_mv}M, \quad v \in Z_m, 
\eqn
and $g \cdot \exp_m v= L_g (\exp_m v)=\exp_{L_g(m)}(L_g)_{\ast,m}( v)$. As a consequence, since $B^{(k)} \in \g_{x^{(k)}}$, we obtain
\begin{align*}
\widetilde{B^{(k)}}_{\exp_{x^{(k)}} \tau_k \, ^\rho \tilde v^{(k)}}&= \frac d{dt} \exp_{x^{(k)}} \big ( L_{\e{-t B^{(k)}}} \big )_{\ast, x^{(k)}}( \tau_k \, ^\rho \tilde v^{(k)})_{|t=0}= (\exp_{x^{(k)}})_{\ast, \tau_k \, ^\rho \tilde v ^{(k)}}\big (\lambda(B^{(k)}) ( \tau_k \, ^\rho \tilde v^{(k)})\big ) \\
&= \tau_k (\exp_{x^{(k)}})_{\ast, \tau_k \, ^\rho \tilde v ^{(k)}}\big (\lambda(B^{(k)})( \, ^\rho \tilde v^{(k)}) \big ),
\end{align*}
where we denoted by
\bqn
\lambda: \g_{x^{(k)}} \longrightarrow \mathfrak{gl}(\nu_{k,x^{(k)}}), \quad B^{(k)} \mapsto \frac d{dt} (L_{\e{-tB^{(k)}}})_{\ast, x^{(k)}|t=0} 
\eqn
the linear representation of $ \g_{x^{(k)}}$ in $\nu_{k,x^{(k)}}$, and made the canonical identification $T_v(\nu_{k,x^{(k)}}) \equiv \nu_{k,x^{(k)}}$ for any $v \in(\nu_k)_{x^{(k)}}$.  With $\pi(\eta)=\exp_{x^{(k)}} \tau_k \, ^\rho \tilde v^{(k)}$ we therefore obtain for the phase function the factorization 
\begin{align*}
\psi(\eta ,X)&=\eta(\tilde X_{\pi(\eta)})=\eta \big (\widetilde{(\tau_k A^{(k)}+B^{(k)})}_{\exp_{x^{(k)}} \tau_k \, ^\rho \tilde v^{(k)}}\big )\\
&= \tau_k \Big [ \eta\big( \widetilde{A^{(k)}}_{\exp_{x^{(k)}} \tau_k \, ^\rho \tilde v^{(k)}}\big ) + \eta \big ( (\exp_{x^{(k)}})_{\ast, \tau_k \, ^\rho \tilde v ^{(k)}}[\lambda(B^{(k)}) ^\rho \tilde v^{(k)}]\big ) \Big]. 
\end{align*} 
Similar considerations hold for $\zeta_k$ in the $\alpha^{(k)}$-charts $\mklm{U_k^\rho}_{c^{(k)}+1 \leq \rho \leq c^{(k)}+d^{(k)}}$, so that we get on the resolution space
\bqn 
 \psi \circ (\id_{fiber} \otimes \zeta_k) =  \,^{(k)} \tilde \psi^{tot}=\tau_k \cdot  \,  ^{(k)} \phw, 
 \eqn
$ \,^{(k)} \tilde \psi^{tot}$ and $  \,  ^{(k)} \phw $ being the \emph{total} and \emph{weak transform} of the phase function $\psi$, respectively. 
\begin{example}
In the case $M=T^\ast \rn$ and $G\subset \GL(n,\R)$ a closed subgroup, the phase function  factorizes with respect to the canonical coordinates $\eta=(q,p)$  according to
\begin{align*}
\psi(q,p,X)&=\eklm{Xq,p}= \eklm{\big (\tau_k A^{(k)}+ B^{(k)}\big )\exp_{x^{(k)}} \tau_k\, ^\rho \tilde v^{(k)},p}\\ & =\tau_k \left [ \eklm{A^{(k)} x^{(k)} + B^{(k)} \,^\rho \tilde v^{(k)}, p} +\tau_k \eklm {A^{(k)}  \,^\rho \tilde v^{(k)},p }\right ],
\end{align*} 
where we took into account that in $\rn$ the exponential map is given by  $\exp_{x^{(k)}} v^{(k)} = x^{(k)}+v^{(k)}$. 
\end{example}

 Introducing a partition $\mklm{u_k^\rho}$ of unity subordinated to the covering $\mklm{U_k^\rho}$ now yields 
\bqn 
I_k(\mu)=\sum_{\rho=1} ^{c^{(k)}}   I_k^\rho (\mu)+\sum_{\rho=c^{(k)}+1} ^{d^{(k)}}   \tilde I_k^\rho(\mu),
\eqn 
where the integrals  $ I_k^\rho (\mu)$ and $  \tilde I_k^\rho(\mu)$ are given by the expressions
\begin{gather*}
\int_{ B_{Z_k} (W_k \times \g) }u_k^\rho  (\id_{fiber} \otimes\,  \zeta_k)^\ast (  e^{i\psi /\mu }a\chi_k  dX d\eta).
\end{gather*}
As we shall see in Section \ref{sec:8}, the weak transform $\, ^{(k)} \phw$ has no critical points in the $\alpha^{(k)}$-charts, which implies that the integrals $ \tilde I_k(\mu)^\rho  $ contribute to $I(\mu)$ only with higher order terms. In what follows, we shall therefore restrict ourselves to the examination of the integrals  $  I_k^\rho(\mu)$. Setting $a_k^\rho=( u_k^\rho \circ \phi_k^\rho) [ (a \chi_k)  \circ ( \id_{fiber} \otimes \zeta_k^\rho)]$ we obtain with \eqref{eq:*} and \eqref{eq:**}
\begin{gather*}
 I_k^\rho (\mu)= \int_{ M_k(H_k) \times (-1,1)  }\Big [\int_{ \gamma^{(k)}((S_k)_{x^{(k)}}) \times  \g_{x^{(k)}} \times \g_{x^{(k)}}^\perp } \Big [ \int_{T^\ast_{\exp_{x^{(k)}} \tau_k \tilde v^{(k)}}W_k} e^{i\frac{\tau_k}\mu \,  ^{(k)} \phw} a_k^\rho \, \tilde \Phi_k^\rho \\   \d(T^\ast_{\exp_{x^{(k)}} \tau_k \tilde v^{(k)}}W_k) \Big ] dA^{(k)} \, dB^{(k)} \, d\tilde v^{(k)}   \Big ]  \d \tau_k \,     \d x^{(k)},
\end{gather*}
 where  $d\tilde v^{(k)}$ is a suitable measure on the set $\gamma^{(k)}((S_k)_{x^{(k)}})$ such that 
$$\d X \d \eta \equiv \tilde \Phi_k^\rho \, \d(T^\ast_{\exp_{x^{(k)}} \tau_k \tilde v^{(k)}}W_k) \, d A^{(k)} \, dB^{(k)} \,  d\tilde v^{(k)} \, d\tau_k \, dx^{(k)},$$ 
$\tilde \Phi_k^\rho$ being a Jacobian. Furthermore, a computation shows that  $\tilde \Phi_k^\rho  = |\tau_k|^{c^{(k)}+d^{(k)}-1} \, \Phi_k\circ   {'\zeta_k^\rho}$.

\subsection*{First reduction} Let us now assume that there exists a $m \in W_k$ with orbit type $G/H_j$, and let  $x^{(k)} \in M_k(H_k), v^{(k)} \in (\nu_k)_{x^{(k)}}$ be such that $m=f_k(x^{(k)},v^{(k)})$. Since we can assume that $m$ lies in a slice at $x^{(k)}$ around the $G$-orbit of $x^{(k)}$, we have $G_m \subset G_{x^{(k)}}$, see Kawakubo \cite[pages 184-185]{kawakubo}, and Bredon \cite[page 86]{bredon}. Hence, $H_j\simeq G_m$ must be conjugate to a subgroup of $H_k\simeq G_{x^{(k)}}$. Now, $G$ acts on $M_k$ with the isotropy types $(H_k),(H_{k+1}), \dots, (H_L)$. The isotropy types occuring in $W_k$ are therefore those for which the corresponding isotropy groups  $H_k,H_{k+1}, \dots, H_L$ are conjugate to a subgroup of $H_k$, and we shall denote them by
\bqn
(H_k) = (H_{i_1}), (H_{i_2}), \dots, (H_L).
\eqn
 Now, for every $x^{(k)}\in M_k(H_k)$, $ (\nu_k)_{x^{(k)}}$ is an orthogonal $ G_{x^{(k)}}$-space; therefore $G_{x^{(k)}}$ acts on  $(S_k)_{x^{(k)}}$ with   isotropy types $(H_{i_2}), \dots, (H_L)$, cp. Donnelly \cite[pp. 34]{donnelly78}. Furthermore,  by the invariant tubular neighborhood theorem, one has the isomorphism
\bqn
W_k/G \simeq (\nu_k)_{x^{(k)}}/ G_{x^{(k)}},
\eqn
so that $G$ acts on $S_k=\mklm{ v \in \nu_k: \norm {v}=1}$ with  isotropy types $(H_{i_2}), \dots, (H_L)$ as well. 
As will turn out, if $G$ acted on $S_k$ only with type $(H_L)$, the critical set of $^{(k)} \phw$ would be clean in the sense of Bott, and we could proceed to apply the stationary phase theorem to compute $I_k(\mu)$. But in general this will not be the case, and we are forced to continue with the iteration.

\subsection*{Second decomposition}

Let now $x^{(k)}\in M_k(H_k)$ be fixed. Since $\gamma^{(k)}: \nu_k \rightarrow \nu_k$ is an equivariant diffeomorphism onto its image,   $\gamma^{(k)}((S_k)_{x^{(k)}})$ is a compact $G_{x^{(k)}}$-manifold, and we consider the covering  
\bqn
\gamma^{(k)}((S_k)_{x^{(k)}}) =W_{ki_2} \cup \dots \cup W_{kL}, \qquad W_{ki_j}= f_{ki_j}(\stackrel \circ D_1(\nu_{ki_j})), \quad W_{kL}=\mathrm{Int} ( \gamma^{(k)}((S_k)_{x^{(k)}})_{L}),
\eqn
where $f_{ki_j}:\nu_{ki_j} \rightarrow  \gamma^{(k)}((S_k)_{x^{(k)}})_{i_j}$ is an invariant tubular neighborhood of $\gamma^{(k)}((S_k)_{x^{(k)}})_{i_j}(H_{i_j})$ in 
\bqn 
\gamma^{(k)}((S_k)_{x^{(k)}})_{i_j}=\gamma^{(k)}((S_k)_{x^{(k)}}) - \bigcup_{r=2}^{j-1} f_{ki_r}(\stackrel \circ D_{1/2}(\nu_{ki_r})),  \qquad j\geq 2, 
\eqn
and $f_{ki_j}(x^{(i_j)},v^{(i_j)})=(\exp_{x^{(i_j)}} \circ \gamma^{(i_j)})(v^{(i_j)})$, $x^{(i_j)} \in  \gamma^{(k)}((S_k)_{x^{(k)}})_{i_j} (H_{i_j})$, $ v ^{(i_j)} \in ( \nu_{ki_j}) _{x^{(i_j)}}$, $\gamma^{(i_j)}: \nu_{ki_j} \rightarrow \nu_{ki_j}$ being an equivariant diffeomorphism onto its image. 
Let further $\{\chi_{ki_j}\}$ denote a partition of  unity subordinated to the covering $\mklm{W_{ki_j}}$,  and define
\begin{align*}
I_{ki_j}^\rho(\mu) =&\int_{ M_k(H_k) \times (-1,1)  }\Big [\int_{ \gamma^{(k)}((S_k)_{x^{(k)}}) \times  \g_{x^{(k)}} \times \g_{x^{(k)}}^\perp} \Big [ \int_{ T^\ast_{\exp_{x^{(k)}} \tau_k \tilde v^{(k)}}W_k} e^{i\frac{\tau_k}\mu \,  ^{(k)} \phw}a_k^\rho \\ & \chi_{ki_j}
 \tilde \Phi_k^\rho \, \d(T^\ast_{\exp_{x^{(k)}} \tau_k \tilde v^{(k)}}W_k)\Big ] \, dA^{(k)} \, dB^{(k)} \, d\tilde v^{(k)}   \Big ]  \d \tau_k \,   \d x^{(k)},
\end{align*}
so that $I_k^\rho(\mu)= I_{ki_2}^\rho(\mu) + \dots + I_{kL}^\rho(\mu)$. It is important to note that the partition functions $\chi_{ki_j}$ depend smoothly on $x^{(k)}$ as a consequence of the tubular neighborhood theorem, by which in particular $\gamma^{(k)}(S_k) /G \simeq \gamma^{(k)}((S_k)_{x^{(k)}})/G_{x^{(k)}}$, and the smooth dependence in $x^{(k)}$ of the induced Riemannian metric on $\gamma^{(k)}((S_k)_{x^{(k)}})$, and  the metrics on the normal bundles $\nu_{ki_j}$.
Since $G_{x^{(k)}}$ acts on $W_{kL}$ only with type $(H_L)$, the iteration process for $I_{kL}^\rho(\mu)$ ends here. For the remaining integrals $I_{ki_j}^\rho(\mu)$ with $k < i_j < L$, let us denote by
\bqn
\frak{iso} \,\gamma^{(k)}((S_k)_{x^{(k)}})_{i_j}(H_{i_j}) \rightarrow \gamma^{(k)}((S_k)_{x^{(k)}})_{i_j}(H_{i_j})
\eqn
the isotropy algebra bundle over $\gamma^{(k)}((S_k)_{x^{(k)}})_{i_j}(H_{i_j})$, and by $\pi_{ki_j}: W_{ki_j} \rightarrow \gamma^{(k)}((S_k)_{x^{(k)}})_{i_j}(H_{i_j})$ the canonical projection.  For $x^{(i_j)} \in \gamma^{(k)}((S_k)_{x^{(k)}})_{i_j}(H_{i_j})$, consider the decomposition
\bqn
\g = \g_{x^{(k)}} \oplus \g_{x^{(k)}}^\perp =(\g_{x^{(i_j)}}\oplus \g_{x^{(i_j)}}^\perp) \oplus \g_{x^{(k)}}^\perp.
\eqn
Let further $A_1^{(i_j)}, \dots ,A_{d^{(i_j)}}^{(i_j)}$ be an orthonormal basis in $ \g_{x^{(i_j)}}^\perp$,  as well as $B_1^{(i_j)}, \dots ,B_{e^{(i_j)}}^{(i_j)}$ be an orthonormal basis in $ \g_{x^{(i_j)}}$, and $\{v_1^{(ki_j)}, \dots, v_{c^{(i_j)}}^{(ki_j)}\}$ an orthonormal frame in $\nu_{ki_j}$. Integrating along the fibers in a neighborhood of $\pi_{ki_j}^\ast \frak{iso} \, \gamma^{(k)}((S_k)_{x^{(k)}})_{i_j}(H_{i_j}) \subset W_{ki_j} \times \g_{x^{(k)}}$ then yields  for $ I_{ki_j}^\rho(\mu)$ the expression
\begin{align*}
 \int_{M_k(H_k)\times (-1,1)} &\Big [ \int_{\gamma^{(k)}((S_k)_{x^{(k)}})_{i_j}(H_{i_j})} \Big [ \int_{\pi_{ki_j}^{-1} (x^{(i_j)})\times \g_{x^{(k)}}\times \g_{x^{(k)}}^\perp} \Big [ \int_{ T^\ast_{\exp _{x^{(k)}} \tau_k \exp_{x^{(i_j)}} v ^{( i_j)}} W_k } e^{i\frac{\tau_k}\mu \, ^{(k)} \phw }   \\\times    a_k^\rho \chi_{ki_j} \,  & \Phi_{ki_j}^\rho \, \d(T^\ast_{\exp _{x^{(k)}} \tau_k \exp_{x^{(i_j)}} v ^{( i_j)}} W_k)  \Big ] d A^{(k)} \, d A^{(i_j)} \, dB^{(i_j)} \,  dv^{(i_j)}     \big ] dx^{(i_j)}  \Big ]   d\tau_k dx^{(k)},
\end{align*}
where $\Phi_{ki_j}^\rho$ is a Jacobian, and 
\bqn
\gamma^{(i_j)} \big (  \stackrel \circ D_1(\nu_{ki_j})_{x^{(i_j)}}\big )  \times \g_{x^{(i_j)}}^\perp \times   \g_{x^{(i_j)}} \ni (v^{(i_j)}, A^{(i_j)},B^{(i_j)})\mapsto (\exp_{x^{(i_j)}} v^{(i_j)},A^{(i_j)} + B^{(i_j)})=(\tilde v^{(k)},B^{(k)})
\eqn
are coordinates on $ \pi_{ki_j}^{-1}(x^{(i_j)})\times \g_{x^{(k)}}$, while $dx^{(i_j)}$, and $dA^{(i_j)}, dB^{(i_j)} ,  dv^{(i_j)} $ are suitable measures in the spaces   $\gamma^{(k)}((S_k)_{x^{(k)}})_{i_j}(H_{i_j})$, and  $\g_{x^{(i_j)}}^\perp$, $\g_{x^{(i_j)}}$, $\gamma^{(i_j)} \big (\stackrel \circ D_1(\nu_{ki_j})_{x^{(i_j)}}\big )$, respectively, such that we have the equality $ \tilde \Phi_k^\rho \d B^{(k)} \d \tilde v^{(k)}\equiv\Phi_{ki_j}^\rho \,  dA^{(i_j)} \, dB^{(i_j)} \,  dv^{(i_j)} \, dx^{(i_j)}$.

\subsection*{Second monoidal transformation}

Let us fix an $l$ such that $k< l < L$, $(H_l) \leq (H_k)$, and consider in $B_{Z_k}(W_k\times \g)$ a monoidal transformation 
\bqn 
\zeta_{kl}: B_{Z_{kl}}(B_{Z_k}(W_k\times \g)) \longrightarrow B_{Z_k}(W_k\times \g)
\eqn
with center
\bqn
Z_{kl}\simeq  \bigcup _{x^{(k)} \in M_k(H_k)}  (-1,1)\times  \frak{iso} \,  \gamma^{(k)}( (S_k)_{x^{(k)}})_l(H_l).
\eqn
Let $A^{(l)}\in \g_{x^{(l)}}^\perp$ and $B^{(l)}\in \g_{x^{(l)}}$ be arbitrary  and write  $ A^{(l)}(x^{(k)},x^{(l)},\alpha^{(l)})$ $=\sum  \alpha_i^{(l)} A_i^{(l)}(x^{(k)}, x^{(l)})\in \g_{x^{(l)}}^\perp$,  $ B^{(l)}(x^{(k)},x^{(l)},\beta^{(l)})=\sum  \beta_i^{(l)} B_i^{(l)}(x^{(l)})\in \g_{x^{(l)}}$, as well as
\bqn
 \gamma^{(l)}(v^{(l)})(x^{(k)}, x^{(l)},\theta^{(l)})= \sum _{i=1}^{c^{(l)}} \theta_i^{(l)} v_i^{(kl)}(x^{(k)},x^{(l)}).
\eqn
Then  $Z_{kl}\simeq\mklm{\alpha^{(k)}=0, \, \alpha^{(l)}=0, \, \theta^{(l)}=0}$ locally, which in particular shows that $Z_{kl}$ is a manifold. If we now cover $B_{Z_{kl}}(B_{Z_k}(W_k\times \g))$ with the standard charts, we shall see again in Section \ref{sec:8} that   modulo higher order terms the main contributions to $I^\rho_{kl}(\mu)$ come from
 the $(\theta^{(k)},\theta^{(l)})$-charts. Therefore it suffices to examine $\zeta_{kl}$ in one of these charts, in which it reads
\begin{gather*}
\zeta_{kl}^{\rho\sigma}: (x^{(k)},\tau_k, x^{(l)}, \tau_l, \tilde v^{(l)}, A^{(k)},  A^{(l)},  B^{(l)}) \stackrel{'\zeta_{kl}^{\rho\sigma}}{\longmapsto} (x^{(k)},\tau_k, x^{(l)},   \tau_l \tilde v^{(l)},  \tau_l A^{(k)},   \tau_l  A^{(l)},  B^{(l)})  \\\longmapsto (x^{(k)},\tau_k, \exp_{x^{(l)}} \tau_l \tilde v^{(l)}, \tau_l A^{(k)}, \tau_l  A^{(l)}+ B^{(l)})\equiv(x^{(k)},\tau_k, \tilde v^{(k)},A^{(k)}, B^{(k)}),
\end{gather*}
where
\bqn
\tilde v^{(l)}(x^{(k)},x^{(l)},\theta^{(l)})\in \gamma^{(l)}\big ((S_{kl}^+)_{x^{(l)}}\big ). 
\eqn
Note that $Z_{kl}$ has normal crossings with the exceptional divisor $E_k=\zeta_k^{-1}(Z_k) = \mklm{\tau_k=0}$, and that 
$$W_{kl} \simeq f_{kl} (S_{kl}^+ \times (-1,1))$$ up to a set of measure zero, where $S_{kl}$ denotes
 the sphere subbundle in $\nu_{kl}$,  and we set $S_{kl}^+=\mklm { v \in S_{kl}:  v=\sum v_i v_i^{(kl)}, \, v_\sigma>0 }$ for some $\sigma$. Consequently,  the phase function factorizes according to 
\bqn 
\psi \circ (\id_{fiber} \otimes (\zeta_k^\rho \circ \zeta_{kl}^{\rho\sigma}))= \,^{(kl)} \tilde \psi^{tot}=\tau_k \, \tau_l \cdot \,  ^{(kl)} \phw,
\eqn
which in the given charts reads
\begin{align*}
\psi(\eta,X)&=\tau_k \left [ \eta \Big (\widetilde{\tau_l A^{(k)}}_{\exp_{x^{(k)}} \tau_k   \exp_{x^{(l)}} \tau_l \tilde v^{(l)}}\Big )\right. \\
&\left. +  \eta \Big ( (\exp_{x^{(k)}})_{\ast, \tau_k  \exp_{x^{(l)}} \tau_l \tilde v^{(l)}}[\lambda(\tau_l A^{(l)}+B^{(l)})  \exp_{x^{(l)}}\tau_l \tilde v^{(l)} ] \Big ) \right] \\
&=\tau_k  \tau_l \left [ \eta \Big (\widetilde{ A^{(k)}}_{\exp_{x^{(k)}} \tau_k   \exp_{x^{(l)}} \tau_l \tilde v^{(l)}}\Big )+  \eta \Big ( (\exp_{x^{(k)}})_{\ast, \tau_k  \exp_{x^{(l)}} \tau_l \tilde v^{(l)}}[\lambda( A^{(l)})  \exp_{x^{(l)}}\tau_l \tilde v^{(l)} ] \Big ) \right. \\
&\left. +  \eta \Big ( (\exp_{x^{(k)}})_{\ast, \tau_k  \exp_{x^{(l)}} \tau_l \tilde v^{(l)}}\big [ (\exp_{x^{(l)}})_{\ast,\tau_l \tilde v^{(l)}} [ ( \lambda(B^{(l)})  \tilde v ^{(l)}]\big ] \Big ) \right] \\
\end{align*}
where we took into account that
$$\lambda( B^{(l)})  \exp_{x^{(l)}}\tau_l \tilde v^{(l)} = \frac d {dt}  \exp_{x^{(l)}}\big (L_{\e{-t B^{(l)}}} \big )_{\ast,x^{(k)}} \tau_l  \tilde v^{(l)} _{|t=0}=(\exp_{x^{(l)}})_{\ast,\tau_l \tilde v^{(l)}} \big ( \lambda(B^{(l)})  \tau_l \tilde v ^{(l)}\big ).
$$
Since the weak transforms $^{kl} \tilde \psi^{wk}$ have no critical points in the $(\theta^{(k)},\alpha^{(l)})$-charts, modulo lower order terms,  $ I_{kl}^\rho(\mu)$ is given by a sum of integrals of the form  
\begin{align*}
I_{kl}^{\rho\sigma}(\mu)
&=\int_{M_k(H_k)\times (-1,1)} \Big [ \int_{\gamma^{(k)}((S_k)_{x^{(k)}})_l(H_{l})\times (-1,1)} \Big [ \int_{\gamma^{(l)}((S_{kl})_{x^{(l)}}) \times \g_{x^{(l)}} \times \g_{x^{(l)}}^\perp \times \g_{x^{(k)}}^\perp} \Big [ \int_{T^\ast _{m^{(kl)}}W_k }  \\ &\times  e^{i\frac{\tau_k\tau_l }\mu  \, ^{(kl)} \phw}  a_{kl}^{\rho\sigma} \,  \tilde  \Phi_{kl}^{\rho\sigma} \, \d(T^\ast _{m^{(kl)}}W_k) \Big ] d A^{(k)} \, d A^{(l)} \, dB^{(l)} \,  d\tilde v^{(l)}   \Big ]  d\tau_l \,dx^{(l)}  \Big ] \, d\tau_k \,   dx^{(k)},
\end{align*}
where we wrote $m^{(kl)}=\exp_{x^{(k)}} \tau_k \exp_{x^{(l)}} \tau_l \tilde v^{(l)}$, $a^{\rho\sigma}_{kl}$ are smooth amplitudes with compact support in a $(\theta^{(k)},\theta^{(l)})$-chart labeled by the indices $\rho, \sigma$, and $d\tilde v^{(l)}$ is a suitable measure in $\gamma^{(l)}((S_{kl})_{x^{(l)}})$ such that we have the equality 
$$\d X \d \eta \equiv \tilde \Phi_{kl}^{\rho\sigma} \, \d(T^\ast _{m^{(kl)}}W_k)\, d A^{(k)} \, dA^{(l)} \, dB^{(l)} \,  d\tilde v^{(l)}  \, d\tau_l \, dx^{(l)}\, d\tau_k \, dx^{(k)}.$$
Furthermore, $\tilde \Phi_{kl}^{\rho\sigma} = |\tau_l | ^{c^{(l)} +d^{(k)} +d^{(l)} -1} \Phi_{kl}^\rho \circ {'\zeta_{kl}^{\rho\sigma}}$.

\subsection*{Second reduction} Now, the group $G_{x^{(k)}}$ acts on $\gamma^{(k)}((S_k)_{x^{(k)}})_l$ with the isotropy types $(H_l)=(H_{i_j}),(H_{i_{j+1}}), \dots, (H_L)$. By the same arguments given in the first reduction, the isotropy types occuring in $W_{kl}$ constitute a subset of these types, and we shall denote them by
\bqn
(H_l) = (H_{i_{r_1}}), (H_{i_{r_2}}), \dots, (H_L).
\eqn
Consequently, $G_{x^{(k)}}$ acts on $S_{kl}$ with the isotropy types $(H_{i_{r_2}}), \dots, (H_L)$. Again, if $G$ acted on $S_{kl}$ only with type $(H_L)$, we shall see later that the critical set of $^{(kl)} \phw$ would be clean. However, in general this will not be the case, and we have to continue with the iteration.

\subsection*{N-th decomposition}
Denote by $\Lambda \leq L$ the maximal number of elements that a totally ordered subset of the set of isotropy types can have. Assume that $3 \leq N < \Lambda$, and let   $\mklm{(H_{i_1}), \dots , (H_{i_{N}})}$ be a totally ordered subset of the set of isotropy types with $i_1<\dots < i_N<L$. Let   $f_{i_1}$, $f_{i_1i_2}$, $S_{i_1}$, $S_{i_1i_2}$, as well as  $x^{(i_1)}\in M_{i_1}(H_{i_1}), \quad x^{(i_2)}\in \gamma^{(i_{1})}\big ((S_{i_1}^+)_{x^{(i_{1})}}\big ) _{i_2}(H_{i_2})$ 
be defined as in the first two iteration steps. Let now $j<N$, and assume that $f_{i_1\dots i_{j}}$,  $S_{i_1\dots i_{j}}$,... have already been defined. For each $x^{(i_{N-1})}$, let $\gamma^{(i_{N-1})}((S_{i_1\dots i_{N-1}})_{x^{(i_{N-1})}})_{i_{N}}$ be the submanifold with corners of the $G_{x^{(i_{N-1})}}$-manifold $\gamma^{(i_{N-1})}((S_{i_1\dots i_{N-1}})_{x^{(i_{N-1})}})$ from which all the isotropy types  less than $(H_{i_{N}})$ have been removed. Consider  the invariant tubular neighborhood 
$$
f_{i_1\dots i_{N}}=\exp \circ \gamma^{(i_{N})}: \nu _{i_1\dots i_{N}} \rightarrow\gamma^{(i_{N-1})}((S_{i_1\dots i_{N-1}})_{x^{(i_{N-1})}})_{i_{N}}
$$
 of the set of maximal singular orbits $\gamma^{(i_{N-1})}((S_{i_1\dots i_{N-1}})_{x^{(i_{N-1})}})_{i_{N}}(H_{i_{N}})$, and define $S_{i_1\dots i_{N}}$  as the sphere subbundle in $ \nu _{i_1\dots i_{N}}$ over $\gamma^{(i_{N-1})}((S_{i_1\dots i_{N-1}})_{x^{(i_{N-1})}})_{i_N} (H_{i_N})$. Put further $W_{i_1\dots i_{N}}=f_{i_1\dots i_{N}}(\stackrel{\circ}{D_1}(\nu_{i_1\dots i_{N}}))$ and denote the corresponding integral in the decomposition of $I_{i_1\dots i_{N-1}}^{\rho_{i_1} \dots \rho_{i_{N-1}}}(\mu)$ by $I_{i_1\dots i_{N}}^{\rho_{i_1} \dots \rho_{i_{N-1}}}(\mu)$. For a point $x^{(i_{N})}\in \gamma^{(i_{N-1})}((S_{i_1\dots i_{N-1}})_{x^{(i_{N-1})}})_{i_{N}}(H_{i_{N}})$ we then consider the decomposition
\bqn 
 \g_{x^{(i_{N-1})}}= \g_{x^{(i_N)}}\oplus \g_{x^{(i_N)}}^\perp, 
\eqn
and set $d^{(i_N)}=\dim \g_{x^(i_N)}^\perp$, $ e^{(i_N)}=\dim \g_{x^(i_N)}$, yielding  the decomposition
\begin{gather}
\label{eq:decomp}
\g = \g_{x^{(i_1)}} \oplus \g_{x^{(i_1)}}^\perp =(\g_{x^{(i_2)}}\oplus \g_{x^{(i_2)}}^\perp) \oplus \g_{x^{(i_1)}}^\perp =\dots = \g_{x^{(i_N)}}\oplus \g_{x^{(i_N)}}^\perp \oplus \cdots \oplus \g_{x^{(i_1)}}^\perp. 
\end{gather}
Denote by  $\{ A_r^{(i_N)}(x^{(i_1)},\dots,x^{(i_N)})\}$ a basis of $\g_{x^(i_N)}^\perp$, and by $\{ B_r^{(i_N)}(x^{(i_1)},\dots,x^{(i_N)})\}$ a basis of $\g_{x^(i_N)}$. For arbitrary elements  $A^{(i_N)}\in \g_{x^{(i_N)}}^\perp$ and $B^{(i_N)}\in \g_{x^{(i_N)}}$   write  
\begin{align*}
A^{(i_N)} &=\sum_{r=1}^{d^{(i_N)}} \alpha^{(i_N)}_r A_r^{(i_N)}(x^{(i_1)},\dots,x^{(i_N)}), \qquad B^{(i_N)} =\sum_{r=1}^{e^{(i_N)}} \beta^{(i_N)}_r B_r^{(i_N)}(x^{(i_1)},\dots,x^{(i_N)}),
\end{align*}
and put
\bqn
\tilde v^{(i_N)}(x^{(i_N)},\theta^{(i_N)})= \gamma^{(i_N)} \left ( \Big (v_\rho^{(i_1\dots i_N)}(x^{(i_N)})+ \sum _{r\not=\rho}^{c^{(i_N)}} \theta_r^{(i_N)} v_r^{(i_1\dots i_N)}(x^{(i_N)})\Big )\Big / \sqrt{1 + \sum\limits_{r\not=\rho} (\theta_r^{(i_N)})^2} \right ) 
\eqn
for some $\rho$, where $\mklm{v_r^{(i_1\dots i_N)}}$ is an orthonormal frame in $\nu_{i_1\dots  i_N}$. 
Finally, we shall use the notations
\begin{align*}
m^{(i_j\dots i_{N})}&=\exp_{x^{(i_j)}}[\tau_{i_j} \exp_{x^(i_{j+1})}[\tau_{i_{j+1}}\exp_{x^(i_{j+2})}[\dots [ \tau_{i_{N-2}}\exp_{x^(i_{N-1})}[\tau_{i_{N-1}}\exp_{x^{(i_N)}}[ \tau_{i_N} \tilde v ^{(i_N)}]]] \dots ]]], \\
X^{(i_j\dots i_{N})}&={\tau_{i_j} \cdots \tau_{i_N}A^{(i_j)}}+{\tau_{i_{j+1}} \cdots \tau_{i_N}A^{(i_{j+1})}}+\dots +{\tau_{i_{N-1}}  \tau_{i_N}A^{(i_{N-1})}} +{\tau_{i_N} A^{(i_N)}}  +B^{(i_N)},
\end{align*}
where $j=1,\dots ,N$. 

 \subsection*{N-th monoidal transformation} 
 Let the monoidal transformations $\zeta_{i_1}$ and $\zeta_{i_1i_2}$ be defined as in the first two iteration steps, and  assume that $\zeta_{i_1\dots i_j}$ have already been defined for $j<N$. Consider the  monoidal transformation 
 \bqn 
 \zeta_{i_1\dots i_N}: B_{Z_{i_1\dots i_N}}(B_{Z_{i_1\dots i_{N-1}}}( \dots B_{Z_{i_1}}(W_k \times \g) \dots )) \longrightarrow B_{Z_{i_1\dots i_{N-1}}}( \dots B_{Z_{i_1}} (W_k \times \g) \dots )
 \eqn
  with center
 \begin{gather*}
 Z_{i_1\dots i_N}\simeq \bigcup _{x^{(i_1)}, \dots, x^{(i_{N-1})}} (-1,1)^{N-1} \times \mathfrak{iso} \gamma^{(i_{N-1})}((S_{i_1\dots i_{N-1}})_{x^{(i_{N-1})}})_{i_N}(H_{i_N}).
 \end{gather*}
 Denote by $\zeta_{i_1}^{\rho_{i_1}} \circ \dots \circ\zeta_{i_1\dots i_N}^{\rho_{i_1}\dots \rho_{i_N}}$ a local realization of the sequence of monoidal transformations $\zeta_{i_1}\circ \dots \circ \zeta_{i_1\dots i_N}$ in a set of $(\theta^{(i_1)}, \dots ,\theta^{(i_N)})$-charts labeled by the indices $\rho_{i_1}, \dots, \rho_{i_N}$.  Now, for an arbitrary element $B ^{(i_1)}\in \g_{i_1}$ one  computes
 \begin{align}
 \label{eq:hola}
 \begin{split}
 (\tilde B ^{i_1)})_{ m^{(i_1\dots i_{N})}}&=\frac d{dt} \e{-t B^{(i_1)}} \cdot  m^{(i_1\dots i_{N})}_{|t=0}= \frac d {dt} \exp_{x^{(i_1)}} \big [(\e{-t B^{(i_1)}})_{\ast, x^{(i_1)}} [\tau_{i_1}  m^{(i_2\dots i_{N})}]\big ]_{|t=0}\\
 &=(\exp_{x^{(i_1)}})_{\ast, \tau_{i_1}  m^{(i_2\dots i_{N})}}[ \lambda(B^{(i_1)})\tau_{i_1}  m^{(i_2\dots i_{N})}].
 \end{split}
 \end{align}
 By iteration we obtain for arbitrary $A ^{(i_j)}\in \g_{i_j}^\perp$, $2
 \leq j\leq N$, 
  \begin{align}
  \begin{split}
  \label{eq:tilde}
 (\tilde A ^{i_j)}&)_{ m^{(i_1\dots i_{N})}}= \frac d {dt} \exp_{x^{(i_1)}} \big [\tau_{i_1} \exp_{x^{(i_2)}} [\dots [\tau_{i_{j-1}}  (\e{-t A^{(i_j)}})_{\ast, x^{(i_1)}} m^{(i_j\dots i_{N})}]\dots ]\big ]_{|t=0}\\
 &=(\exp_{x^{(i_1)}})_{\ast, \tau_{i_1}  m^{(i_2\dots i_{N})}} \big [ \tau_{i_1} (\exp_{x^{(i_2)}})_{\ast, \tau_{i_2} m^{(i_3\dots i_N)}}[ \dots [\tau_{i_{j-1}}  \lambda(A^{(i_j)})  m^{(i_j\dots i_{N})}]\dots ]\big ],
 \end{split}
 \end{align}
 and similarly
   \begin{align}
  \begin{split}
  \label{eq:tildeB}
 (\tilde B ^{i_N)}&)_{ m^{(i_1\dots i_{N})}}=(\exp_{x^{(i_1)}})_{\ast, \tau_{i_1}  m^{(i_2\dots i_{N})}} \big [ \tau_{i_1} (\exp_{x^{(i_2)}})_{\ast, \tau_{i_2} m^{(i_3\dots i_N)}}[ \dots [\tau_{i_{N}}  \lambda(B^{(i_N)})  \tilde v^{(i_N)}]\dots ]\big ].
 \end{split}
 \end{align}
 As a consequence, the phase function factorizes locally according to 
\bqn
\, ^{(i_1\dots i_N)} \tilde \psi^{tot}=\psi \circ (\id_{fiber}\otimes ( \zeta_{i_1}^{\rho_{i_1}} \circ \dots \circ\zeta_{i_1\dots i_N}^{\rho_{i_1}\dots \rho_{i_N}}))= \J( \eta_{m^{(i_1 \dots i_N)}})( X^{(i_1\dots i_N)})=\tau_{i_1} \cdots \tau_{i_N} \, ^{(i_1\dots i_N)}\tilde \psi^ {wk},
\eqn
where in the given charts $\, ^{(i_1\dots i_N)}\tilde \psi^ {wk}$ is given by
\begin{align}
\label{eq:phasewk}
\begin{split}
& \eta_{m^{(i_1 \dots i_N)}} \Big (\widetilde{ A^{(i_1)}}_{m^{(i_1 \dots i_N)}}\Big ) + \sum_{j=2}^N  \eta_{m^{(i_1 \dots i_N)}} \Big ( (\exp_{x^{(i_1)}})_{\ast, \tau_{i_1}  m^{(i_2\dots i_N)} } \\
 & \big [ (\exp_{x^{(i_2)}})_{\ast,\tau_{i_2}  m^{(i_3\dots i_N)}} \big [ \dots (\exp_{x^{(i_{j-1})}})_{\ast, \tau_{i_{j-1}}m^{(i_j\dots i_N)}}[\lambda(A^{(i_j)})  m ^{(i_j\dots i_N)}] \dots \big ]
\big ] \Big ) \\
& +  \eta_{m^{(i_1 \dots i_N)}} \Big ( (\exp_{x^{(i_1)}})_{\ast, \tau_{i_1}  m^{(i_2\dots i_N)} }\big [ (\exp_{x^{(i_2)}})_{\ast,\tau_{i_2}  m^{(i_3\dots i_N)}} \big [ \dots \\
&(\exp_{x^{(i_{N})}})_{\ast, \tau_{i_{N}}\tilde v^{(i_N)}} [\lambda(B^{(i_N)})  \tilde v ^{(i_N)}] \dots \big ]
\big ] \Big ).
\end{split}
\end{align}
Modulo lower order terms, $I(\mu)$ is then given by a sum of integrals of the form
\begin{gather}
\label{eq:N}
\begin{split}
I_{i_1\dots i_{N}}^{\rho_{i_1} \dots \rho_{i_{N}}}(\mu)\qquad \qquad \qquad \qquad\qquad \qquad\qquad \qquad\qquad \qquad \\
 =\int_{M_{i_1}(H_{i_1})\times (-1,1)} \Big [ \int_{\gamma^{(i_1)}((S_{i_1})_{x^{(i_1)}})_{i_2}(H_{i_2})\times (-1,1)} \dots \Big [  \int_{\gamma^{(i_{N-1})}((S_{i_1\dots i_{N-1}})_{x^{(i_{N-1})}})_{i_{N}}(H_{i_{N}})\times (-1,1)} \\
\Big [ \int_{\gamma^{(i_{N})}((S_{i_1\dots i_{N}})_{x^{(i_N)}})\times \g_{x^{(i_{N})}}\times \g_{x^{(i_{N})}}^\perp \times \cdots \times \g_{x^{(i_{1})}}^\perp\times T^\ast _{m^{(i_1\dots i_N)}}W_{i_1}}  e^{i\frac {\tau_1 \dots \tau_N}\mu \, ^{(i_1\dots i_N)} \tilde \psi ^{wk}}  \, a_{i_1\dots i_N}^{\rho_{i_1} \dots \rho_{i_{N}}} \,   \tilde \Phi_{i_1\dots i_N}^{\rho_{i_1} \dots \rho_{i_{N}}} \\
  \d(T^\ast _{m^{(i_1\dots i_N)}}W_{i_1})  \d A^{(i_1)} \dots  \d A^{(i_N)}  \d B^{(i_N)} \d \tilde v^{(i_N)} \Big ]  \d \tau_{i_N} \d x^{(i_{N})} \dots  \Big ] \d \tau_{i_2} \d x^{(i_{2})} \Big ]\d \tau_{i_1} \d x^{(i_{1})}.
\end{split}
\end{gather}
Here
$a_{i_1\dots i_N}^{\rho_{i_1} \dots \rho_{i_{N}}}$ are amplitudes with compact support in a system of  $(\theta^{(i_1)}, \dots, \theta^{(i_N)})$-charts labelled by the indices $\rho_{i_1} \dots \rho_{i_{N}}$, while 
\begin{align*}
\tilde \Phi_{i_1\dots i_N}^{\rho_{i_1} \dots \rho_{i_{N}}} &=\prod_{j=1}^N |\tau_{i_j}|^{c^{(i_j)}+\sum_{r=1}^j d^{(i_r)}-1}\Phi_{i_1\dots i_N}^{\rho_{i_1} \dots \rho_{i_{N}}},
\end{align*}
where $\Phi_{i_1\dots i_N}^{\rho_{i_1} \dots \rho_{i_{N}}} $ are smooth functions which do not depend on the variables $\tau_{i_j}$.

\subsection*{N-th reduction} For each $x^{(i_{N-1})}$, the isotropy group $G_{x^{(i_{N-1})}}$ acts on $ \gamma^{(i_{N-1})} ((S_{i_1\dots i_{N-1}})_{x^{(i_{N-1})}})_{i_{N}}$ by the types $(H_{i_N}), \dots, (H_L)$. The types occuring in $W_{i_1 \dots i_N}$ constitute a subset of these, and   $G_{x^{(i_{N-1})}}$ acts on  the sphere bundle $S_{i_1\dots i_N}$ over the submanifold $\gamma^{(i_{N-1})}((S_{i_1\dots i_{N-1}})_{x^{(i_{N-1})}})_{i_N} (H_{i_N})\subset W_{i_1 \dots i_N}$  with one type less.

 \medskip

\subsection*{End of iteration}

As before, let $\Lambda\leq L$ be the maximal number of elements of a totally ordered subset of the set of isotropy types. After  maximally $N=\Lambda-1$ steps, the end of the iteration is reached.

\section{Phase analysis of the weak transforms. Smoothness of the critical sets}

We shall now prove the smoothness of the critical sets of the weak transforms. We continue with  the notation of the previous sections, and consider a sequence of local monoidal transformations 
 $ \zeta^{\rho_{i_1}}_{i_1} \circ \dots \circ  \zeta^{\rho_{i_1}\dots \rho_{i_{N}}}_{i_1\dots i_{N}}$ corresponding to a totally ordered subset  $\mklm{(H_{i_1}), \dots, (H_{i_{N}})}$ of non-principal isotropy types  that are maximal in the sense that, if there is an isotropy type $(H_{i_{N+1}})$ with $i_N < i_{N+1}$ such that $\mklm{(H_{i_1}), \dots , (H_{i_{N+1}}) }$ is a totally ordered subset, then  $(H_{i_{N+1}})= (H_L)$. For later purposes, let us  define certain geometric distributions $E^{(i_j)}$ and $F^{(i_N)}$ on $M$ by setting
\begin{equation}
\label{eq:EF}
\begin{split}
E^{(i_1)}_{m^{(i_1 \dots i_N)}}&=\mathrm{Span} \{  \tilde Y _{m^{(i_1 \dots i_N)}}: Y \in \g_{x^{(i_1)}} ^\perp\}, \\ 
E^{(i_j)} _{m^{(i_1 \dots i_N)}}&= (\exp_{x^{(i_1)}})_{\ast, \tau_{i_1}  m^{(i_2\dots i_N)} } \dots  (\exp_{x^{(i_{j-1})}})_{\ast, \tau_{i_{j-1}}m^{(i_j\dots i_N)}}[\lambda(\g_{x^{(i_j)}}^\perp)  m ^{(i_j\dots i_N)}]
, \\
F^{(i_N)}_{m^{(i_1 \dots i_N)}} &=  (\exp_{x^{(i_1)}})_{\ast, \tau_{i_1}  m^{(i_2\dots i_N)} } \dots (\exp_{x^{(i_{N})}})_{\ast, \tau_{i_{N}}\tilde v^{(i_N)}}[\lambda(\g_{x^{(i_N)}})  \tilde v ^{(i_N)}],
\end{split}
\end{equation}
where $2 \leq j \leq N$. Note that by \eqref{eq:decomp}, \eqref{eq:tilde} and \eqref{eq:tildeB} we have
\begin{align}
\label{eq:tangent}
T_{m^{(i_1 \dots i_N)}} ( G\cdot m^{(i_1 \dots i_N)})=E^{(i_1)}_{m^{(i_1 \dots i_N)}}\oplus  \bigoplus _{j=2}^N \tau_{i_1}\dots \tau_{i_{j-1}} E^{(i_j)} _{m^{(i_1 \dots i_N)}}  \oplus  \tau_{i_1}\dots \tau_{i_N} F^{(i_N)}_{m^{(i_1 \dots i_N)}} .
\end{align}
By construction, for $\tau_{i_j}\not=0$, $1\leq j\leq N$, the $G$-orbit through $m^{(i_1\dots i_N)}$ is of principal type $G/H_L$, which amounts to the fact that $G_{x^{(i_{N-1})}}$ acts on $S_{i_1\dots i_N}$ only with the isotropy type $(H_L)$, where we understand that $G_{x^{(i_{0})}}=G$.  We then have the following

\begin{theorem}
\label{thm:I}
Let  $\mklm{(H_{i_1}), \dots, (H_{i_{N}})}$ be a maximal, totally ordered subset of non-principal isotropy types, and 
  $ \zeta^{\rho_{i_1}}_{i_1} \circ \dots \circ  \zeta^{\rho_{i_1}\dots \rho_{i_N}}_{i_1\dots i_N}$  a corresponding sequence of local monoidal transformations  in a set of $(\theta^{(i_1)}, \dots, \theta^{(i_N)})$-charts labeled by the indices $\rho_{i_1},\dots ,\rho_{i_N}$.  Let $\eta_{m^{(i_1 \dots i_N)}}  \in \pi^{-1} (m^{(i_1\dots i_N)})$, and consider the factorization 
\bqn
\J( \eta_{m^{(i_1 \dots i_N)}})( X^{(i_1\dots i_N)})= \, ^{(i_1\dots i_N)} \tilde \psi^{tot}=
\tau_{i_1} \cdots \tau_{i_N} \, ^{(i_1\dots i_N)}\tilde \psi^ {wk, \, pre}
\eqn
of the phase function $\psi$ after $N$ iteration steps, where $\, ^{(i_1\dots i_N)}\tilde \psi^ {wk,pre}$ is given by \eqref{eq:phasewk}.\footnote{Note that $\, ^{(i_1\dots i_N)}\tilde \psi^ {wk,pre}$ was denoted in \eqref{eq:phasewk} by $\, ^{(i_1\dots i_N)}\tilde \psi^ {wk}$.} Let further 
\bqn 
 \, ^{(i_1\dots i_N)}\tilde \psi^ {wk}
\eqn
denote the pullback of  $ \, ^{(i_1\dots i_N)}\tilde \psi^ {wk,\,pre}$ along the  substitution $\tau=\delta_{i_1\dots i_N}(\sigma)$ given by the sequence of monoidal transformations
\begin{align*}
\delta_{i_1\dots i_N}: (\sigma_{i_1}, \dots \sigma_{i_N}) &\mapsto \sigma_{i_1}( 1, \sigma_{i_2}, \dots, \sigma_{i_N})= (\sigma_{i_1}', \dots ,\sigma_{i_N}')\mapsto \sigma_{i_2}'(\sigma_{i_1}',1,\dots, \sigma_{i_N}')= (\sigma_{i_1}'', \dots, \sigma_{i_N}'')\\
 &\mapsto \sigma_{i_3}''(\sigma_{i_1}'',\sigma_{i_2}'', 1,\dots, \sigma_{i_N}'')= \cdots \mapsto \dots = (\tau_{i_1}, \dots ,\tau_{i_N}).
\end{align*}
Then the critical set $\Crit(\, ^{(i_1\dots i_N)} \phw)$ of $\, ^{(i_1\dots i_N)} \phw$ is given by all points
$$(\sigma_{i_1}, \dots, \sigma_{i_N}, x^{(i_1)}, \dots, x^{(i_N)},  \tilde v ^{(i_N)}, A^{(i_1)}, \dots, A^{(i_N)}, B^{(i_N)}, \eta_{m^{(i_1 \dots i_N)}}) $$
satisfying the conditions

\medskip
\begin{tabular}{ll}
\emph{(I)} &  $A^{(i_j)} =0$ for all $j=1,\dots,N$, and $\lambda(  B^{(i_N)})\tilde v^{(i_N)}= 0$; \\[2pt]
\emph{(II)} &  $\eta_{m^{(i_1 \dots i_N)}} \in \mathrm{Ann}\big (E^{(i_j)} _{m^{(i_1 \dots i_N)}}\big )$ for all $j=1,\dots, N$; \\[2pt]
\emph{(III)} &   $ \eta_{m^{(i_1 \dots i_N)}}\in\mathrm{Ann} \big (  F^{(i_N)} _{m^{(i_1 \dots i_N)}}\big )$.
\end{tabular}
\medskip

\noindent
Furthermore,  $\Crit(\, ^{(i_1\dots i_N)} \phw)$ is a $\Cinft$-submanifold of codimension $2\kappa$, where $\kappa=\dim G/H_L$ is the dimension of a principal orbit.
\end{theorem}
\begin{proof}
 To begin with, let  $\sigma_{i_1} \cdots \sigma_{i_N}\not=0$, so that all $\tau_{i_j}$ are non-zero. In this case, the sequence of monoidal transformations $\zeta^{\rho_{i_1}}_{i_1} \circ \dots \circ  \zeta^{\rho_{i_1}\dots \rho_{i_N}}_{i_1\dots i_N}\circ \delta_{i_1\dots i_N}$ constitutes a diffeomorphism, so that 
 \begin{align*}
 \mathrm{Crit}(\, ^{(i_1\dots i_N)} \tilde \psi^{tot})_{\sigma_{i_1} \cdots \sigma_{i_N}\not=0}=\{&(\sigma_{i_1}, \dots, \sigma_{i_N}, x^{(i_1)}, \dots, x^{(i_N)},  \tilde v ^{(i_N)}, A^{(i_1)}, \dots, A^{(i_N)}, B^{(i_N)}, \eta_{m^{(i_1 \dots i_N)}}):  \\ & (\eta_{m^{_{(i_1\dots i_N)}}}, X^{_{(i_1\dots i_N)}}) \in \mathrm{Crit}(\psi), \quad  {\sigma_{i_1} \cdots \sigma_{i_N}\not=0}  \}.
  \end{align*}
  Now, 
 \bqn
 (\eta_{m^{(i_1 \dots i_N)}},X^{(i_1\dots i_N)})\in \mathrm{Crit}(\psi) \quad \Leftrightarrow \quad  \eta_{m^{(i_1 \dots i_N)}} \in \Omega,  \quad \tilde X^{(i_1\dots i_N)}_{\eta_{m^{(i_1 \dots i_N)}}}=0.
 \eqn
Furthermore, $\tilde X_\eta=0$ clearly implies $\tilde X_{\pi(\eta)} =\pi_{\ast}(\tilde X_\eta)=0$. 
Since the point $m^{(i_1\dots i_N)}$ lies in a slice at $x^{(i_1)}$, the condition $\tilde X^{(i_1\dots i_N)}_{m^{(i_1\dots i_N)}}=0$ means that 
the vector field  $\tilde X^{(i_1\dots i_N)}$ must vanish at $x^{(i_1)}$ as well. Hence, $X^{(i_1\dots i_N)} \in \g_{x^{(i_1)}}$, since
\bqn
\g_m=\mathrm{Lie}{(G_m}) =\mklm{X \in \g: \tilde X_m=0}, \qquad m \in M.
\eqn
 Now
\begin{align*}
\g_{x^{(i_N)}} \subset \g_{x^{(i_{N-1})}} \subset \dots \subset \g_{x^{(i_1)}}
\end{align*}
and $\g_{x^{(i_{j+1})}}^\perp \subset \g_{x^{(i_j)}}$ imply
\bqn
\tilde X^{(i_1\dots i_N)}_{x^{(i_1)}}= {\tau_{i_1} \dots \tau_{i_N}\sum \alpha_r^{(i_1)} (\tilde A_r ^{(i_1)}})_{ x^{(i_1)}} =0.
\eqn
Thus we conclude $\alpha^{(i_1)} =0$, which gives $X^{(i_2\dots i_N)}=X^{(i_1\dots i_N)} \in \g_{m^{(i_1\dots i_N)}}$, and consequently $X^{(i_2\dots i_N)} \in \g_{m^{(i_2\dots i_N)}}$ by \eqref{eq:hola}. A repetition of the above argument yields that  the condition $\tilde X^{(i_1\dots i_N)}_{m^{(i_1\dots i_N)}}=0$ is equivalent to (I) in the case that all $\sigma_{i_j}$ are different from zero. Actually, the same argument shows that for  $\sigma_{i_j}\not=0$
\bq
\label{eq:G}
\g_{m^{(i_1 \dots i_N)}} = \g _{\tilde v ^{(i_N)}},
\eq
since $\g_{\tilde v ^{(i_N)}} \subset \g_{x ^{(i_N)}}$. Next, $ \eta_{m^{(i_1 \dots i_N)}} \in \Omega$ means that 
\bqn
\J(\eta_{m^{(i_1 \dots i_N)}})(X)=\eta_{m^{(i_1 \dots i_N)}} (\tilde X_{m^{(i_1 \dots i_N)}}) =0 \qquad \forall X \in \g,
\eqn
which by \eqref{eq:Ann} is equivalent to  $ \eta_{m^{(i_1\dots i_ N)}} \in \mathrm{Ann}(T_{m^{(i_1\dots i_ N)}} (G \cdot {m^{(i_1\dots i_ N)}}))$. If $\sigma_{i_j}\not=0$ for all $j=1,\dots,N$, (II) and (III) imply that 
\bqn 
\eta_{m^{(i_1\dots i_ N)}}\Big ( (\exp_{x^{(i_1)}})_{\ast, \tau_{i_1}  m^{(i_2\dots i_N)} } [\dots  (\exp_{x^{(i_{j-1})}})_{\ast, \tau_{i_{N-1}}m^{(i_N)}}[\lambda(\g_{x^{(i_{N-1})}})  m ^{(i_N)}] \dots \big ] \Big )=0,
\eqn
since $ \g_{x^{(i_{N-1})}} =  \g_{x^{(i_{N})}}\oplus  \g_{x^{(i_{N})}}^\perp$. By repeatedly using this argument, we  conclude with \eqref{eq:tangent}  that for $\sigma_{i_j}\not= 0$
\bq
\label{eq:IVb}
\mathrm{(II), \, (III)} \quad \Longleftrightarrow \quad \eta_{m^{(i_1 \dots i_N)}} \in \mathrm{Ann}(T_{m^{(i_1\dots i_ N)}} (G \cdot {m^{(i_1\dots i_ N)}})).
\eq
Taking everything together therefore gives
\begin{align}
\begin{split}
\label{eq:XX}
\mathrm{Crit}(\, ^{(i_1\dots i_N)}& \psi^{tot})_{\sigma_{i_1}\cdots \sigma_{i_N}\not=0}\\ 
 &=\{(\sigma_{i_1}, \dots, \sigma_{i_N}, x^{(i_1)}, \dots, x^{(i_N)},  \tilde v ^{(i_N)}, A^{(i_1)}, \dots, A^{(i_N)}, B^{(i_N)}, \eta_{m^{(i_1 \dots i_N)}}):  \\    &{\sigma_{i_1} \cdots \sigma_{i_N}\not=0}, \, \text{(I)-(III) are fulfilled and $ \tilde B^{(i_N),\mathrm{v}}_{\eta_{m^{(i_1 \dots i_N)}}}=0$}   \}.
\end{split}
  \end{align}  
  Here $\X_\eta^{\mathrm{v}}$ denotes the vertical component of  a vector field $\X \in T(T^\ast M)$ with respect to the decomposition $T_\eta(T^\ast M)=T^{\mathrm{v}} \oplus T^{\mathrm{h}}$, $T^{\mathrm{v}}$ being  the tangent space to the fiber, and $T^{\mathrm{h}}$ the tangent space to the zero section at $\eta$. We now assert that   
$$ \mathrm{Crit}(\, ^{(i_1\dots i_N)} \tilde  \psi^{wk})=\overline{\mathrm{Crit}(\, ^{(i_1\dots i_N)} \tilde \psi^{tot})_{\sigma_{i_1} \cdots \sigma_{i_N}\not=0}}.$$
To show this, let  $(\kappa,\mathcal{O})$ be a chart on $M$ with coordinates $\kappa(m)=(q_1,\dots, q_n)$, and introduce on $T^\ast \mathcal{O}$ the coordinates
\bqn  
\eta_m = \sum p_i (dq_i)_m, \qquad \tilde\kappa(\eta)= (q_1, \dots, q_n,p_1,\dots,p_n),  \quad \eta \in T^\ast \mathcal{O}.
\eqn
Write  $\eta_{m^{(i_1 \dots i_N)}} =\sum p_i (\d q_i)_{m^{(i_1 \dots i_N)}}$, and still assume that all $\sigma_{i_j}$ are different from zero. Then all $\tau_{i_j}$ are different from zero, too, and $\gd_p \, ^{(i_1\dots i_N)} \phw=0$ is equivalent to
\begin{gather*}
 \gd _p \J(\eta_{m^{(i_1 \dots i_N)}})( X^{(i_1\dots i_N)})= ( \d q_1 (\tilde X^{(i_1\dots i_N)}_{m^{(i_1\dots i_N)}}), \dots, \d q_n (\tilde X^{(i_1\dots i_N)}_{m^{(i_1\dots i_N)}}))=0,
\end{gather*}
which gives us the condition $\tilde X^{(i_1\dots i_N)}_{m^{(i_1\dots i_N)}}=0$. By \eqref{eq:G} we therefore obtain condition I) in the case that all $\sigma_{i_j}$ are different from zero. Let next $N_{x^{(i_1)}} ( G \cdot x^{(i_1)})$ be the normal space in $T_{x^{(i_1)}}M$ to the orbit $G \cdot x^{(i_1)}$, on which $G_{x^{(i_1)}}$ acts, and define  $N_{x^{(i_{j+1})}} ( G_{x^{(i_{j})}} \cdot x^{(i_{j+1})})$ successively as the normal space to the orbit $ G_{x^{(i_{j})}} \cdot x^{(i_{j+1})}$  in the $ G_{x^{(i_{j})}}$-space $N_{x^{(i_j)}} ( G_{x^{(i_{j-1})}} \cdot x^{(i_j)})$, where we understand that $G_{x^{(i_0)}}=G$. By Bredon \cite[page 308]{bredon}, these actions can be assumed to be orthogonal. Set  
\bq
\label{eq:V}
V^{(i_1\dots i_{j})}= \bigcap_{r=1} ^{j} N_{x^{(i_r)}} ( G_{x^{(i_{r-1})}} \cdot x^{(i_r)})= N_{x^{(i_j)}} ( G_{x^{(i_{j-1})}} \cdot x^{(i_j)}).
\eq
With the identification $T_0(T_mM) \simeq T_mM$  one has 
\bq
\label{eq:identif}
(\exp_m)_{\ast,0}: T_0(T_mM) \longrightarrow T_mM, \qquad  (\exp_m)_{\ast,0}\simeq \id,
\eq
and similarly $(\exp_{x^{(i_j)}})_{\ast,0}\simeq \id$ for all $j=2,\dots, N$. Therefore, if   $\tau_{i_j}=0$ for all $j$, then $E^{(i_1)}_{x^{(i_1)}}  = T_{x^{(i_1)}}(G\cdot x^{(i_1)})$, and  
\bqn 
E^{(i_j)}_{x^{(i_1)}} \simeq T_{x ^{(i_j)}}(G_{x^{(i_{j-1})}} \cdot x ^{(i_j)}) \subset  V^{(i_{1} \dots i_{j-1})}, \qquad 2 \leq j \leq N, 
\eqn
while $F^{(i_N)}_{x^{(i_1)}}\simeq T_{\tilde v^{(i_N)}} (G_{x^{(i_{N})}}\cdot \tilde v^{(i_N)})  \subset V^{(i_1 \dots i_{N})}$.  Therefore $E^{(i_j)}_{x^{(i_1)}} \cap V^{(i_1 \dots i_{j})}=\mklm{0}$, so that we obtain the direct sum of vector spaces
\bq
\label{eq:directsum}
E_{x^{(i_1)}}^{(i_1)}\oplus E_{x^{(i_1)}}^{(i_2)}\oplus \dots  \oplus E_{x^{(i_1)}}^{(i_N)}\oplus F_{x^{(i_1)}}^{(i_N)} \subset T_{x^{(i_1)}}M.
\eq
Let now one of the  $\sigma_{i_j}$ be equal to zero, so that  all $\tau_{i_j}$ are zero. 
With the identification \eqref{eq:identif} one has
\begin{align}
\label{eq:B}
\, ^{(i_1\dots i_N)}\tilde \psi^ {wk}&= \sum p_i \, dq_i \Big (\widetilde{ A^{(i_1)}}_{x^{(i_1)}} + \sum_{j=2}^N  \lambda(A^{(i_j)})  x ^{(i_j)} + \lambda(B^{(i_N)})  \tilde v ^{(i_N)}  \Big ),
\end{align}
and  $\gd_p \, ^{(i_1\dots i_N)} \phw=0$ is equivalent to
\bqn
\widetilde{ A^{(i_1)}}_{x^{(i_1)}} + \sum_{j=2}^N  \lambda(A^{(i_j)})  x^{(i_j)} + \lambda(B^{(i_N)})  \tilde v ^{(i_N)}=0.
\eqn
Since $x^{(i_j)} \in \gamma^{(i_{j-1})} (S_{i_1\dots i_{j-1}})_{x^{(i_{j-1})}})\subset V^{(i_1\dots i_{j-1})}$, we see that for every $j=2,\dots, N$
\bqn 
\lambda\Big (\sum_r \alpha_r^{(i_j)} A_r^{(i_j)} \Big ) \, x^{(i_j)} \in   T_{x^{(i_j)}} ( G_{x^{(i_{j-1})}} \cdot x^{(i_j)})\subset V^{(i_1\dots i_{j-1})} .
\eqn 
In addition, $(\tilde  A_r^{(i_1)})_{x^{(i_1)}}\in   T_{x^{(i_1)}} ( G \cdot x^{(i_1)})$, and
$
\lambda\Big (\sum_r \beta^{(i_N)} _r B_r^{(i_N)}\Big ) \tilde v ^{(i_N)} \in V^{(i_1\dots i_{N})}
$, so that taking everything together we obtain with \eqref{eq:directsum} for arbitrary $\sigma_{i_j}$
\bqn
\gd_p \, ^{(i_1\dots i_N)} \phw=0 \quad \Longleftrightarrow \quad \mathrm{(I)}.
\eqn
 In particular, one concludes that $\, ^{(i_1\dots i_N)} \phw$ must vanish on its critical set. Since 
\bqn 
d(\, ^{(i_1\dots i_N)} \tilde \psi^{tot})= d(\tau_{i_1}\dots \tau_{i_N}) \cdot \, ^{(i_1\dots i_N)}  \tilde\psi^{wk} +  \tau_{i_1}\dots \tau_{i_N} d\,( ^{(i_1\dots i_N)}  \tilde\psi^{wk}),
\eqn
one sees that 
\bqn 
 \mathrm{Crit}(\, ^{(i_1\dots i_N)}  \tilde\psi^{wk})\subset \mathrm{Crit}(\, ^{(i_1\dots i_N)}  \tilde\psi^{tot}).
\eqn
In turn, the vanishing of $\psi$ on its critical set implies
\bqn 
 \mathrm{Crit}(\, ^{(i_1\dots i_N)}  \tilde \psi^{wk})_{\sigma_{i_1}\cdots \sigma_{i_N}\not=0}=  \mathrm{Crit}(\, ^{(i_1\dots i_N)}  \tilde\psi^{tot})_{\sigma_{i_1}\dots \sigma_{i_N}\not=0}.
\eqn
Therefore, by continuity, 
\bq
\label{eq:XXbis}
\overline{\mathrm{Crit}(\, ^{(i_1\dots i_N)}  \tilde\psi^{tot})_{\sigma_{i_1}\cdots \sigma_{i_N}\not=0}} \subset  \mathrm{Crit}(\, ^{(i_1\dots i_N)}  \tilde\psi^{wk}).
\eq
In order to see the converse inclusion, let us consider next the $\alpha$-derivatives. Clearly,  
\begin{align*}
\gd_{\alpha^{(i_1)}} \, ^{(i_1\dots i_N)} \phw=0 \quad & \Longleftrightarrow \quad \eta_{m^{(i_1 \dots i_N)}}(\tilde Y _{m^{(i_1 \dots i_N)}})=0 \quad \forall \, Y \in \g_{x^{(i_1)}}^\perp.
\end{align*}
For the remaining  derivatives one computes 
\begin{gather*}
\gd_{\alpha_r^{(i_j)}} \, ^{(i_1\dots i_N)} \phw \\ =\eta_{m^{(i_1 \dots i_N)}} \Big ( (\exp_{x^{(i_1)}})_{\ast, \tau_{i_1}  m^{(i_2\dots i_N)} }\big [ \dots (\exp_{x^{(i_{j-1})}})_{\ast, \tau_{i_{j-1}}m^{(i_j\dots i_N)}}[\lambda(A^{(i_j)}_r)  m ^{(i_j\dots i_N)}] \dots \big ] \Big ), 
\end{gather*}
from which one deduces  that for $j=2,\dots, N$
\begin{align*}
\gd_{\alpha^{(i_j)}} \, ^{(i_1\dots i_N)} \phw=0 \quad  \Longleftrightarrow   \quad \forall \, Y \in &\g_{x^{(i_j)}}^\perp \\ 
\eta_{m^{(i_1 \dots i_N)}} \Big ( (\exp_{x^{(i_1)}})_{\ast, \tau_{i_1}  m^{(i_2\dots i_N)} } \big [\dots  (\exp_{x^{(i_{j-1})}})_{\ast, \tau_{i_{j-1}}m^{(i_j\dots i_N)}}&[\lambda(Y)  m ^{(i_j\dots i_N)}] \dots \big ] \Big )=0.
\end{align*}
In a similar way,
\begin{align*}
\gd_{\beta^{(i_j)}} \, ^{(i_1\dots i_N)} \phw=0 \quad  \Longleftrightarrow   \quad \forall \, Z \in &\g_{x^{(i_N)}} \\ 
\eta_{m^{(i_1 \dots i_N)}} \Big ( (\exp_{x^{(i_1)}})_{\ast, \tau_{i_1}  m^{(i_2\dots i_N)} } \big [\dots  (\exp_{x^{(i_{N})}})_{\ast, \tau_{i_{N}}\tilde v^{(i_N)}}&[\lambda(Z)  \tilde v ^{(i_N)}] \dots \big ] \Big )=0.
\end{align*}
by which the necessity of the conditions (I)--(III) is established. In order to see their sufficiency, let them be fulfilled, and assume again that $\sigma_{i_j}\not=0$ for all $j=1,\dots,N$. Then \eqref{eq:IVb} implies that  $ \eta_{m^{(i_1 \dots i_N)}} \in \mathrm{Ann}(T_{m^{(i_1\dots i_ N)}} (G \cdot {m^{(i_1\dots i_ N)}}))$. Now, if $\sigma_{i_j}\not=0$, $G  \cdot m^{(i_1\dots i_N)}$ is of principal type $G/ H_L$ in $M$, so that the isotropy group of $m^{(i_1\dots i_N)}$ must act trivially on $N_{m^{(i_1\dots i_N)}}(G\cdot m^{(i_1\dots i_N)})$, compare Bredon \cite[page 181]{bredon}. If therefore $ \X= \X_T+ \X_N$ denotes an arbitrary element in $ T_{m^{(i_1\dots i_N)}}M = T_{m^{(i_1\dots i_ N)}} (G \cdot {m^{(i_1\dots i_ N)}}))\oplus N_{m^{(i_1\dots i_ N)}} (G \cdot {m^{(i_1\dots i_ N)}}))$, and $g \in G_{m^{(i_1\dots i_N)}}$, one computes 
\begin{align*} 
g \cdot \eta_{m^{(i_1\dots i_N)}}( \X)&= [ ( L_{g^{-1}})^\ast _{gm^{(i_1\dots i_N)}} \eta_{m^{(i_1\dots i_N)}}]( \X)= \eta_{m^{(i_1\dots i_N)}} ( ( L_{g^{-1}}) _{\ast, m^{(i_1\dots i_N)}}( \X_N))\\
&= \eta_{m^{(i_1\dots i_N)}} (  \X_N)= \eta_{m^{(i_1\dots i_N)}} (  \X).
\end{align*}
In view of  $\lambda( B^{(i_N)})\tilde v^{(i_N)}=0$ and \eqref{eq:G} we therefore get the condition $\tilde B^{(i_N), \mathrm{v}}_{\eta_{m^{(i_1\dots i_N)}}}=0$.  Let us now assume that one of the $\sigma_{i_j}$ equals zero. Then  
\begin{align}
\label{eq:VII}
\mathrm{(II), \, (III)} \quad \Leftrightarrow \quad & \left \{
\begin{array}{l}
\eta_{x^{(i_1)}} \in \mathrm{Ann}(T_{x^{(i_j)}} (G_{x^{(i_{j-1})}} \cdot {x^{(i_j)}}))  \quad \forall \, j=1, \dots, N, \\
\eta_{x^{(i_1)}} \in \mathrm{Ann}(T_{\tilde v^{(i_N)}} (G_{x^{(i_{N})}} \cdot {\tilde v^{(i_N)}})).
\end{array} \right.
\end{align}
\begin{lemma}
The orbit of the point $\tilde v^{(i_N)}$ in the $G_{x^{(i_N)}}$-space $V^{(i_1\dots i_N)}$ is of principal type.
\end{lemma}
\begin{proof}[Proof of the lemma]
By assumption, for  $\sigma_{i_j}\not=0$,  $1 \leq j \leq  N$, the $G$-orbit of $m^{(i_1\dots i_N)}$ is of principal type $G/ H_L$ in $M$. The theory of compact group actions then implies that this is equivalent to the fact that $m^{(i_2 \dots i_N)} \in V^{(i_1)}$ is of principal type in the $G_{x^{(i_1)}}$-space $V^{(i_1)}$, see Bredon \cite[page 181]{bredon},  which in turn is equivalent to the fact that $m^{(i_3 \dots i_N)} \in V^{(i_1i_2)}$ is of principal type in the $G_{x^{(i_2)}}$-space $V^{(i_1i_2)}$, and so forth. Thus, $m^{(i_j \dots i_N)} \in V^{(i_1 \dots i_{j-1})}$ must be of principal type in the $G_{x^{(i_{j-1})}}$-space $V^{(i_1\dots i_{j-1})}$ for all $j=1,\dots N$, and the assertion follows.
\end{proof}
As a consequence of the previous lemma,  the stabilizer of $\tilde v^{(i_N)} $  must act trivially on $N_{\tilde v^{(i_N)}} ( G_{x^{(i_N)}} \cdot\tilde v^{(i_N)} )$. If therefore  $ \X= \X_{T}+  \X_N$ denotes an arbitrary element in 
\begin{align*}
T_{x^{(i_1)}} M 
 \simeq \bigoplus_{j=1}^N T_{x^{(i_j)}} (G_{x^{(i_{j-1})}} \cdot {x^{(i_j)}}) \oplus  T_{\tilde v^{(i_N)}} (G_{x^{(i_{N})}}\cdot \tilde v^{(i_N)}) \oplus  N_{\tilde v^{(i_N)}} (G_{x^{(i_{N})}}\cdot \tilde v^{(i_N)}),
\end{align*}
 we obtain with \eqref{eq:VII} 
\begin{align*} 
g \cdot \eta_{x^{(i_1)}}( \X)&= [ ( L_{g^{-1}})^\ast _{gx^{(i_1)}} \eta_{x^{(i_1)}}]( \X)= \eta_{x^{(i_1)}} ( ( L_{g^{-1}}) _{\ast, x^{(i_1)}}( \X_N))\\
&= \eta_{x^{(i_1)}} (  \X_N)= \eta_{x^{(i_1)}} (  \X), \qquad  g \in G_{\tilde v^{(i_N)}}.
\end{align*}
Collecting everything together we have shown for arbitrary $\sigma_{i_j}$ that 
\begin{align}
\label{eq:IV}
 \gd_{p, \alpha^{(i_1)}, \dots , \alpha^{(i_N)} , \beta^{(i_N)}} \, ^{(i_1\dots i_N)} \phw=0 \quad \Longleftrightarrow \quad \mathrm{(I), \,(II), \,(III)} \quad & \Longrightarrow \quad \tilde B^{(i_N),\mathrm{v}}_{\eta_{m^{(i_1\dots i_N)}}}=0.
\end{align}
By \eqref{eq:XX} and \eqref{eq:XXbis} we therefore conclude 
\bq
\label{eq:CC}
\overline{\mathrm{Crit}(\, ^{(i_1\dots i_N)} \tilde \psi^{tot})_{\sigma_{i_1}\cdots \sigma_{i_N}\not=0}} =  \mathrm{Crit}(\, ^{(i_1\dots i_N)} \tilde \psi^{wk}).
\eq
Thus we have computed the critical set of $\, ^{(i_1\dots i_N)} \phw$, and it remains to show that it is a $\Cinft$-submanifold of codimension $2\kappa$. By our previous considerations, we have the characterization
\begin{align}
\label{eq:C}
\begin{split}
&\Crit(\, ^{(i_1\dots i_N)} \phw)\\ 
= \Big \{ A^{(i_j)}=0, \quad \lambda(B^{(i_N)}) \tilde v ^{(i_N)}&=0, \quad  \eta_{m^{(i_1\dots i_N)}}\in \mathrm{Ann} \Big(\bigoplus_{j=1}^N E^{(i_j)}_{m^{(i_1\dots i_N)}}\oplus F^{(i_N)}_{m^{(i_1\dots i_N)}}\Big ) \Big \}.
\end{split}
\end{align}
Note that the condition $\tilde  B^{(i_N), \mathrm{v}}_{\eta_{m^{(i_1\dots i_N)}}}  =0$ is already implied by the others. Now, $
\dim E^{(i_j)}_{m^{(i_1\dots i_N)}} = \dim G_{x^{(i_{j-1})}} \cdot x ^{(i_j )}$.
Since for  $\sigma_{i_1}\cdots \sigma_{i_N}\not=0$ the $G$-orbit of $m^{(i_1\dots i_N)}$ is of principal type $G/ H_L$ in $M$, one computes in this case with \eqref{eq:tangent}
\begin{align*}
\kappa =& \dim G \cdot m^{(i_1\dots i_N)}= \dim T_{m^{(i_1 \dots i_N)}} ( G\cdot m^{(i_1 \dots i_N)})\\
=&\dim [E^{(i_1)}_{m^{(i_1 \dots i_N)}}\oplus  \bigoplus _{j=2}^N \tau_{i_1}\dots \tau_{i_{j-1}} E^{(i_j)} _{m^{(i_1 \dots i_N)}}  \oplus  \tau_{i_1}\dots \tau_{i_N} F^{(i_N)}_{m^{(i_1 \dots i_N)}}] \\
 =& \sum_{j=1} ^N \dim E^{(i_j)}_{m^{(i_1\dots i_N)}} + \dim F^{(i_N)}_{m^{(i_1\dots i_N)}}.
\end{align*}
But since the dimension of the spaces $E^{(i_j)}_{m^{(i_1 \dots i_N)}}$ and $F^{(i_N)}_{m^{(i_1 \dots i_N)}}$  does not depend on the variables $\sigma_{i_j}$, we obtain the equality 
\bq
\label{eq:kappa}
\kappa=\sum_{j=1} ^N \dim E^{(i_j)} _{m^{(i_1 \dots i_N)}}+ \dim F^{(i_N)}_{m^{(i_1 \dots i_N)}}
\eq
for arbitrary ${m^{(i_1 \dots i_N)}}$. Note that, in contrast, the dimension of  $T_{m^{(i_1 \dots i_N)}} ( G\cdot m^{(i_1 \dots i_N)})$ collapses, as soon as one of the $\tau_{i_j}$ becomes zero. Since the annihilator of a subspace of $T_m M$ is itself a linear subspace of $T^\ast_m M$,  we arrive at a vector bundle with $(n-\kappa)$-dimensional fiber that is locally given by the trivialization 
\bqn 
\Big ((\sigma_{i_j}, x^{(i_j)},  \tilde v ^{(i_N)}), \mathrm{Ann}  \big(\bigoplus _{j=1}^N  E^{(i_j)} _{m^{(i_1 \dots i_N)}}  \oplus  F^{(i_N)}_{m^{(i_1 \dots i_N)}}\big ) \Big )\mapsto (\sigma_{i_j}, x^{(i_j)},  \tilde v ^{(i_N)}).
\eqn
Consequently, by equation \eqref{eq:C} we see that $\Crit(\, ^{(i_1\dots i_N)} \phw)$ is equal to the total space of the fiber product of the mentioned vector bundle with the isotropy algebra bundle given by the local trivialization
\bqn 
(\sigma_{i_j}, x^{(i_j)},  \tilde v ^{(i_N)}, \g_{\tilde v ^{(i_N)}})\mapsto (\sigma_{i_j}, x^{(i_j)},  \tilde v ^{(i_N)}).
\eqn
Lastly, since by equation \eqref{eq:G} we have $\g_{\tilde v ^{(i_N)}}=\g_{m^{(i_1,\dots, i_N)}}$ in case that all $\sigma_{i_j}$ are different from zero, we necessarily have $\dim \g_{\tilde v ^{(i_N)}}=d-\kappa$, which concludes the proof of the theorem.
\end{proof}

\section{Phase analysis of the weak transforms. Non-degeneracy of the transversal Hessians}

In this section, we  prove the non-degeneracy of the transversal Hessians of the weak transforms. To begin with, let $M$ be a $n$-dimensional Riemannian manifold, and $C$ the critical set of a function $\psi \in \Cinft(M)$, which is assumed to be a smooth submanifold in a chart $\mathcal{O} \subset M$. Let further 
\bqn
\alpha:(x,y) \mapsto m, \qquad \beta:(q_1,\dots, q_n) \mapsto m, 	\qquad m \in \mathcal{O},
\eqn
be two systems of  local coordinates on $\mathcal{O}$, such that $\alpha(x,y) \in C$ if and only if $y=0$. 
As one computes, the transversal Hessian is given by
\bq
\label{eq:Hess}
\gd_{y_k} \gd_{y_l} (\psi \circ \alpha)(x,0)= \mathrm{Hess}\,  \psi_{|\alpha(x,0)} (\alpha_{\ast,(x,0)}(\gd_{y_k}),\alpha_{\ast,(x,0)}(\gd_{y_l})),
\eq
Let us now write $x=(x',x'')$, and consider the restriction of $\psi$ onto  the $\Cinft$-submanifold
\bdm
M_{c'}=\mklm { m \in \mathcal{O}: m=\alpha(c',x'',y)}.
\edm
We write $\psi_{c'}=\psi_{|M_{c'}}$, and denote the critical set of $\psi_{c'}$ by $C_{c'}$, which contains $C \cap M_{c'}$ as a subset. Introducing on $M_{c'}$ the local coordinates $
\alpha':(x'',y) \mapsto \alpha(c',x'',y)$, we obtain 
\bqn
\gd_{y_k} \gd_{y_l} (\psi_{c'} \circ \alpha')(x'',0)= \mathrm{Hess}\,  \psi_{c'|\alpha(x'',0)} (\alpha'_{\ast,(x'',0)}(\gd_{y_k}),\alpha'_{\ast,(x'',0)}(\gd_{y_l})).
\eqn
Let us now assume $C_{c'}=C \cap M_{c'}$, a transversal intersection.  Then $C_{c'}$ is a submanifold of $M_{c'}$, and the normal space to $C_{c'}$ as a submanifold of $M_{c'}$ at a point $\alpha'(x'',0)$ is spanned by the vector fields $\alpha'_{\ast,(x'',0)} (\gd _{y_k})$.
Since clearly
\bqn
\gd_{y_k} \gd_{y_l} (\psi_{c'} \circ \alpha')(x'',0)=\gd_{y_k} \gd_{y_l} (\psi \circ \alpha)(x,0),\qquad x=(c',x''),
\eqn
we thus have proven the following
\begin{lemma}
\label{lemma:A}
Assume that $C_{c'}=C \cap M_{c'}$. Then the restriction
\bqn
\mathrm{Hess} \, \psi({\alpha(c',x'',0)})_{|N_{\alpha(c',x'',0)}C}
\eqn
of the Hessian of $\psi$ to the normal space $N_{\alpha(c',x'',0)}C$ defines a non-degenerate quadratic form if, and only if the restriction
\bqn
\mathrm{Hess} \, \psi_{c'}({\alpha'(x'',0)})_{|N_{\alpha'(x'',0)}C_{c'}}
\eqn
of the Hessian of $\psi_{c'}$ to the normal space $N_{\alpha'(x'',0)}C_{c'}$ defines a non-degenerate quadratic form.
\end{lemma} \qed

We can now state the main result of this section, the notation being the same as in the previous ones.

\begin{theorem}
\label{thm:II}
Let  $\mklm{(H_{i_1}), \dots, (H_{i_{N}})}$ be a maximal, totally ordered subset of non-principal isotropy types of the $G$-action on $M$, and  
  $ \zeta^{\rho_{i_1}}_{i_1} \circ \dots \circ  \zeta^{\rho_{i_1}\dots \rho_{i_N}}_{i_1\dots i_N}$  a corresponding sequence of local monoidal transformations labeled by the indices $\rho_{i_1},\dots ,\rho_{i_N}$. Consider the corresponding factorization 
\bqn
\, ^{(i_1\dots i_N)} \tilde \psi^{tot}=\tau_{i_1} \dots \tau_{i_N} \, ^{(i_1\dots i_N)}\tilde \psi^ {wk, \, pre}=\tau_{i_1}(\sigma) \dots \tau_{i_N}(\sigma) \, ^{(i_1\dots i_N)}\tilde \psi^ {wk}
\eqn
 of the phase function \eqref{eq:phase}. Then, for each point of the critical manifold $\Crit(\, ^{(i_1\dots i_N)}\tilde \psi^ {wk})$,  the restriction of 
\bqn 
\mathrm{Hess} \, ^{(i_1\dots i_N)}\tilde \psi^ {wk}
\eqn
to the normal space to $\Crit(\, ^{(i_1\dots i_N)}\tilde \psi^ {wk})$ at the given point defines a non-degenerate symmetric bilinear form.
\end{theorem}
Note that by construction, for $\tau_{i_j}\not=0$, $1\leq j\leq N$, the $G$-orbit through $m^{(i_1\dots i_N)}$ is of principal type $G/H_L$. 
For the proof of Theorem \ref{thm:II} we need  the following
\begin{lemma}
\label{lemma:Reg}
 Let  $(\eta,X) \in \Crit{(\psi)}$, and  $ \pi(\eta) \in M({H_L})$. Then $(\eta,X) \in \mathrm{Reg} \,\Crit{(\psi)}$. Furthermore, the restriction of the Hessian of $\psi$ at the point $(\eta,X)$ to the normal space ${N_{(\eta,X)} \mathrm{Reg} \,\Crit{(\psi)}}$ defines  a non-degenerate quadratic form. 
\end{lemma}
\begin{proof}
The first assertion is clear from \eqref{eq:7} and \eqref{eq:z}, since 
\bqn 
\eta \in \Omega, \quad G_{\pi(\eta)} \sim H_L \quad \Rightarrow  \quad G_\eta =G_{\pi(\eta)}.
\eqn
 To see the second, note that by the last implication
\bq
\label{eq:yz}
\eta \in \Omega \cap T^\ast M(H_{L}), \tilde X_{\pi(\eta)}=0 \quad \Longrightarrow \quad \tilde X _\eta =0.
\eq
Let now $\mklm{q_1,\dots, q_n}$ be local coordinates on $M$, $\pi(\eta)=m=m(q)$, and write $\eta_m = \sum p_i (dq_i)_m$, $X=\sum s_i X_i$, where $\mklm{X_1,\dots,X_d}$ denotes a basis of $\g$. Then 
\bqn 
\psi(\eta,X)=\sum p_i (d q_i)_m (\tilde X_m),
\eqn
and 
\bqn 
\gd_p \psi (\eta, X) =0 \quad \Longleftrightarrow  \quad \tilde X_m =0, \qquad \quad \gd_s \psi(\eta, X) =0 \quad \Longleftrightarrow \quad \eta \in \Omega.
\eqn 
As a consequence of \eqref{eq:yz}, on  $T^\ast M(H_{L})\times \g$ we  get
\bqn 
\gd_{p,s} \psi (\eta, X) =0 \quad \Longrightarrow  \quad \gd_q \psi(\eta, X) =0.
\eqn 
Let $\psi_q(p,s)$ denote the phase function regarded as a function of the coordinates $p,s$ alone, while $q$ is regarded as a parameter. 
Lemma \ref{lemma:A} then implies that  on $T^\ast M(H_{L})\times \g$ the study of the transversal Hessian of $\psi$ can be reduced to the study of the transversal Hessian of $\psi_q$. Now, with respect to the coordinates $s,p$, the Hessian of $\psi_q$ is given by 
\bqn 
\left ( \begin{array}{cc}
0 & (dq_i)_m((\tilde X_{j})_m) \\ (dq_j)_m((\tilde X_{i})_m) & 0\\
\end{array} \right ).
\eqn 
A computation shows that the kernel of the corresponding linear transformation is isomorphic to
$$
T_{p,s}( \Crit \, \psi_q)\simeq\mklm{(\tilde p,\tilde s)\in \R^n \times \R^d: \sum \tilde p_j (dq_j)_{m(q)} \in \mathrm{Ann}(T_{m(q)} ( G \cdot m(q))), \sum \tilde s_j X_j \in \g_{m(q)}}.
$$ 
The lemma then follows with the following general observation.
Let $\mathcal{B}$ be a symmetric bilinear form on an $n$-dimensional $\mathbb{K}$-vector space $V$, and $B=(B_{ij})_{i,j}$ the corresponding Gramsian matrix with respect to a basis $\mklm{v_1,\dots,v_n}$ of $V$ such that 
\bqn 
\mathcal{B}(u,w) = \sum_{i,j} u_i w_j B_{ij}, \qquad  u=\sum u_i v_i, \quad w=\sum w_i v_i.
\eqn
We denote the linear operator given by $B$ with the same letter, and write 
\bqn 
V =\ker B \oplus W.
\eqn
Consider the restriction $\mathcal{B}_{|W \times W}$ of $\mathcal{B}$ to $W\times W$, and assume that $\mathcal{B}_{|W\times W}(u,w) =0$ for all $u \in W$, but $w\not=0$. Since the Euclidean scalar product in $V$ is non-degenerate, we necessarily must have $Bw=0$, and consequently $ w \in \ker  B \cap W=\mklm{0}$, which is a contradiction. Therefore $\mathcal{B}_{|W \times W}$ defines a non-degenerate symmetric bilinear form. 
\end{proof}

\begin{proof}[Proof of Theorem \ref{thm:II}] As before, let $m=m(q_1,\dots,q_n)$ be local coordinates on $M$, and write $\eta_m=\sum p_i (dq_i)_m$. For $\sigma_{i_1} \cdots \sigma_{i_N}\not=0$,  the sequence of monoidal transformations $ \zeta^{\rho_{i_1}}_{i_1} \circ \dots \circ  \zeta^{\rho_{i_1}\dots \rho_{i_N}}_{i_1\dots i_N} \circ \delta_{i_1\dots i_N}$  constitutes a diffeomorphism, so that by the previous lemma the restriction of 
\bqn 
\mathrm{Hess} \, ^{(i_1\dots i_N)} \tilde \psi^{tot} (\sigma_{i_j},x^{(i_j)},\tilde v^{(i_N)}, \alpha^{(i_j)}, \beta^{(i_N)},p)
\eqn
to the normal space of 
\bqn 
\mathrm{Crit}(\, ^{(i_1\dots i_N)} \psi^{tot})_{\sigma_{i_1}\cdots \sigma_{i_N}\not=0}
\eqn
defines a non-degenerate quadratic form.  Next, one computes for the Hessian of the total transform
\begin{align*}
\left (\frac{\gd^2 \, ^{(i_1\dots i_N)}\tilde \psi^ {tot}}{\gd \gamma_k\gd \gamma_l}  \right )_{k,l} & = \tau_{i_1}(\sigma) \cdots  \tau_{i_N }(\sigma) 
\left (\frac{\gd^2 \, ^{(i_1\dots i_N)}\tilde \psi^ {wk}}{\gd \gamma_k\gd \gamma_l}  \right )_{k,l} \\ &+
\left (\begin{array}{cc}
\left (   \frac{ \gd ^2 (\tau_{i_1}(\sigma) \cdots  \tau_{i_N }(\sigma))}{\gd \sigma_{i_r}\sigma_{i_s}} \right )_{r,s} & 0 \\ 0 & 0
\end{array}\right ) \, ^{(i_1\dots i_N)}\tilde \psi^ {wk} +R,
\end{align*}
where $R$ is a matrix whose entries contain first order derivatives of $^{(i_1\dots i_N)}\tilde \psi^ {wk}$ as factors. But since $^{(i_1\dots i_N)}\tilde \psi^ {wk}$ vanishes along its critical set, and 
\bqn 
\mathrm{Crit}(\, ^{(i_1\dots i_N)} \tilde \psi^{tot})_{\sigma_{i_1}\cdots \sigma_{i_N}\not=0} =\Crit(^{(i_1\dots i_N)}\tilde \psi^ {wk})_{|\sigma_{i_1}\cdots \sigma_{i_N}\not=0}, 
\eqn
 we conclude that  the transversal Hessian of $^{(i_1\dots i_N)}\tilde \psi^ {wk}$ does not degenerate along the manifold $\Crit(^{(i_1\dots i_N)}\tilde \psi^ {wk})_{|\sigma_{i_1}\cdots \sigma_{i_N}\not=0}$. Therefore, it  remains to study the transversal Hessian of $^{(i_1\dots i_N)}\tilde \psi^ {wk}$ in the case that any of the $\sigma_{i_j}$ vanishes. Now, the proof of Theorem \ref{thm:I},  in particular  \eqref{eq:IV},  showed that
\bqn 
\gd _{p, \alpha^{(i_1)}, \dots, \alpha^{(i_N)}, \beta^{(i_N)}} \, ^{(i_1\dots i_N)}\tilde \psi^ {wk}=0 \quad \Longrightarrow \quad \gd _{\sigma_{i_1}, \dots \sigma_{i_N}, x^{(i_1)}, \dots, x^{(i_N)},\tilde v^{(i_N)}} \, ^{(i_1\dots i_N)}\tilde \psi^ {wk}=0.
\eqn
 If therefore 
$$
\, ^{(i_1\dots i_N)}\tilde \psi^ {wk}_{\sigma_{i_j}, x^{(i_j)},\tilde v^{(i_N)}}(\alpha^{(i_j)}, \beta^{(i_N)},p)
$$ 
denotes the weak transform of the phase function $\psi$ regarded as a function of the variables $(\alpha^{(i_1)},\dots, \alpha^{(i_N)}, \beta^{(i_N)},p)$ alone, while the variables $(\sigma_{i_1},\dots,\sigma_{i_N}, x^{(i_1)},\dots, x^{(i_N)},\tilde v^{(i_N)})$ are kept fixed,
\bqn 
\Crit \big ( \, ^{(i_1\dots i_N)}\tilde \psi^ {wk}_{\sigma_{i_j}, x^{(i_j)},\tilde v^{(i_N)}}\big )=\Crit \big ( \, ^{(i_1\dots i_N)}\tilde \psi^ {wk}\big )  \cap \mklm{\sigma_{i_j}, x^{(i_j)},\tilde v^{(i_N)} = \, \, \text{constant}},
\eqn
a transversal intersection. Thus, the critical set of $\, ^{(i_1\dots i_N)}\tilde \psi^ {wk}_{\sigma_{i_j}, x^{(i_j)},\tilde v^{(i_N)}}$ is equal to the fiber over $(\sigma_{i_j}, x^{(i_j)},\tilde v^{(i_N)})$ of the vector bundle
\bqn 
\Big ((\sigma_{i_j}, x^{(i_j)},\tilde v^{(i_N)}), \g _{\tilde v^{(i_N)}} \times \mathrm{Ann} \big ( \bigoplus\limits_{j=1}^N E^{(i_j)}_{m^{(i_1\dots i_N)}} \oplus F^{(i_N)}_{m^{(i_1\dots i_N)}} \big ) \Big ) \mapsto  (\sigma_{i_j}, x^{(i_j)},\tilde v^{(i_N)}),
\eqn
and in particular  a smooth submanifold. Lemma \ref{lemma:A} then implies that the study of the transversal Hessian  of $\, ^{(i_1\dots i_N)}\tilde \psi^ {wk}$ can be reduced to the study of the transversal Hessian of $\, ^{(i_1\dots i_N)}\tilde \psi^ {wk}_{\sigma_{i_j}, x^{(i_j)},\tilde v^{(i_N)}}$. The crucial fact is now contained in the following 
\begin{proposition}
\label{prop:1}
Assume that 
$\sigma_{i_1} \cdots \sigma_{i_N}=0$. Then 
\bqn 
\ker  \mathrm{Hess} \, ^{(i_1\dots i_N)}\tilde \psi^ {wk}_{\sigma_{i_j}, x^{(i_j)},\tilde v^{(i_N)}}(0,\dots, 0, \beta^{(i_N)},p)\simeq T_{(0, \dots, 0,\beta^{(i_N)},p)}  \mathrm{Crit} \big (\,^{(i_1\dots i_N)}\tilde \psi^ {wk}_{\sigma_{i_j}, x^{(i_j)},\tilde v^{(i_N)}} \big )
\eqn
for all $(0,\dots, 0, \beta^{(i_N)},p) \in \mathrm{Crit} \big (\,^{(i_1\dots i_N)}\tilde \psi^ {wk}_{\sigma_{i_j}, x^{(i_j)},\tilde v^{(i_N)}} \big )$, and arbitrary $x^{(i_j)}$, $\tilde v^{(i_j)}$.
\end{proposition}
\begin{proof}
Let $\sigma_{i_1} \cdots \sigma_{i_N}=0$. With \eqref{eq:phasewk}, or directly from  \eqref{eq:B} 
one computes
or the second derivatives of the weak transform at a critical point $(0,\dots,0,\beta^{(i_N)},p)$
\begin{align*}
\gd_{\alpha^{(i_1)}_s} \gd _{p_r} \, ^{(i_1\dots i_N)}\tilde \psi^ {wk}_{\sigma_{i_j}, x^{(i_j)},\tilde v^{(i_N)}} &=dq_r((\tilde A^{(i_1)}_{s})_{x^{(i_1)}}), \\
\gd_{\alpha^{(i_j)}_s} \gd _{p_r} \, ^{(i_1\dots i_N)}\tilde \psi^ {wk}_{\sigma_{i_j}, x^{(i_j)},\tilde v^{(i_N)}} &=dq_r(\lambda(A_s^{(i_j)}) x^{(i_j)}), \\
\gd_{\beta^{(i_N)}_s} \gd _{p_r} \, ^{(i_1\dots i_N)}\tilde \psi^ {wk}_{\sigma_{i_j}, x^{(i_j)},\tilde v^{(i_N)}} &=dq_r(\lambda(B_s^{(i_N)}) \tilde v^{(i_N)}),
\end{align*}
while all other second derivatives vanish.
Thus,  for $\sigma_{i_1}\cdots \sigma_{i_j}=0$, the Hessian of the function $\, ^{(i_1\dots i_N)}\tilde \psi^ {wk}_{\sigma_{i_j}, x^{(i_j)},\tilde v^{(i_N)}}$  with respect to the coordinates $p,\alpha^{(i_j)}, \beta^{(i_j)}$ is given on its critical set by the matrix
\bqn 
\left ( \begin{array}{ccccc}
0 & dq_r((\tilde A^{(i_1)}_{s})_{x^{(i_1)}}) & \dots & dq_r(\lambda(A_s^{(i_N)}) x^{(i_j)})& dq_r(\lambda(B_s^{(i_N)}) \tilde v^{(i_N)}) \\
\, dq_s((\tilde A^{(i_1)}_{r})_{x^{(i_1)}})& 0 & \dots & 0 & 0 \\
\vdots & \vdots &\vdots &\vdots &\vdots \\
\,dq_s(\lambda(A_r^{(i_N)}) x^{(i_j)})& 0 & \dots & 0 & 0 \\
 \,dq_s(\lambda(B_r^{(i_N)}) \tilde v^{(i_N)})& 0 & \dots & 0 & 0
\end{array} \right ). 
\eqn
Let us now compute the kernel of the linear transformation corresponding to this matrix. Cleary,  the vector $(\tilde p, \tilde \alpha^{(i_1)}, \dots, \tilde \alpha^{(i_N)}, \tilde \beta^{(i_N)})$ lies in the kernel if and only if

\medskip
\begin{tabular}{ll}
{(a)} & $\sum \tilde \alpha_s^{(i_1)} (\tilde A_s^{(i_1)})_{ x^{(i_1)}}+  \dots +\sum \tilde \alpha_s^{(i_N)} \lambda (A_s^{(i_N)}) x^{(i_N)}+ \sum \tilde \beta_s^{(i_N)} \lambda(B_s^{(i_N)}) \tilde v^{(i_N)}=0$ ; \\[2pt]
{(b)} &  $\sum \tilde p_s dq_s((\tilde Y^{(i_1)})_{x^{(i_1)}})=0$ for all $Y^{(i_1)} \in \g_{x^{(i_1)}}^\perp$,   $\sum \tilde p_s dq_s(\lambda ( \g_{x^{(i_j)}}^\perp)x^{(i_j)})=0$, $2 \leq j \leq N$;\\[2pt]
{(c)} &$\sum \tilde p_s dq_s(\lambda ( \g_{x^{(i_N)}})\tilde v^{(i_N)})=0$.
\end{tabular}
\medskip

Let $E^{(i_j)}$,  $F^{(i_N)}$, and  $V^{(i_1\dots i_N)}$ be defined as in \eqref{eq:EF} and \eqref{eq:V}. Then
\bqn 
\sum \tilde \alpha_r^{(i_j)} (\tilde A_r^{(i_1)})_{x^{(i_1)}}+  \dots+ \sum \tilde \alpha_r^{(i_N)}\lambda( A_r^{(i_N)}) x^{(i_N)}+ \sum \tilde \beta_r^{(i_N)} \lambda( B_r^{(i_N)}) \tilde v^{(i_N)} \in \bigoplus_{j=1}^N E^{(i_j)}_{x^{(i_1)}} \oplus F^{(i_N)}_{x^{(i_1)}},
\eqn
so that for condition (a) to hold, it is  necessary  and sufficient that 
\bqn 
\tilde \alpha^{(i_j)} =0, \quad 1 \leq j \leq N, \qquad \sum \tilde \beta_r^{(i_N)} \lambda(B_r^{(i_N)}) \tilde v^{(i_N)}=0.
\eqn
Condition (b) is equivalent to $\sum \tilde p_s (dq_s)_{x^{(i_1)}} \in \mathrm{Ann}( E^{(i_j)}_{x^{(i_1)}})$ for al $j=1,\dots, N$. Similarly, condition (c) is equivalent to $\sum \tilde p_s (dq_s)_{x^{(i_1)}} \in \mathrm{Ann}( F^{(i_N)}_{x^{(i_1)}})$.  On the other hand, by \eqref{eq:C},
\begin{gather*}
T_{(0, \dots, 0,\beta^{(i_N)},p)}  \mathrm{Crit} \big (\,^{(i_1\dots i_N)}\tilde \psi^ {wk}_{\sigma_{i_j}, x^{(i_j)},\tilde v^{(i_N)}} \big )= \Big \{( \tilde \alpha^{(i_1)}, \dots, \tilde \alpha^{(i_N)}, \tilde \beta^{(i_N)}, \tilde p): \tilde \alpha^{(i_j)}=0, \\  \sum \tilde  \beta^{(i_N)}_r \lambda(B^{(i_N)}_r) \in \g_{\tilde v^{(i_N)}}, \, \sum \tilde p_s (dq_s)_{x^{(i_1)}} \in \mathrm{Ann}\Big ( \bigoplus_{j=1}^N E^{(i_j)}_{x^{(i_1)}}\oplus F^{(i_N)}\Big ) \Big \},
\end{gather*}
and the proposition follows.
\end{proof}
The previous proposition   implies that for $\sigma_{i_1}\cdots \sigma_{i_N}=0$
\bqn 
\mathrm{Hess} \,^{(i_1\dots i_N)}\tilde \psi^ {wk}_{\sigma_{i_j}, x^{(i_j)},\tilde v^{(i_N)}}(0, \dots, 0,\beta^{(i_N)},p)_{|N_{(0, \dots, 0,\beta^{(i_N)},p)}  \mathrm{Crit} \big (\,^{(i_1\dots i_N)}\tilde \psi^ {wk}_{\sigma_{i_j}, x^{(i_j)},\tilde v^{(i_N)}} \big )}
\eqn
defines a non-degenerate symmetric bilinear form for all points $(0, \dots, 0,\beta^{(i_N)},p)$ lying in the critical set of $\,^{(i_1\dots i_N)}\tilde \psi^ {wk}_{\sigma_{i_j}, x^{(i_j)},\tilde v^{(i_N)}}$, and Theorem \ref{thm:II} follows with  Lemma \ref{lemma:A}.
\end{proof}

\section{Asymptotics in the resolution space }

We are now in position to  give an asymptotic description of the integrals $I_{i_1\dots i_N}^{\rho_{i_1}\dots \rho_{i_N}}(\mu)$  defined in \eqref{eq:N}.  Since the considered integrals are absolutely convergent, we can interchange the order of integration by Fubini, and write
\bqn
I_{i_1\dots i_N}^{\rho_{i_1}\dots \rho_{i_N}}(\mu)= \int_{(-1,1)^N}  \hat J_{i_1\dots i_{N}}^{\rho_{i_1} \dots \rho_{i_{N}}}\Big ( \frac \mu{\tau_{i_1}\cdots \tau_{i_N}} \Big ) \prod_{j=1}^N |\tau_{i_j}|^{c^{(i_j)} + \sum _{r=1}^j d^{(i_r)} -1} \d \tau_{i_N} \dots \d \tau_{i_1},
\eqn
where we set
\begin{gather*}
\hat J_{i_1\dots i_{N}}^{\rho_{i_1} \dots \rho_{i_{N}}}(\nu) 
 =\int_{M_{i_1}(H_{i_1})} \Big [ \int_{\gamma^{(i_1)}((S_{i_1})_{x^{(i_1)}})_{i_2}(H_{i_2})} \dots \Big [  \int_{\gamma^{(i_{N-1})}((S_{i_1\dots i_{N-1}})_{x^{(i_{N-1})}})_{i_{N}}(H_{i_{N}})} \\
\Big [ \int_{\gamma^{(i_{N})}((S_{i_1\dots i_{N}})_{x^{(i_N)}})\times \g_{x^{(i_{N})}}\times \g_{x^{(i_{N})}}^\perp \times \cdots \times \g_{x^{(i_{1})}}^\perp\times T^\ast _{m^{(i_1\dots i_N)}}W_{i_1}}  e^{i  \, ^{(i_1\dots i_N)} \tilde \psi ^{wk,pre}/\nu}  \, a_{i_1\dots i_N}^{\rho_{i_1} \dots \rho_{i_{N}}} \,    \Phi_{i_1\dots i_N}^{\rho_{i_1} \dots \rho_{i_{N}}} \\
  \d(T^\ast _{m^{(i_1\dots i_N)}}W_{i_1})  \d A^{(i_1)} \dots  \d A^{(i_N)}  \d B^{(i_N)} \d \tilde v^{(i_N)} \Big ]  \d \tau_{i_N} \d x^{(i_{N})} \dots  \Big ] \d \tau_{i_2} \d x^{(i_{2})} \Big ]\d \tau_{i_1} \d x^{(i_{1})},
\end{gather*}
and introduced the new parameter
\bqn
\nu =\frac \mu {\tau_{i_1}\cdots \tau_{i_N}}.
\eqn
Now,  for an arbitrary $0<\epsilon < T$ to be chosen later we define
\begin{align*}
\,^1I_{i_1\dots i_N}^{\rho_{i_1}\dots \rho_{i_N}}(\mu)&= \int_{((-1,1)\setminus (-\epsilon,\epsilon))^N}  \hat J_{i_1\dots i_{N}}^{\rho_{i_1} \dots \rho_{i_{N}}}\Big ( \frac \mu{\tau_{i_1}\cdots \tau_{i_N}} \Big ) \prod_{j=1}^N |\tau_{i_j}|^{c^{(i_j)} + \sum _{r=1}^j d^{(i_r)} -1} \d \tau_{i_N} \dots \d \tau_{i_1},\\
\, ^2I_{i_1\dots i_N}^{\rho_{i_1}\dots \rho_{i_N}}(\mu)&= \int_{(-\epsilon,\epsilon)^N}  \hat J_{i_1\dots i_{N}}^{\rho_{i_1} \dots \rho_{i_{N}}}\Big ( \frac \mu{\tau_{i_1}\cdots \tau_{i_N}} \Big ) \prod_{j=1}^N |\tau_{i_j}|^{c^{(i_j)} + \sum _{r=1}^j d^{(i_r)} -1} \d \tau_{i_N} \dots \d \tau_{i_1} . 
\end{align*}
\begin{lemma}
\label{lemma:kappa}
One has $c^{(i_j)} + \sum _{r=1}^j d^{(i_r)} -1\geq \kappa$ for arbitrary $j=1,\dots, N$.
\end{lemma}
\begin{proof}
We first note that for $j=1,\dots, N-1$
\bqn
c^{(i_j)} = \dim (\nu_{i_1\dots i_j})_{x^{(i_j)}} \geq \dim G_{x^{(i_j)}} \cdot m^{(i_{j+1}\dots i_N)} +1.
\eqn
Indeed, $(\nu_{i_1\dots i_j})_{x^{(i_j)}}$ is an orthogonal $G_{x^{(i_j)}}$-space, so that the dimension of the $G_{x^{(i_j)}}$-orbit of $m^{(i_{j+1}\dots i_N)}\in \gamma^{(i_j)}((S_{i_1 \dots i_j})_{x^{(i_j)}})$ can be at most $c^{(i_j)}-1$. Now, under the assumption $\sigma_{i_1}\cdots \sigma_{i_N}\not=0$, \eqref{eq:decomp},  \eqref{eq:tilde} and \eqref{eq:tildeB} imply 
\begin{align*}
T_{m^{(i_{j+1}\dots i_N)}} &(G_{x^{(i_j)}} \cdot m^{(i_{j+1}\dots i_N)}) \simeq T_{m^{(i_{1}\dots i_N)}} (G_{x^{(i_j)}} \cdot m^{(i_{1}\dots i_N)}) \\
&= E^{(i_{j+1})}_{m^{(i_1 \dots i_N)}}\oplus  \bigoplus _{k=j+2}^N \tau_{i_{j+1}}\dots \tau_{i_{k-1}} E^{(i_k)} _{m^{(i_1 \dots i_N)}}  \oplus  \tau_{i_{j+1}}\dots \tau_{i_N} F^{(i_N)}_{m^{(i_1 \dots i_N)}}, 
\end{align*}
where the distributions $E^{(i_{j})}$, $F^{(i_{N})}$ where defined in \eqref{eq:EF}. On then computes
\begin{align*}
\dim G_{x^{(i_j)}}  \cdot m^{(i_{j+1}\dots i_N)}= &\dim T_{m^{(i_{j+1}\dots i_N)}} (G_{x^{(i_j)}} \cdot m^{(i_{j+1}\dots i_N)}) \\
=&\sum_{l=j+1} ^N \dim E^{(i_l)} _{m^{(i_1 \dots i_N)}}+ \dim F^{(i_N)}_{m^{(i_1 \dots i_N)}},
\end{align*}
which implies 
\bqn
c^{(i_j)} \geq \sum_{l=j+1} ^N \dim E^{(i_l)}_{m^{(i_1 \dots i_N)}} + \dim F^{(i_N)}_{m^{(i_1 \dots i_N)}} +1
\eqn
for arbitrary $\sigma_{i_j}$. On the other hand, one has
\begin{align*}
d^{(i_j)}&=\dim \g_{x^{(i_j)}}^\perp =\dim [\lambda( \g_{x^{(i_j)}}^\perp) \cdot x^{(i_j)} ]=\dim [ \lambda( \g_{x^{(i_j)}}^\perp)  \cdot m^{(i_j\dots i_N)} ]= \dim E^{(i_j)}_{m^{(i_1 \dots i_N)}}.
\end{align*}
For $j=1,\dots,N-1$, the assertion of the lemma now follows with \eqref{eq:kappa}. Since 
\bqn
c^{(i_N)} = \dim (\nu_{i_1\dots i_N})_{x^{(i_N)}} \geq \dim G_{x^{(i_N)}} \cdot \tilde v^{(i_N)} +1,
\eqn
 a similar argument  yields the assertion for $j=N$. 
\end{proof}
As a consequence of the lemma, we obtain for $^2I_{i_1\dots i_N}^{\rho_{i_1}\dots \rho_{i_N}}(\mu)$ the estimate
\begin{align}
\label{eq:I2}
\begin{split}
^2I_{i_1\dots i_N}^{\rho_{i_1}\dots \rho_{i_N}}(\mu)&\leq C \int_{(-\epsilon,\epsilon)^N}  \prod_{j=1}^N |\tau_{i_j}|^{c^{(i_j)} + \sum _{r=1}^j d^{(i_r)} -1} \d \tau_{i_N} \dots \d \tau_{i_1} \\
&\leq C \int_{(-\epsilon,\epsilon)^N}  \prod_{j=1}^N |\tau_{i_j}|^{\kappa} \d \tau_{i_N} \dots \d \tau_{i_1} =\frac{2C}{\kappa+1} \epsilon^{N (\kappa+1)}
\end{split}
\end{align}
for some $C>0$. Let us now turn to the integral $^1I_{i_1\dots i_N}^{\rho_{i_1}\dots \rho_{i_N}}(\mu)$. After performing the change of variables $\delta_{i_1\dots i_N}$
one obtains
\begin{align*}
^1I_{i_1\dots i_N}^{\rho_{i_1}\dots \rho_{i_N}}(\mu)&  = \int\limits _{\epsilon < |\tau_{i_j} (\sigma)|< 1} \hspace{-.5cm}  J_{i_1\dots i_{N}}^{\rho_{i_1} \dots \rho_{i_{N}}}\Big ( \frac \mu{\tau_{i_1}(\sigma)\cdots \tau_{i_N}(\sigma)} \Big )  \prod_{j=1}^N |\tau_{i_j}(\sigma)|^{c^{(i_j)} + \sum _{r=1}^j d^{(i_r)} -1} \, |\det D\delta_{i_1\dots i_N}(\sigma) | \d \sigma,
\end{align*}
where $ J_{i_1\dots i_{N}}^{\rho_{i_1} \dots \rho_{i_{N}}}(\nu)$ is defined like $\hat J_{i_1\dots i_{N}}^{\rho_{i_1} \dots \rho_{i_{N}}}(\nu)$, but with $^{(i_1\dots i_N)} \tilde \psi ^{wk,pre}$ being replaced by $ ^{(i_1\dots i_N)} \tilde \psi ^{wk}_\sigma $, which denotes the weak transform of the phase function $\psi$ as a function of the variables $ x^{(i_j)}$, $\tilde v^{(i_N)}, \alpha^{(i_j)}, \beta^{(i_N)}, p$ alone, while the variables $\sigma=(\sigma_{i_1},\dots \sigma_{i_N})$ are regarded as parameters. The idea is now to make use of the principle of the stationary phase to give an asymptotic expansion of $J_{i_1\dots i_{N}}^{\rho_{i_1} \dots \rho_{i_{N}}}(\nu)$. 

\begin{theorem}
\label{thm:J}
Let $\sigma=(\sigma_{i_1},\dots, \sigma_{i_N})$ be a fixed set of parameters. Then, for every $\tilde N \in \N$ there exists a constant $C_{\tilde N, ^{(i_1\dots i_N)} \tilde \psi ^{wk}_\sigma}>0$ such that 
\bqn
|J_{i_1\dots i_{N}}^{\rho_{i_1} \dots \rho_{i_{N}}} (\nu) -(2\pi |\nu|)^\kappa\sum_{j=0} ^{\tilde N-1} |\nu|^j Q_j (^{(i_1\dots i_N)} \tilde \psi ^{wk}_\sigma;a_{i_1\dots i_N} \Phi_{i_1\dots i_N})| \leq C_{\tilde N,^{(i_1\dots i_N)} \tilde \psi ^{wk}_\sigma} |\nu|^{\tilde N},
\eqn
with explicit expressions and estimates for the coefficients $Q_j$. Moreover,  the constants $C_{\tilde N, ^{(i_1\dots i_N)} \tilde \psi ^{wk}_\sigma}$ and the coefficients $Q_j$ have uniform bounds in $\sigma$.
\end{theorem}
\begin{proof}
As a consequence of  Theorems \ref{thm:I} and \ref{thm:II}, together with Lemma \ref{lemma:A}, the phase function $^{(i_1\dots i_N)} \tilde \psi ^{wk}_\sigma$ has a clean critical set, meaning that 
\begin{itemize}
\item the critical set $\Crit( ^{(i_1\dots i_N)} \tilde \psi ^{wk}_\sigma )$ is a $\Cinft$-submanifold of codimension $2\kappa$ for arbitrary $\sigma$;
\item  the transversal Hessian
\bqn
\mathrm{Hess}  \,^{(i_1\dots i_N)} \tilde \psi ^{wk}_\sigma ( x^{(i_j)},  \tilde v^{(i_N)}, \alpha^{(i_j)}, \beta^{(i_N)},p)_{|N_{  (x^{(i_j)},  \tilde v^{(i_N)}, \alpha^{(i_j)}, \beta^{(i_N)},p)}\Crit \big ( ^{(i_1\dots i_N)} \tilde \psi ^{wk}_\sigma \big)}
\eqn 
defines a non-degenerate symmetric bilinear form for arbitrary $\sigma$ at every point  of the critical set of $^{(i_1\dots i_N)} \tilde \psi ^{wk}_\sigma$.
\end{itemize}
Thus, the necessary conditions for applying the principle of the stationary phase to the integral $J_{\sigma_{i_1},\dots,\sigma_{i_N}}(\nu)$ are fulfilled, and we obtain the desired asymptotic expansion by  Theorem \ref{thm:SPM}. To see  the existence of the uniform bounds, note that as an examination of the proof of Theorem \ref{thm:SP} shows,  the constants $C_{N,\psi}$ in Theorem \ref{thm:SPM} are  bounded from above by 
\bqn 
\sup_{m \in \Ccal \cap \supp a} \norm {\Big ( \psi''(m)_{|N_m\Ccal}\Big ) ^{-1}}
\eqn
see also \cite[Remark 1]{ramacher10}. We therefore have
\bqn 
C_{\tilde N, ^{(i_1\dots i_N)} \tilde \psi ^{wk}_\sigma} \leq C'_{\tilde N} \sup_{ x^{(i_j)},  \tilde v^{(i_N)}, \alpha^{(i_j)}, \beta^{(i_N)},p} \norm{\Big ({\mathrm{Hess} \,  ^{(i_1\dots i_N)} \tilde \psi ^{wk}_\sigma}_{|N \Crit ( ^{(i_1\dots i_N)} \tilde \psi ^{wk}_\sigma)}\Big ) ^{-1}}.
\eqn
But since by Lemma \ref{lemma:A} the transversal Hessian 
\bqn 
{\mathrm{Hess} \,  ^{(i_1\dots i_N)} \tilde \psi ^{wk}_\sigma}_{|N_{ (x^{(i_j)},  \tilde v^{(i_N)}, \alpha^{(i_j)}, \beta^{(i_N)},p ) } \Crit ( ^{(i_1\dots i_N)} \tilde \psi ^{wk}_\sigma)}
\eqn 
is given by 
\bqn 
{\mathrm{Hess} \,  ^{(i_1\dots i_N)} \tilde \psi ^{wk}}_{|N_{ (\sigma_{i_j}, x^{(i_j)},  \tilde v^{(i_N)}, \alpha^{(i_j)}, \beta^{(i_N)},p ) } \Crit ( ^{(i_1\dots i_N)} \tilde \psi ^{wk})},
\eqn 
we finally obtain the estimate
\bqn 
C_{\tilde N, ^{(i_1\dots i_N)} \tilde \psi ^{wk}_\sigma} \leq C'_{\tilde N} \sup_{\sigma_{i_j},  x^{(i_j)},  \tilde v^{(i_N)}, \alpha^{(i_j)}, \beta^{(i_N)},p} \norm{\Big ({\mathrm{Hess} \,  ^{(i_1\dots i_N)} \tilde \psi ^{wk}}_{|N \Crit ( ^{(i_1\dots i_N)} \tilde \psi ^{wk})}\Big ) ^{-1}} \leq C_{\tilde N, {i_1\dots i_N}}
\eqn
by a constant independent of $\sigma$. Similarly, one can show the existence of bounds of the form
\bqn 
|Q_j(^{(i_1\dots i_N)} \tilde \psi ^{wk}_\sigma; a_{i_1\dots i_N} \Phi_{i_1\dots i_N}) |\leq \tilde C_{j, {i_1\dots i_N}},
\eqn
with constants $\tilde C_{j, {i_1\dots i_N}}$ independent of $\sigma$. 
\end{proof}

\begin{remark}
Before going on, let us remark that for the computation of the integrals $^1I_{i_1\dots i_N}^{\rho_{i_1}\dots \rho_{i_N}}(\mu)$ it is only necessary to have an asymptotic expansion for the integrals $J_{i_1\dots i_{N}}^{\rho_{i_1} \dots \rho_{i_{N}}} (\nu)$ in the case that $\sigma_{i_1} \cdots \sigma_{i_N}\not=0$, which can also be  obtained without Theorems \ref{thm:I} and \ref{thm:II} using only the factorization of the phase function $\psi$ given by the resolution process, together with Lemma \ref{lemma:Reg}. Nevertheless, the main consequence to be drawn from Theorems \ref{thm:I} and \ref{thm:II} is that the constants $C_{\tilde N, ^{(i_1\dots i_N)} \tilde \psi ^{wk}_\sigma}$ and the coefficients $Q_j$ in Theorem \ref{thm:J} have uniform bounds in $\sigma$. 
\end{remark}
As a consequence of Theorem \ref{thm:J}, we obtain for arbitrary $\tilde N\in \N$
\begin{align*}
|J_{i_1\dots i_{N}}^{\rho_{i_1} \dots \rho_{i_{N}}} (\nu) & -(2\pi |\nu|)^\kappa Q_0( ^{(i_1\dots i_N)} \tilde \psi ^{wk}_\sigma;a_{i_1\dots i_N} \Phi_{i_1\dots i_N})| \\&\leq 
\Big |J_{i_1\dots i_{N}}^{\rho_{i_1} \dots \rho_{i_{N}}}(\nu) -(2\pi |\nu|)^\kappa\sum_{l=0} ^{\tilde N-1} |\nu|^l Q_l (^{(i_1\dots i_N)} \tilde \psi ^{wk}_\sigma;a_{i_1\dots i_N} \Phi_{i_1\dots i_N})\Big | \\ & + (2\pi |\nu|)^\kappa \sum_{l=1} ^{\tilde N-1} |\nu|^l |Q_l (^{(i_1\dots i_N)} \tilde \psi ^{wk}_\sigma;a_{i_1\dots i_N} \Phi_{i_1\dots i_N})| \leq c_1 |\nu|^{\tilde N}+c_2 |\nu|^\kappa \sum_{l=1}^{\tilde N-1} |\nu|^l
\end{align*}
with constants $c_i>0$ independent  of both $\sigma$ and $\nu$.
From this we deduce
\begin{align*}
\Big |\, ^1I_{i_1\dots i_N}^{\rho_{i_1}\dots \rho_{i_N}}&(\mu)  - (2\pi \mu)^{\kappa}  \int _{\epsilon < |\tau_{i_j}(\sigma) |< 1}   Q_0 \prod_{j=1}^N |\tau_{i_j}(\sigma)|^{c^{(i_j)} + \sum _{r=1}^j d^{(i_r)} -1-\kappa}  |\det D\delta_{i_1\dots i_N}(\sigma) | \d \sigma \Big| \\&\leq c_3 \mu^{\tilde N}   \int_{\epsilon < |\tau_{i_j}(\sigma) |< 1}   \prod_{j=1}^N |\tau_{i_j}(\sigma)|^{c^{(i_j)} + \sum _{r=1}^j d^{(i_r)} -1-\tilde N} \, |\det D\delta_{i_1\dots i_N}(\sigma) | \d \sigma\\
&+ c_4 \mu^{\kappa} \sum_{l=1}^{\tilde N-1} \mu ^l  \int_{\epsilon < |\tau_{i_j}(\sigma) |< 1}   \prod_{j=1}^N |\tau_{i_j}(\sigma)|^{c^{(i_j)} + \sum _{r=1}^j d^{(i_r)} -1-\kappa -l} \, |\det D\delta_{i_1\dots i_N}(\sigma) | \d \sigma\\
& \leq c_5 \mu^{\tilde N}\prod_{j=1}^N (- \log \eps)^{i_j} \max \Big \{1, \prod _{j=1}^N  \epsilon^{c^{(i_j)} + \sum _{r=1}^j d^{(i_r)} -\tilde N} \Big \} \\ 
&+c_6 \sum_{l=1}^{\tilde N -1} \mu^{\kappa +l} \prod_{j=1}^N (- \log \eps)^{i_{lj}} \max \Big \{1, \prod _{j=1}^N  \epsilon^{c^{(i_j)} + \sum _{r=1}^j d^{(i_r)} -\kappa -l} \Big \},
\end{align*}
where the exponents $i_j$ and $i_{lj}$ can take the values $0$ or $1$. We now set
$
\epsilon=\mu^{1/N}
$.
Taking into account Lemma \ref{lemma:kappa}, one infers  that the right hand side of the last inequality can be estimated by 
\bqn
\mu^{k+1} (\log \mu)^N.
\eqn
so that for sufficiently large $\tilde N \in \N$ we finally obtain an asymptotic expansion for $I_{i_1\dots i_N}^{\rho_{i_1}\dots \rho_{i_N}}(\mu)$ by taking into account \eqref{eq:I2}, and the fact that 
\bqn 
(2\pi \mu)^\kappa \int_{0 < |\tau_{i_j} |< \mu^{1/N}}  Q_0  \prod_{j=1}^N |\tau_{i_j}|^{c^{(i_j)} + \sum _{r=1}^j d^{(i_r)} -1-\kappa} \d \tau_{i_N} \dots \d \tau_{i_1} = O(\mu^{\kappa+1}). 
\eqn
\begin{theorem}
\label{thm:5}
Let the assumptions of Theorem \ref{thm:I}  be fulfilled. Then 
\bqn 
I_{i_1\dots i_N}^{\rho_{i_1}\dots \rho_{i_N}}(\mu)=(2 \pi \mu)^\kappa L_{i_1\dots i_N}^{\rho_{i_1}\dots \rho_{i_N}}+ O(\mu^{\kappa+1}(\log \mu)^N),
\eqn
where the leading coefficient $L_{i_1\dots i_N}^{\rho_{i_1}\dots \rho_{i_N}}$ is given by 
\bq
\label{eq:L}
L_{i_1\dots i_N}^{\rho_{i_1}\dots \rho_{i_N}}=\int_{\Crit( ^{(i_1\dots i_N)} \tilde \psi^{wk})} \frac { a_{i_1\dots i_N}^{\rho_{i_1}\dots \rho_{i_N}} \Phi_{i_1\dots i_N}^{\rho_{i_1}\dots \rho_{i_N}} \, d\Crit( ^{(i_1\dots i_N)} \tilde \psi^{wk})} {|\mathrm{Hess} ( ^{(i_1\dots i_N)} \tilde \psi^{wk})_{N\Crit( ^{(i_1\dots i_N)} \tilde \psi^{wk})}|^{1/2}},
\eq
where $ d\Crit( ^{(i_1\dots i_N)} \tilde \psi^{wk})$ denotes the induced measure.
\end{theorem}\qed

\section{Statement of the main result}
\label{sec:8}

Let us now return to our departing point, that is, the asymptotic behavior of the integral  \eqref{int} in case that $\varsigma=0$ is a singular value of the momentum map. 
For this, we still have to examine the contributions to $I(\mu)$ coming from integrals of the form
\begin{gather*}
\begin{split}
\tilde I_{i_1\dots i_\Theta}^{\rho_{i_1}\dots \rho_{i_\Theta}}(\mu) = \qquad \qquad \qquad \qquad\qquad \qquad\qquad \qquad\qquad \qquad \\ \int_{M_{i_1}(H_{i_1})\times (-1,1)} \Big [ \int_{\gamma^{(i_1)}((S_{i_1})_{x^{(i_1)}})_{i_2}(H_{i_2})\times (-1,1)} \dots \Big [  \int_{\gamma^{(i_{\Theta-1})}((S_{i_1\dots i_{\Theta-1}})_{x^{(i_{\Theta-1})}})_{i_{\Theta}}(H_{i_{\Theta}})\times (-1,1)} \\
\Big [ \int_{\gamma^{(i_{\Theta})}((S_{i_1\dots i_{\Theta}})_{x^{(i_\Theta)}})\times \g_{x^{(i_{\Theta})}}\times \g_{x^{(i_{\Theta})}}^\perp \times \cdots \times \g_{x^{(i_{1})}}^\perp\times T^\ast _{m^{(i_1\dots i_\Theta)}} W_{i_1}} e^{i\frac {\tau_1 \dots \tau_\Theta}\mu \, ^{(i_1\dots i_\Theta)} \tilde \psi ^{wk}}  \, a_{i_1\dots i_\Theta}^{\rho_{i_1}\dots \rho_{i_\Theta}} \,   \tilde \Phi_{i_1\dots i_\Theta}^{\rho_{i_1}\dots \rho_{i_\Theta}} \\
 \d (T^\ast _{m^{(i_1\dots i_\Theta)}} W_{i_1})(\eta)  \d A^{(i_1)} \dots  \d A^{(i_\Theta)}  \d B^{(i_\Theta)} \d \tilde v^{(i_\Theta)} \Big ]  \d \tau_{i_\Theta} \d x^{(i_{\Theta})} \dots  \Big ] \d \tau_{i_2} \d x^{(i_{2})} \Big ]\d \tau_{i_1} \d x^{(i_{1})},
\end{split}
\end{gather*}
where $\{(H_{i_1}), \dots, (H_{i_\Theta})\}$ is an arbitrary totally ordered subset of non-principal isotropy types, while 
$a_{i_1\dots i_\Theta}^{\rho_{i_1}\dots \rho_{i_\Theta}}$ is a smooth  amplitude which is supposed to have compact support in a system of $(\theta^{(i_1)}, \dots, \theta^{(i_{N-1})},\alpha^{(i_N)})$-charts labeled by the indices ${\rho_{i_1}\dots \rho_{i_\Theta}}$, and
\begin{align*}
\tilde \Phi_{i_1\dots i_\Theta} ^{\rho_{i_1}\dots \rho_{i_\Theta}}&=\prod_{j=1}^\Theta |\tau_{i_j}|^{c^{(i_j)}+\sum_r^j d^{(i_r)}-1}\Phi_{i_1\dots i_\Theta}^{\rho_{i_1}\dots \rho_{i_\Theta}},
\end{align*}
 $\Phi_{i_1\dots i_\Theta}$ being a smooth function which does not depend on the variables $\tau_{i_j}$. Now, a computation of the $p$-derivatives of $\, ^{(i_1\dots i_\Theta)} \tilde \psi ^{wk}$ in any of the $\alpha^{(i_\Theta)}$-charts shows that $\, ^{(i_1\dots i_\Theta)} \tilde \psi ^{wk}$ has no critical points  there. 
 By the non-stationary phase theorem, see H\"{o}rmander \cite[Theorem 7.7.1]{hoermanderI},  one then computes for arbitrary $\tilde N \in \N$
 \begin{align*}
| \tilde I_{i_1 \dots i_\Theta}^{\rho_{i_1}\dots \rho_{i_\Theta}}(\mu)| \leq c_7 \mu^{\tilde N} \int_{\epsilon<|\tau_{i_j}| < 1} \prod_{j=1} ^\Theta |\tau_{i_j}|^{c^{(i_j)}+\sum_r^j d^{(i_r)}-1-\tilde N}d\tau + c_8 \epsilon^{\Theta (\kappa+1)} \leq c_9 \max\mklm{\mu^{\tilde N},\mu^{\kappa+1}},
 \end{align*}
 where we took $\epsilon=\mu^{1/\Theta}$. Choosing $\tilde N$ large enough, we conclude that
 \bqn
 | \tilde I_{i_1 \dots i_\Theta}^{\rho_{i_1}\dots \rho_{i_\Theta}}(\mu)| =O(\mu^{\kappa+1}).
 \eqn
  As a consequence of this we see that, up to terms of order $O(\mu^{\kappa+1})$,  $I(\mu)$ can be written as a sum
\begin{align}
\label{eq:65}
I(\mu)&=\sum _{N=1}^{\Lambda-1}\, \sum_{\stackrel{i_1<\dots< i_N}{ \rho_{i_1}, \dots ,\rho_{i_N}}} I_{i_1\dots i_N}^{\rho_{i_1} \dots \rho_{i_{N}}}( \mu)+\sum_{N=1}^{\Lambda-1} \, \sum_{\stackrel{i_1<\dots< i_{N-1}<L}{ \rho_{i_1}, \dots ,\rho_{i_{N-1}}}} I_{i_1\dots i_{N-1} L}^{\rho_{i_1} \dots \rho_{i_{N-1}}}( \mu),
\end{align}
where the first term is a sum over  maximal, totally ordered subsets of non-principal isotropy types, while the second term is a sum over  totally ordered subsets of non-principal isotropy types. The asymptotic behavior of the integrals $I_{i_1\dots i_N}^{\rho_{i_1}\dots \rho_{i_N}}(\mu)$ has been determined in the previous section, and using Lemma \ref{lemma:Reg} it is not difficult to see that the integrals $I_{i_1\dots i_{N-1} L}^{\rho_{i_1} \dots \rho_{i_{N-1}}}( \mu)$ have analogous asymptotic descriptions.  We can now state the main result of this paper.

\begin{theorem}
\label{thm:main}
Let $M$ be a connected  Riemannian manifold, and $G$  a compact, connected Lie group $G$ with Lie algebra $\g$ acting isometrically and effectively on $M$. Consider the oscillatory integral
\bqn
I(\mu)=   \int _{T^\ast M}  \int_{\g} e^{i \psi( \eta,X)/\mu }   a(\eta,X)  \d X \d\eta ,   \qquad \mu >0,
\eqn 
where the phase function 
\bqn 
\psi(\eta,X)=\J(\eta)(X)
\eqn
is given by the momentum map $\J:T^\ast M \rightarrow \g^\ast$ corresponding to  the Hamiltonian action on $T^\ast M$,  $d\eta$ is the Liouville measure on $T^\ast M$, and $dX$ an Euclidean measure given by an $\Ad(G)$-invariant inner product on $\g$, while   $a \in \CT( T^\ast M \times \g)$.   Then $I(\mu)$ has the asymptotic expansion 
\bqn 
I(\mu) = (2\pi\mu)^\kappa L_0 + O(\mu^{\kappa+1}(\log \mu)^{\Lambda-1}),  \qquad \mu \to 0^+.
\eqn
Here $\kappa$ is the dimension of an orbit of principal type in $M$, $\Lambda$ the maximal number of elements of a totally ordered subset of the set of isotropy types, and the leading coefficient is given by \footnote{A more explicit expression for $L_0$ will be given in Proposition \ref{prop:LT}.}
\bq
\label{eq:L0}
L_0=\int_{\mathrm{Reg}\, \Ccal} \frac {a(\eta,X)}{|\mathrm{Hess}  \, \psi(\eta,X)_{|N_{(\eta,X)}\mathrm{Reg}\, \Ccal}|^{1/2}} \d(\mathrm{Reg}\, \Ccal)(\eta,X),
\eq
where $\mathrm{Reg}\, \Ccal$ denotes the regular part of the critical set $\Ccal=\Crit(\psi)$ of $\psi$, and $\d(\mathrm{Reg}\, \Ccal)$ the  measure  induced by $d\eta \d X.$ 
In particular, the integral over $\mathrm{Reg}\, \Ccal$ exists. 
\end{theorem}
\begin{remark}
Note that equation \eqref{eq:L0} in particular means that the obtained asymptotic expansion for $I(\mu)$ is independent of the explicit partial resolution we used.
\end{remark}
\begin{proof}
By \eqref{eq:65} and Theorem \ref{thm:5} one has 
\bqn 
I(\mu) = (2\pi\mu)^\kappa L_0 + O(\mu^{\kappa+1}(\log \mu)^{\Lambda-1}),  \qquad \mu \to 0^+,
\eqn
where $L_0$ is given by a sum of integrals of the form \eqref{eq:L}. It therefore remains to show the equality \eqref{eq:L0}. For this, we shall  introduce  certain cut-off functions for the singular part $\Sing \Omega$ of $\Omega$. Choose a Riemmanian metric on  $T^\ast M$, and denote the corresponding distance on $T^\ast M$  by $d$.   Let $K$ be a compact subset in $T^\ast M$, $\delta >0$, and consider the set
 \bqn
 (\Sing \Omega \cap K)_{\delta}=\mklm{\eta \in T^\ast M : d(\eta,\eta') < \delta \text{ for some } \eta' \in \Sing \Omega\cap K}.
 \eqn
By using a partition of unity, one can show the existence of a  test function $u_\delta \in \CT((\Sing \Omega \cap K)_{3\delta})$ satisfying $u_\delta =1$ on $(\Sing \Omega \cap K)_\delta$,  see H\"{o}rmander \cite[Theorem 1.4.1]{hoermanderI}. Now, let $K$ be such that  $\supp_\eta a \subset K$. We then assert that  the  limit 
\bq
\label{eq:MC}
\lim_{\delta \to 0}  \int_{\Reg \Ccal}\frac{ [a (1-u_\delta)] (\eta,X)}{|\det  \, \psi'' (\eta,X)_{|N_{(\eta,X)}\Reg \Ccal} |^{1/2}} d(\Reg \Ccal)(\eta,X)
\eq
exists and is equal to $L_0$, where $d(\Reg {\Ccal})$ is the measure on $\Reg {\Ccal}$ induced by $d\eta \, dX$.
Indeed,  define
\bqn
I_\delta(\mu)=\int_{T^\ast M} \int_{\g}  e^{\frac{i}{\mu} \psi(\eta,X) }  [a(1-u_\delta)](\eta,X) \, dX \, d\eta.
\eqn
Since $(\eta,X) \in \Sing {\Ccal}$ implies $ \eta \in \Sing \Omega$, a direct application of Theorem \ref{thm:SPM} for fixed $\delta >0$ gives
\bq
\label{eq:asympt}
| I_\delta(\mu)- (2\pi \mu)^\kappa L_0(\delta) | \leq C_\delta \mu^{\kappa+1},
\eq
where $C_\delta>0$ is a constant depending only on $\delta$, and
\bqn
L_0(\delta)=  \int_{\Reg {\Ccal}}\frac{ [a(1-u_\delta)] (\eta,X)}{|\det  \, \psi'' (\eta,X)_{|N_{(\eta,X)}\Reg {\Ccal}} |^{1/2}} d(\Reg {\Ccal})(\eta,X).
\eqn
On the other hand, applying our previous considerations to  $I_\delta(\mu)$ instead of $I(\mu)$, we obtain again an asymptotic expansion of the form \eqref{eq:asympt} for $I_\delta(\mu)$, where now the first coefficient is given by a sum of integrals of the form \eqref{eq:L} with $a$ replaced by $a(1-u_\delta)$. Since the first term in the asymptotic expansion \eqref{eq:asympt} is uniquely determined, the two expressions for $L_0(\delta)$ must be identical. The existence of the limit  \eqref{eq:MC} now follows by  the Lebesgue theorem on bounded convergence, the corresponding limit being given by $L_0$.  Let now $a ^+ \in \CT(T^\ast M\times \g, \R^+)$. Since one can assume that  $|u_\delta| \leq 1$, the lemma of Fatou implies that 
\bqn 
  \int_{\Reg \Ccal} \lim_{\delta \to 0}  \frac{ [a^+ (1-u_\delta)] (\eta,X)}{|\det  \, \psi'' (\eta,X)_{|N_{(\eta,X)}\Reg \Ccal} |^{1/2}} d(\Reg \Ccal)(\eta,X)
\eqn
is mayorized by the limit  \eqref{eq:MC}, with $a$ replaced by $a^+$, and we obtain
\bqn 
  \int_{\Reg \Ccal} \frac{ a^+ (\eta,X)}{|\det  \, \psi'' (\eta,X)_{|N_{(\eta,X)}\Reg \Ccal} |^{1/2}} |d(\Reg \Ccal)(\eta,X)| < \infty.
\eqn
Choosing  $a^+$ to be equal  $1$ on a neighborhood of the support of $a$, and applying the theorem of Lebesgue on bounded convergence to the limit \eqref{eq:MC}, we obtain equation \eqref{eq:L0}. 
\end{proof}

In what follows, we  shall compute the leading term \eqref{eq:L0} in a more explicit way, and begin by computing the determinant of the  transversal Hessian of the phase function  $\psi(\eta,X)$, the notation being as in Theorem \ref{thm:main}.

\begin{lemma}
\label{lem:TransvHess}
Let $(\eta,X) \in \Reg \Ccal$ be fixed. Then 
\bqn 
\det \mathrm{Hess}  \, \psi(\eta,X)_{|N_{(\eta,X)}\mathrm{Reg}\, \Ccal}=\det (\Xi - L_X \circ L_X)_{|\g \cdot \eta},
\eqn
where $L_X:\g \cdot \eta \rightarrow \g \cdot \eta$ denotes the linear mapping  \eqref{eq:LieX} given by the Lie derivative, and  $\Xi$ the linear transformation on $\g \cdot \eta$  defined in \eqref{eq:Xi}.
\end{lemma}
\begin{proof}
Let $(\eta,X)\in \Reg \Ccal$ be fixed and $\mklm{A_1,\dots,A_d}$ an orthonormal basis of $\g$ such that $\mklm{A_1,\dots,A_\kappa}$ is a basis of $\g_\eta^\perp$ and $\mklm{A_{\kappa+1},\dots,A_d}$ a basis of $\g \cdot \eta$. With respect to the basis
\bqn 
((\widetilde \X_i)_\eta; 0), \qquad (0; e_j), \qquad i=1,\dots, 2n, \quad j=1,\dots,d,
\eqn
 of $T_{(\eta,X)}(T^\ast M \times \g) = T_\eta(T^\ast M) \times \R^d$ introduced in the proof of Proposition \ref{prop:regasymp}, the Hessian  
\bqn 
\mathrm{Hess} \, \psi: T_{(\eta,X)}(T^\ast M\times \g) \times T_{(\eta,X)}(T^\ast M\times \g) \rightarrow \C, \qquad (v_1,v_2) \mapsto \tilde v_1(\tilde v_2(\psi))(\eta,X)
\eqn
 is given by the matrix
 \bqn 
 \mathcal{A}=\left ( \begin{matrix}  \omega_\eta([\widetilde X, \widetilde \X_i],\widetilde \X_j) &-  \omega_\eta(\widetilde A_j, \widetilde \X_i) \\  - \omega_\eta(\widetilde A_i,\widetilde \X_j) & 0 \end{matrix} \right ).
 \eqn
Indeed,   $\widetilde \X_i(J_X)=dJ_X(\widetilde \X_i)=-\iota_{\widetilde X} \omega(\widetilde \X_i)$, and   by \eqref{eq:40cis} we have 
$(\tilde \X_i)_\eta(\omega(\widetilde X,\widetilde \X_j))=-\omega_\eta([\widetilde X, \widetilde \X_i], \widetilde \X_j)$, since $\widetilde X_\eta=0$.  If therefore $\mathcal{J}:T(T^\ast M) \rightarrow T(T^\ast M)$ denotes the bundle homomorphism introduced in Section \ref{sec:2}, we obtain
 \bqn 
 \mathcal{A}=\left ( \begin{matrix}  \mathcal{J} L_X & -g_\eta(\mathcal{J} \widetilde A_j, \widetilde \X_i) \\  - g_\eta(\mathcal{J} \widetilde A_i,\widetilde \X_j) & 0 \end{matrix} \right ),
 \eqn
  where $L_X:T_\eta(T^\ast M) \rightarrow T_\eta(T^\ast M), \X \mapsto [\widetilde X, \widetilde \X]_\eta$ denotes the linear transformation induced by the Lie derivative, and restricts to a map on $\g \cdot \eta$ by Remark \ref{rem:2}. Let $\mklm{B_1, \dots, B_\kappa}$ be another basis of $\g_\eta^\perp$ such that $\{(\widetilde B_1)_\eta, \dots, (\widetilde B_\kappa)_\eta\}$ is an orthonormal basis of $\g \cdot \eta$, and recall that by \eqref{eq:7} we have $T_\eta \Reg \, \Omega =(\g \cdot \eta)^\omega$. Taking into account \eqref{eq:10} and $\g\cdot \eta \subset (\g \cdot \eta)^\omega$ one sees that 
\bqn
\mathcal{B}_k=(\mathcal{J} (\widetilde B_k)_\eta;0), \qquad \mathcal{B}'_k=( L_X (\widetilde B_k)_\eta; g_\eta(\widetilde A_1,\widetilde B_k), \dots, g_\eta(\widetilde A_\kappa,\widetilde B_k),0, \dots, 0), \qquad k=1,\dots, \kappa,
\eqn
constitutes a basis of $N_{(\eta,X)} \Reg \Ccal$  with $\eklm{\mathcal{B}_k,\mathcal{B}_l}=\delta_{kl}$, $\mathcal{B}_k\perp \mathcal{B}_l'$,  and $\eklm{\mathcal{B}_k',\mathcal{B}_l'}=(\Xi + L_XL_X)_{kl}$, where $\Xi$ was defined in \eqref{eq:Xi}. One now computes
\begin{align*}
\mathcal{A}(\mathcal{B}_k)=&\Big (\mathcal{J} L_X \mathcal{J} (\widetilde B_k)_\eta; -\sum_{j=1}^{2n} g_\eta(\mathcal{J} \widetilde A_1, \widetilde \X_j) g_\eta( \mathcal{J} \widetilde B_k, \widetilde \X_j), \dots\Big )\\=&(-L_X  (\widetilde B_k)_\eta; -g_\eta(\mathcal{J} \widetilde A_1,  \mathcal{J} \widetilde B_k), \dots,-g_\eta(\mathcal{J} \widetilde A_\kappa,  \mathcal{J} \widetilde B_k),0, \dots, 0 )=-\mathcal{B}_k',\\
\mathcal{A}(\mathcal{B}_k')=&\Big (\mathcal{J} L_X L_X(\widetilde B_k)_\eta- \Big (\sum\limits_{j=1}^{\kappa} g_\eta(\mathcal{J} \widetilde A_j, \widetilde \X_1) g_\eta(  \widetilde A_j, \widetilde B_k), \dots \Big );  \\
-&\sum_{j=1}^{2n} g_\eta(\mathcal{J} \widetilde A_1, \widetilde \X_j) g_\eta(L_X  (\widetilde B_k)_\eta, \widetilde \X_j), \dots\Big )=(\mathcal{J} L_X L_X(\widetilde B_k)_\eta+(g_\eta(\Xi (\widetilde B_k)_\eta, \mathcal{J} \widetilde \X_1), \dots);\\
&-g_\eta(\mathcal{J} \widetilde A_1, L_X(\widetilde B_k)_\eta), \dots).
\end{align*}
Since $L_X$ defines an endomorphism of $\g\cdot \eta$ and $\g \cdot \eta \subset  (\g \cdot \eta)^\omega$ we have
 $g_\eta(\mathcal{J} \widetilde A_1, L_X(\widetilde B_k)_\eta)=\omega_\eta( \widetilde A_1, L_X(\widetilde B_k)_\eta)=0$. Furthermore, the $\{\mathcal{J} (\widetilde B_1)_\eta,\dots, \mathcal{J} (\widetilde B_\kappa)_\eta\}$  form an orthonormal basis of $\mathcal{J}(\g \cdot \eta)$, and we obtain
 \bqn 
 \mathcal{A}(\mathcal{B}_k')=(\mathcal{J} (L_X L_X-\Xi)(\widetilde B_k)_\eta;0)=\sum_{j=1}^\kappa g_\eta(\mathcal{J} (L_X L_X-\Xi)(\widetilde B_k)_\eta, \mathcal{J} (\widetilde B_j)_\eta) \, \mathcal{B}_j.
 \eqn
 Taking all together, one sees that the transversal Hessian $\mathrm{Hess}  \, \psi(\eta,X)_{|N_{(\eta,X)}\mathrm{Reg}\, \Ccal}$ is given by the matrix
\bqn 
 \left ( \begin{matrix}  0 &-\1_\kappa \\ (L_XL_X-\Xi)_{|\g\cdot \eta}  & 0 \end{matrix} \right ),
\eqn 
and the assertion follows. 
  \end{proof}

\begin{proposition} 
\label{prop:LT}  The leading term in \eqref{eq:L0} is given by 
\begin{align*}
L_0 =\frac{\vol G}{\vol H} \int_{\mathrm{Reg}\, \Omega} \left [\int_{\g_\eta} {a(\eta,X)\d X } \right ] \frac{\d(\mathrm{Reg}\, \Omega)(\eta)}{\vol \, \mathcal{O}_\eta},
\end{align*}
where $H$ denotes a principal isotropy group, and $\vol \, \mathcal{O}_\eta$  the volume of the $G$-orbit through $\eta$, while $dX$ is the measure on $\g_\eta$ induced by the invariant inner product on $\g$.
\end{proposition}
\begin{proof} 
 The proof is based on  the following integration formula, compare  \cite[Lemma 3.4]{cassanas}. Let $({\bf X},h_{\bf X})$ and $({\bf Y},h_{\bf Y})$ be two Riemannian manifolds and $F:{\bf X \rightarrow Y}$ a smooth submersion. Then, for $b \in \CT({\bf X})$ one has 
\bq
\label{eq:intsub}
\int_{\bf X} b(x) \d{\bf X}(x)=\int_{\bf Y} \left [ \int_{F^{-1}(y)} b(z) \frac{d(F^{-1}(y))(z)}{|\det d_z F \circ  \, ^td_z F)|^{1/2}} \right ] \d{\bf Y}(y),
\eq
where $d(F^{-1}(y))$ denotes the Riemannian measure induced by the one of ${\bf X}$ on $F^{-1}(y)$, and  the transposed operator of the differential $d_xF:T_{x} {\bf X} \rightarrow T_{F(x)} {\bf Y}$ is given by the operator $^t d_xF:T_{F(x)} {\bf Y} \rightarrow T_{x} {\bf X}$ which is uniquely determined  by the condition
\bqn 
h_{\bf X}(\X, \, ^t d_xF(\mathfrak{Y}))=h_{\bf Y}(d_xF(\X), \mathfrak{Y}), \qquad \X \in T_x{\bf X}, \quad \mathfrak{Y} \in T_{F(x)}{\bf Y}.
\eqn
 Consider now the map $P:\Reg \, \Ccal \rightarrow \Reg \, \Omega, (\eta,X) \rightarrow \eta$, which is a submersion by Proposition \ref{prop:submersion}. In order to apply the previous integration formula, we have to compute the determinant of $d_{(\eta,X)} P \circ  \, ^t d_{(\eta,X)} P$ at a point $(\eta,X) \in \Reg \, \Ccal$. For this, let $\mathcal{G}$ denote the orthogonal complement of $\g \cdot \eta$ in $T_\eta \Reg \, \Omega$. We then assert that 
 \bq
 \label{eq:49}
 d_{(\eta,X)} P \circ  \, ^t d_{(\eta,X)} P_{|\mathcal{G}}=\id. 
 \eq
 Indeed, let $\mathfrak{Y} \in \mathcal{G}$. As was shown in the proof of Proposition \ref{prop:submersion}, $[\widetilde {\mathfrak{Y}}, \widetilde X]_\eta \in \g \cdot \eta$. On the other hand, the fact that  $\g \cdot \eta$ and $\mathcal{G}$ are invariant under $G_\eta$, together with  \eqref{eq:11}, imply that $[\widetilde {\mathfrak{Y}}, \widetilde X]_\eta \in \mathcal{G}$. Hence $[\widetilde {\mathfrak{Y}}, \widetilde X]_\eta =0$. Taking into account \eqref{eq:10} we  infer from this that  $(\mathfrak{Y},0) \in T_{(\eta,X)} \Reg \Ccal$, and consequently $^t dP_{(\eta,X)}(\mathfrak{Y})= (\mathfrak{Y},0)$. Thus, $d_{(\eta,X)} P \circ  \, ^t d_{(\eta,X)} P (\mathfrak{Y})=\mathfrak{Y}$, and \eqref{eq:49} follows. For the computation of the determinant of $d_{(\eta,X)} P \circ  \, ^t d_{(\eta,X)} P$ it therefore suffices to consider its restriction to $\g \cdot \eta$, and with  the notation as in Lemma \ref{lem:TransvHess} we shall show that 
 \bq
 \label{eq:48}
  d_{(\eta,X)} P \circ  \, ^t d_{(\eta,X)} P_{|\g \cdot \eta}=(\Xi- L_X \circ L_X)^{-1} \circ \Xi.  
  \eq
 Consider thus an element $\X \in \g \cdot \eta$, and write $ ^t d_{(\eta,X)} P(\X)=(\mathfrak{Y},w)$. Denote the   $\Ad(G)$-invariant inner product in $\g$ by $\langle\cdot,\cdot\rangle$, and let again $\mklm{A_1,\dots, A_d}$ be an orthonormal basis of $\g$ such that $\g_\eta^\perp$ is spanned by the elements $\mklm{A_1,\dots, A_\kappa}$, and  $\g_\eta$ by  $\mklm{A_{\kappa+1}, \dots, A_d}$.  From \eqref{eq:10} it is clear that for each $j=1, \dots, \kappa$ we have $((\widetilde A_j)_\eta; \eklm{[X,A_j], A_1}, \dots, \eklm{[X,A_j], A_d}) \in T_{(\eta,X)} \Reg \, \Ccal$. By definition of the transposed we therefore have 
 \bqn 
 g(\X, (\widetilde A_j)_\eta)=g(\mathfrak{Y},(\widetilde A_j)_\eta) +\sum_{k=1}^d w_k  \eklm{[X,A_j], A_k}.
 \eqn
 Consequently, $  g(\X-\mathfrak{Y}, (\widetilde A_j)_\eta)=\sum_{k=1}^d w_k  \eklm{[X,A_j], A_k}$. 
 If $\Xi$ denotes the linear transformation introduced in \eqref{eq:Xi}, we obtain 
 \begin{align*}
 \Xi(\X-\mathfrak{Y})=\sum_{j=1}^\kappa \sum_{k=1}^d w_k  \eklm{[X,A_j], A_k} (\widetilde A_j)_\eta=\sum_{j=1}^d \sum_{k=1}^d w_k  \eklm{A_j, [A_k,X]} (\widetilde A_j)_\eta= \sum_{k=1}^d w_k  \widetilde{[A_k,X] _\eta}.
 \end{align*}
Let $f \in \Cinft(T^\ast M)$. Due to $\widetilde X_\eta =0$ we have $\widetilde{[A_k,X] _\eta} f= (\widetilde A_k)_\eta(\widetilde X f)$. Combined with the fact that $\sum_{k=1}^d w_k(\widetilde A_k)_\eta=-[\widetilde {\mathfrak{Y}},\widetilde X]_\eta$ this implies
\bqn 
-\sum_{k=1}^d w_k  \widetilde{[A_k,X] _\eta} f=[\widetilde {\mathfrak{Y}},\widetilde X]_\eta (\widetilde X f)=[[\widetilde {\mathfrak{Y}},\widetilde X],\widetilde X]_\eta f=[\widetilde X,[\widetilde X,\widetilde {\mathfrak{Y}}]]_\eta f,
\eqn
and consequently
\bqn 
\Xi(\mathfrak{Y}-\X)=[\widetilde X,[\widetilde X,\widetilde {\mathfrak{Y}}]]_\eta=L_X([\widetilde X, \widetilde{\mathfrak{Y}}]_\eta)=L_X \circ L_X(\mathfrak{Y}).
\eqn
Thus, $\mathfrak{Y}= (\Xi-L_X \circ L_X)^{-1} \circ \Xi(\X)$, and \eqref{eq:48} follows. Taking all together we have shown that 
$$\det d_{(\eta,X)} P \circ  \, ^t d_{(\eta,X)} P=\det^{-1}(\Xi- L_X \circ L_X) \cdot \det \Xi,$$
 and with Lemma \ref{lem:TransvHess} and the integration formula \eqref{eq:intsub} we obtain
\bqn 
L_0=\int_{\mathrm{Reg}\, \Ccal} \frac {a(\eta,X) \d(\mathrm{Reg}\, \Ccal)(\eta,X) }{|\mathrm{Hess}  \, \psi(\eta,X)_{|N_{(\eta,X)}\mathrm{Reg}\, \Ccal}|^{1/2}}  =\int_{\mathrm{Reg}\, \Omega} \left [\int_{\g_\eta} {a(\eta,X)\d X } \right ] \frac{\d(\mathrm{Reg}\, \Omega)(\eta)}{|\det \Xi_{|\g \cdot \eta}|^{1/2}},
\eqn
where $d(\mathrm{Reg}\, \Omega)$ denotes the volume form induced by $d\eta \, dX$. 
The assertion of the proposition now follows by noting that  $|\det \Xi_{|\g \cdot \eta}|^{1/2}= \vol \mathcal{O}_\eta \cdot \vol G_\eta/\vol G$, compare  \cite[Lemma 3.6]{cassanas}. 

\end{proof}

\section{Residue formulae for ${\bf X}=T^\ast M$}

We are now in position to derive  residue formulae for the cotangent bundle of a $G$-manifold. Thus,  let $M$ be an $n$-dimensional, connected Riemannian manifold, and $G$  a $d$-dimensional, compact, connected Lie group  with maximal torus $T \subset G$ acting on $M$ by isometries. Let $\Theta$ be the Liouville form on $T^\ast M$, $\omega=d\Theta$ the symplectic form,   denote the corresponding momentum map by $\J:T^\ast M \rightarrow \g^\ast$, $\J(\eta)(X)=J_X(\eta)=\Theta(\widetilde X)(\eta)$, and write  $\Omega =\J^{-1}(0)$. Let further 
 $\pi:\Reg \Omega \rightarrow \Reg {\bf X}_{red}=\Reg \Omega /G$ be the canonical  projection,  and consider the map
 \bqn 
 \widetilde {\mathcal{K}}: H^{\ast+\kappa}_G(T^\ast M) \stackrel{r}{\longrightarrow} H^\ast_G(\Reg \Omega ) \stackrel{(\pi^\ast)^{-1}}{\longrightarrow} H^\ast (\Reg {\bf X}_{red}),
 \eqn 
 where $r:\Lambda^\ast(T^\ast M) \rightarrow \Lambda^{\ast-\kappa}( \Reg \Omega)$ denotes the natural restriction map described in \eqref{eq:64} and $\kappa$ is the dimension of a principal $G$-orbit. 
As an application of Theorem \ref{thm:main}, we are able to compute the limit \eqref{eq:50} in case that  $\kappa$ equals $d=\dim \g$. It corresponds to the leading term in the expansion. 
 
\begin{corollary}
\label{cor:2}
Assume that the dimension $\kappa$ of a principal $G$-orbit in $M$ equals $d=\dim \g$. Let $\alpha \in \Lambda_c (T^\ast M)$ and $\phi \in \CT(\g^\ast)$ have total integral one.   Then
\bqn 
\lim_{\eps \to 0} \eklm{\F_{\g} L_\alpha,\phi_\eps}=    \frac{ (2\pi)^d\, \vol \, G}{|H|} \int_{\mathrm{Reg}\, \Omega}   a(\eta)\frac{\d(\mathrm{Reg}\, \Omega)(\eta)}{\vol \, \mathcal{O}_\eta}=  \frac{ (2\pi)^d\, \vol \, G}{|H|} \int_{\mathrm{Reg}\, \Omega}   \frac{r(\alpha)}{\vol \, \mathcal{O}_\eta},
\eqn
where  $H$ denotes a principal isotropy group of the $G$-action, and we wrote  $\alpha_{[2n]}= a(\eta) d\eta$, $d\eta$ being Liouville measure. 
\end{corollary}
\begin{proof} By \eqref{eq:50}, Theorem \ref{thm:main},  and Proposition \ref{prop:LT} one deduces
\bqn 
L_0=\lim_{\eps \to 0} \eklm{\F_{\g} L_\alpha,\phi_\eps}= \frac{(2\pi)^d \vol G}{\vol H}  \int_{\mathrm{Reg}\, \Omega} \left [\int_{\g_\eta} {\hat \phi( X)\d X } \right ]  a(\eta)\frac{\d(\mathrm{Reg}\, \Omega)(\eta)}{\vol \, \mathcal{O}_\eta}.
\eqn
Since $\kappa=\dim \g$, we have $\g_\eta=\mklm{0}$ for all $\eta \in \Reg \Omega$; in particular, $H\sim G_\eta$ is a finite group. Hence, $\vol H\equiv|H|$ and $\int_{\g_\eta} \hat \phi( X)\d X=\hat \phi(0)=1$, and we obtain the first equality. To see the second, assume that $\alpha$ is  supported in a neighborhood of $\Ccal$. Let $K\subset T^\ast M$ be a compact subset such that $\supp \alpha \subset K$, and  $u_\delta \in \CT( \Sing \Omega \cap K)_{3\delta}$ a family of cut-off functions as in the proof of Theorem \ref{thm:main}.  Denote the normal bundle to $\Reg \Ccal=\Reg \Omega \times \mklm{0}\equiv \Reg \Omega$ by $\nu:N\, \Reg \Ccal \rightarrow \Ccal$,  and identify a tubular neighborhood of $\Reg \Ccal$ with a neighborhood of the zero section in $N \, \Reg \Ccal$.  A direct application of Theorem \ref{thm:SP}  then yields  with Lemma \ref{lem:TransvHess}
\bqn 
L_0(\delta)=\lim_{\eps\to 0} \int_{\g} \int_{{\bf X}} e^{i J_X/\eps}  (1-u_\delta) \, \alpha  \,  \hat \phi(X) \frac {dX}{\eps^d} = \frac{ (2\pi)^d\, \vol G}{|H|}  \int_{\mathrm{Reg}\, \Omega}   \frac{r((1-u_\delta)\alpha)}{\vol \, \mathcal{O}_\eta},
\eqn
where  only the leading term  \eqref{eq:63} is relevant. Repeating the arguments in the proof of Theorem \ref{thm:main} then shows that
\bqn 
 L_0=\lim_{\delta \to 0} L_0(\delta)= \frac{ (2\pi)^d\, \vol G}{|H|}  \int_{\mathrm{Reg}\, \Omega}   \frac{r(\alpha)}{\vol \, \mathcal{O}_\eta}.
 \eqn
\end{proof}
With the notation as in Sections \ref{sec:2} and \ref{sec:4}, we  finally arrive at the following  
\begin{theorem}
\label{thm:res}
Let $\rho\in H^\ast_G(T^\ast M)$ be  of the form $\rho(X)=\alpha+D\nu(X)$, where $\alpha$ is a closed, basic differential form on $T^\ast M$ of compact support, and $\nu$ an equivariant differential form of compact support.  Assume that the dimension $\kappa$ of a principal $G$-orbit equals $d=\dim \g$. Then  
\bqn
(2\pi)^{d}  \int_{\Reg {\bf X}_{red}} \widetilde {\mathcal{K}}( e^{-i\omega}\alpha)=\frac{|H|}{|W| \,\vol \,T} \Res \Big ( \Phi^2\sum_{F \in \F} u_F \Big ).
\eqn
\end{theorem}
\begin{proof}
Let $\alpha$ be  a basic differential form on $T^\ast M$. By definition, $\alpha$ is  $G$-invariant and satisfies $\iota_{\widetilde X} \alpha=0$ for all $X \in \g$. It is therefore a constant map from $\g$ to $\Lambda(T^\ast M)$, and belongs to $(S(\g^\ast) \otimes \Lambda(T^\ast M))^G$. Furthermore, $D\alpha=0$ iff $d\alpha=0$, so that $\alpha \in H^\ast_G(T^\ast M)$. The assertion is now a consequence  of Corollaries  \ref{cor:A} and  \ref{cor:2}, together with Lemma \ref{prop:exact}, by which
\begin{align*}
\frac{\vol G}{|W| \, \vol T} \Res \Big ( \Phi^2\sum_{F \in \F} u_F \Big )&=\lim_{\eps \to 0} \eklm{\F_{\g} \Big (L_{e^{-i\omega}\rho(\cdot)}(\cdot )\Big )  ,\phi_\eps}=\frac{ (2\pi)^d\, \vol G}{|H|}  \int_{\mathrm{Reg}\, \Omega}   \frac{r(e^{-i\omega}\alpha)}{\vol \, \mathcal{O}_\eta}\\&=\frac{ (2\pi)^d\, \vol G}{|H|} \int_{\Reg {\bf X}_{red}}  \widetilde {\mathcal{K}}( e^{-i\omega}\alpha). 
\end{align*}
 \end{proof} 
  \begin{remark}
  In order to fully describe the cohomology of the quotient $\Reg {\bf X}_{red}$, it would still be necessary to consider more general forms $\rho \in H^\ast_G(T^\ast M)$ than the ones examined in Theorem \ref{thm:res}. For this, one would need  a full asymptotic expansion for the integrals studied in Theorem \ref{thm:main}, and we intend to tackle this problem in a future paper. Nevertheless, the considered forms $\rho$ are already quite general in the following sense.  Let $G$ act  \emph{locally freely} on a symplectic manifold $\bf X$, which means that all stabilizer groups are finite, and assume that the action is Hamiltonian. As a consequence, $0$ is a regular value of the momentum map and   ${\bf X}/G$ is an orbifold. Furthermore, one has the isomorphism
\bqn 
H^\ast_G({\bf X}) \simeq H^\ast({\bf X}/G),
\eqn
which implies that any equivariantly closed differential form $\rho$ can be written in the form 
$$\rho(X)=\alpha+D\nu(X),$$
 where $\alpha$ is a closed, basic differential form on $T^\ast M$ of compact support, and $\nu$ is an equivariant differential form of compact support \cite{duistermaat94}.  
\end{remark}

Let $\bf X$ be a $2n$-dimensional symplectic manifold with a Hamiltonian $G$-action. 
 For general, not necessarily equivariantly closed $\alpha \in \Lambda_c({\bf X})$, no similar formulae can be expected, and non-local remainder terms will occur. To see this, let us first deduce an expansion for  $L_\alpha(X)$ using  the stationary phase principle.    For this, recall that for fixed $X \in \g$ the critical set of $J_X$ is clean in the sense of Bott, and equal to $F^T$ in case that $X\in \t'$ is a regular element.

\begin{lemma}
\label{lemma:7}
Let $X\in \g$, and suppose that  $\supp \alpha \cap \Crit \, J_X=\emptyset$. Then $L_\alpha \in \S(\g)$.
\end{lemma}
\begin{proof}
Let $(\gamma, \mathcal{O})$ be a Darboux chart on $\bf X$,  so that the symplectic form $\omega$ and the corresponding Liouville form read 
\bqn 
\omega \equiv \sum_{i=1}^n dp_i \wedge dq_i, \qquad \frac{\omega^n}{n!} \equiv dp_1 \wedge dq_1 \wedge \dots \wedge dp_n \wedge dq_n. 
\eqn
Assume that  $\alpha_{[2n]}= f \cdot \frac{\omega^n}{n!} \in \Lambda _c({\bf X})$ is supported in $ \mathcal{O}$, so that 
\bqn 
\int_{\bf X} e^{iJ_X} \alpha = \int_{\gamma(\mathcal{O})} e^{iJ_X\circ \gamma^{-1}(q,p)} (f \circ \gamma^{-1}) (q,p) \d q \d p,
\eqn
where $J_X\circ \gamma^{-1}(q,p)$ depends linearly on $X$. Let now $\supp \alpha \cap \Crit \, J_X= \emptyset$. Writing 
\bqn 
 e^{iJ_X\circ \gamma^{-1}}=\frac 1 {i|(J_X\circ \gamma^{-1})'|^2}\sum_{j=1}^n \left ( \frac{\gd}{\gd q_j} (J_X\circ \gamma^{-1}) \frac \gd{\gd q_j}+ \frac{\gd}{\gd p_j} (J_X\circ \gamma^{-1}) \frac \gd{\gd p_j}\right )e^{iJ_X\circ \gamma^{-1}},
\eqn
and integrating by parts we obtain   $L_\alpha(X)=O(|X|^{-\infty})$ on $\g$. Similarly, if  $\mklm{X_1,\dots, X_d}$ denotes a basis of $\g$, and $X=\sum s_i X_i$, the same arguments yield for arbitrary multi-indices $\gamma$ the estimate
$\gd_s^\gamma L_\alpha(X)=O(|X|^{-\infty})$ on $\g$, and the assertion follows.
\end{proof}

Next, let $Y \in \t'$ be a regular element, so that $\Crit J_Y=F^T$, $F\in \F$  a connected component of $F^T$, and $\nu:NF\rightarrow F$ the normal bundle of $F$. As usual, we identify a neighborhood of the zero section of $NF$ with a tubular neighborhood of $F$, and assume in the following that the support of $\alpha $ is contained in that neighborhood.  
Integration along the fiber yields
\bqn 
L_\alpha(Y)= \int_F \nu_\ast(e^{iJ_Y} \alpha). 
\eqn
To obtain a  localization formula for $L_\alpha(Y)$ via the stationary phase principle,   consider   an oriented trivialization $\mklm{(U_j,\phi_j)}_{j \in I}$  of $\nu:NF\rightarrow F$. Let $\mklm{s_1,\dots,s_l}$  be the fiber coordinates on $NF_{|U_j}$ given by $\phi_j$, and  Assume that $\alpha$ is given on $\nu^{-1}(U_j)$ by 
\bqn 
\alpha_j =f_j (x,s) \,  (\nu^\ast \beta_j) \wedge ds_1\wedge \cdots \wedge ds_l, \qquad \beta_j \in \Lambda^{2n-l}(U_j), \quad x \in U_j, 
\eqn
where  $f_j$ is compactly supported.
The cleanness  of $\Crit J_Y$ implies that  the function $s\mapsto J_Y(x,s)= J_Y \circ \phi_j^{-1}(x,s)$ has a non-degenerate critical point at $s=0$ for each $x \in U_j$, so that by choosing the support of $f_j$  sufficiently small we can assume that there are no other critical points. Define now the function $H_Y(x,s)=J_Y(x,s) -\eklm{ J_Y''(x,0) s, s}/2$, which depends linearly on $Y$.  As in the proof of Theorem \ref{thm:SP} one computes for any   $N\in \N$ 
 \begin{gather*}
\nu_\ast(e^{iJ_Y} \alpha_j)= \frac{1}{\det(J_Y''(x,0)/2\pi i)^{1/2}} \\
\cdot 
 \left [ \sum_{r-k\leq  N} \sum_{3k\leq 2r} \frac{1}{r!k!} 
\left ( \eklm{D_s, \frac{J_Y''(x,0)^{-1}}{2i} D_s}^r  (iH_Y(x,\cdot ))^k f_j(x,\cdot ) \right )(x,0) + R_{j, N+1} (Y) \right ]  \cdot \beta_j,
\end{gather*}
where $R_{j, N+1}$ is an explicitly given smooth function on $\t'$ of order $O(|Y|^{- N-1})$  given by 
\begin{gather*}
R_{j, N+1}(Y)=\frac{\beta_j}{\det (J_Y''(x,0)/2\pi i)^{1/2}} \\ \cdot \sum_{k=0}^{\infty}  \int_{\R^l} 
\sum_{r=3 N+1}^\infty \frac 1 { (2\pi)^l k! r!}\left (\frac{\eklm{J_Y''(x,0)^{-1} \xi, \xi }} {2 i  }\right )^r \mathcal{F}\big ( H_Y(x,\cdot )^k f_j(x,\cdot ) \big )(\xi) \d \xi.
\end{gather*}
As a consequence, we obtain the desired localization formula.

\begin{proposition} 
Let $\alpha \in \Lambda_c(T^\ast M)$, and $Y\in \t'$. Then, for arbitrary $ N\in \N$,
\begin{gather*}
L_\alpha(Y)
=\sum_{F\in \mathcal{F}}  \sum_j \int_F  \frac{1}{\det(J_Y''(x,0)/2\pi i)^{1/2}}\\ 
\cdot \left [ \sum_{r-k\leq  N} \sum_{3k\leq 2r} \frac{1}{r!k!} 
\left ( \eklm{D_s, \frac{J_Y''(x,0)^{-1}}{2i} D_s}^r  (iH_Y(x,\cdot ))^k f_j(x,\cdot ) \right )(x,0)  \right ]  \cdot \beta_j + R_{N+1} (Y),
\end{gather*}
where  $R_{N+1}$ is an explicitely given, smooth function on $\t'$ of order $O(|Y|^{- N-1})$. 
\end{proposition}
\qed

\medskip

The limit \eqref{eq:43} can now be studied taking into account \eqref{eq:40} and  Cauchy's integral theorem, together with the theorems of Paley-Wiener-Schwartz, leading to corresponding residue formulae with  non-local terms.

\appendix

\section{The generalized stationary phase theorem}

In this appendix, we include  a proof of the generalized stationary phase theorem in the setting of vector bundles. It is a direct consequence of the projection formula and the stationary phase approximation, and implies the classical generalized stationary phase theorem for manifolds. Sketches of proofs for the latter can also be found in  Combescure-Ralston-Robert \cite[Theorem 3.3]{combescure-ralston-robert}, as well as Varadarajan \cite[pp. 199]{varadarajan97}.

\begin{theoremA}[Stationary phase theorem for vector bundles]
 Let $M$ be an $n$-dimensional, oriented manifold, and  $\pi:E \rightarrow M$ an oriented vector bundle of rang $l$.  Let further $\alpha \in \Lambda_{cv}^{q}(E)$ be a differential form on $E$ with compact support along the fibers, $\tau \in \Lambda^{n+l-q}_c(M)$ a differential form on $M$ of compact support,  $\psi \in \Cinft(E)$, and consider the integral
\bqn
I(\mu)=\int_E e^{i\psi/\mu} (\pi^\ast \tau) \wedge \alpha, \qquad \mu >0.
\eqn
Let $\iota:M \hookrightarrow E$ denote the zero section. Assume that the critical set of $\psi$ coincides with $\iota(M)$, and that the transversal Hessian $\mathrm{Hess}_{trans}\,  \psi$  of $\psi$ is non-degenerate along $\iota(M)$. Then,  for each $N \in \N$,  $I(\mu)$ possesses an asymptotic expansion of the form
\bq
\label{eq:8}
I(\mu) = e^{i\psi_0/\mu}e^{\frac{i\pi}4\sigma_{\psi}}(2\pi \mu)^{\frac {l}{2}}\sum_{j=0} ^{N-1} \mu^j Q_j (\psi;\alpha,\tau)+R_N(\mu),
\eq
where $\psi_0$ and $\sigma_{\psi}$ denote the value of $\psi$ and the signature of the transversal Hessian along $\iota(M)$, respectively. The coefficients $Q_j$ are given by measures supported on $M$,  and can be computed explicitly, as well as the remainder term $R_N(\mu)$ which is of order $O(\mu^{l/2+N})$. 
\end{theoremA}
\begin{proof}
Let $\pi_\ast:\Lambda^\ast_{cv}(E) \rightarrow \Lambda^{\ast-l}(M)$ denote integration along the fiber in $E$, which lowers the degree by the fiber dimension. By the projection formula \cite[Proposition 6.15]{bott-tu} one has
\bqn
\int_E e^{i\psi/\mu}  (\pi^\ast \tau) \wedge \alpha = \int_M \tau \wedge \pi_\ast(e^{i\psi/\mu} \alpha).
\eqn
Let $\mklm{U_j}_{j \in I}$ be an open covering of $M$ and  $\mklm{(U_j,\phi_j)}_{j \in I}$,  $\phi_j:\pi^{-1}(U_j) \rightarrow U_j \times \R^l$,   an oriented trivialization of $\pi:E\rightarrow M$. 
Write ${s_1,\dots,s_l}$ for the fiber coordinates on $E_{|U_j}$ given by $\phi_j$.  Since $I(\mu)$ vanishes if $q<l$, we  assume in the following   that $q \geq l$ and that $\alpha$ is given on $\pi^{-1}(U_j)$ by 
\bqn 
\alpha_j=f_j (x,s) \,  (\pi^\ast \beta_j) \wedge ds_1\wedge \cdots \wedge ds_l, \qquad \beta_j \in \Lambda^{q-l}(U_j),\quad x \in U_j,
\eqn
where the function $f_j\in \Cinft(U_j \times \R^l)$ is compactly supported along the fibers. By assumption,  $s\mapsto \psi(x,s)= \psi \circ \phi_j^{-1}(x,s)$ has a non-degenerate critical point at $s=0$ for each $x \in U_j$, so that  in view of the non-stationary phase theorem \cite[Theorem 7.7.1]{hoermanderI} we can assume that there are no other critical points by choosing the support of $f_j$  sufficiently small.  
Then, letting $\psi(x,0)=0 $ and setting $H(x,s)= \psi(x,s)- \eklm{\psi''(x,0) s, s}/2 $, one computes on $\pi^{-1}(U_j)$
\begin{align*}
\begin{split}
\pi_\ast(e^{i\psi/\mu}  \alpha_j)&=\int_{\R^l} e^{i\psi(x,s)/\mu} f_j(x,s) ds \cdot \beta_j=\int_{\R^l} e^{i \eklm{\psi''(x,0) s, s } /{2\mu}} e^{i H(x,s)/\mu} f_j(x,s) ds \cdot \beta_j\\
&= \sum_{k=0}^\infty \frac {i^k} {\mu^kk!} \int_{\R^l} e^{i \eklm{\psi''(x,0) s, s } /{2\mu }} { H(x,s) ^k} f_j(x,s) ds \cdot \beta_j.
\end{split}
\end{align*}
Note that it is permissible to interchange the  order of summation and integration, since  $H(x,s)=O(|s|^3)$, so that under the hypothesis $\supp_s f_j(x,\cdot) \subset B(0,1)$ one has for suitable $C>0$ the estimate
\bqn 
\Big | f_j(x,\cdot)\sum_{k=0}^{\tilde N} \frac{ H(x,\cdot )^k}{\mu^k k!} \Big | \leq  C| f_j(x,\cdot)| \sum_{k=0}^{\tilde N} \frac{1}{\mu^kk!}   \leq C e^{1/\mu}  |f_j(x,\cdot)|, \qquad \tilde N \in \N, 
\eqn
yielding an integrable majorand. 
Put $D_k=-i\gd_k$.  Taking into account
 \begin{align*}
 \int_{\R^l} \eklm{\xi,\psi''(x,0)^{-1} \xi}^r &\mathcal{F}\big ( H(x,\cdot )^k f_j(x,\cdot ) \big )(\xi) \d \xi =(2 \pi)^l \Big ( \eklm{D_s, \psi''(x,0)^{-1} D_s}^r  H(x,\cdot )^k f_j(x,\cdot ) \Big )(0)
 \end{align*}
 we obtain with Parseval's formula for arbitrary $\tilde N\in \N$ 
\begin{gather*}
\pi_\ast(e^{i\psi/\mu} \alpha_j)
= \frac {\beta_j} {\det( \psi''(x,0)/2\pi \mu i)^{1/2}}
\sum_{k=0}^\infty \frac {i^k} {(2\pi)^l\mu^kk!} \int_{\R^l} e^{-i \mu  \eklm{\psi''(x,0)^{-1} \xi, \xi } /{2}} \mathcal{F}\big ( H(x,\cdot )^k f_j(x,\cdot ) \big )(\xi) \d \xi  \\
= \frac {\beta_j} {\det( \psi''(x,0)/2\pi \mu i)^{1/2}}
\sum_{k=0}^\infty \frac {i^k} {\mu^kk! } \left   [ \sum_{r=0}^{\tilde N-1} \frac{(-i \mu)^r}{2^rr!} 
\left ( \eklm{D_s, {\psi''(x,0)^{-1}} D_s}^r  H(x,\cdot )^k f_j(x,\cdot ) \right )(0) \right. \\
\left. + \int_{\R^l} 
\sum_{r=\tilde N}^\infty \frac {(-i\mu )^r} {  (2\pi)^l2^rr!}\left ({\eklm{\psi''(x,0)^{-1} \xi, \xi }} \right )^r \mathcal{F}\big ( H(x,\cdot )^k f_j(x,\cdot ) \big )(\xi) \d \xi  \right ]. 
\end{gather*}
Note that interchanging integration and summation in the last term is in general not possible due to the lack of an integrable majorand. Since $H(x,s)$ vanishes of third order at $s=0$, the  local terms  are zero unless $3k\leq 2r$. Consequently, for general $\psi$ and arbitrary $N\in \N$ we arrive at
 \begin{align}
 \label{eq:67}
 \begin{split}
&\pi_\ast(e^{i\psi/\mu} \alpha_j)= \frac{e^{i\psi(x,0)/\mu} \cdot \beta_j }{\det (\psi''(x,0)/2\pi \mu i )^{1/2}} \\
\cdot 
 &\left [ \sum_{r-k\leq  N}  \mu^{r-k} \sum_{3k\leq 2r}  \frac{1}{r!\, k!\, 2^r\,  i^{r-k}} 
\left ( \eklm{D_s, {\psi''(x,0)^{-1}} D_s}^r  H(x,\cdot )^k f_j(x,\cdot ) \right )(0)+ R_{j,N+1} \right  ] ,
\end{split}
\end{align}
where $R_{j, N+1}$ is  explicitly given by
\begin{gather*}
R_{j,N+1}= \frac {e^{i\psi(x,0)/\mu} \cdot \beta_j} {\det( \psi''(x,0)/2\pi \mu i)^{1/2}} \\ \cdot
\sum_{k=0}^\infty \frac {i^k} {\mu^kk! } \int_{\R^l} 
\sum_{r=3 N+1}^\infty \frac {(-i\mu )^r} {  (2\pi)^l2^rr!}\left ({\eklm{\psi''(x,0)^{-1} \xi, \xi }} \right )^r \mathcal{F}\big ( H(x,\cdot )^k f_j(x,\cdot ) \big )(\xi) \d \xi. 
\end{gather*}
Moreover, by  \cite[Theorem 7.7.5]{hoermanderI} one has  $R_{j,N+1}=O(\mu^{N+1})$. The assertion now follows by integrating over $M$, and by taking $\det (\psi''(x,0)/2\pi \mu i )^{1/2}=(2\pi \mu)^{-l/2}|\det \psi''(x,0)|^{1/2} e^{\frac{-i\pi}4 \sigma_{\psi}}$  into account. In particular, the leading coefficient is given by 
\bq
\label{eq:63}
Q_0(\psi; \alpha,\tau)=\int_M \frac{ \tau \wedge r(\alpha) }{|\det \mathrm{Hess}_{trans} \, \psi|^{1/2}},
\eq
where the restriction map $r:\Lambda^q(E) \rightarrow \Lambda^{q-l}(M)$ is locally given  by 
\begin{align}
\label{eq:64}
h_j \,  (\pi^\ast \gamma_j) \wedge ds_{\sigma(1)}\wedge \cdots \wedge ds_{\sigma(p)} \quad &\longmapsto \quad
\begin{cases} 
\quad (-1) ^{\sgn \sigma}\iota^\ast(h_j)\, \gamma_j,  \qquad  p=l,\\
\quad 0, \qquad \quad \qquad p<l,
\end{cases}
\end{align}
 $\gamma_j \in \Lambda^{q-p}(U_j)$, $ h_j \in \Cinft(U_j\times \R^l)$, $\sigma$ being a permutation in $p$ variables.

\end{proof}

\begin{rem_anh} (1) In the proof of the last theorem, one can also use the lemma of Morse. This simplifies the proof, but gives less explicit expressions for the coefficients $Q_j$, since the Morse diffeomorphism is not given explicitly. Indeed, by  Morse's Lemma, we can choose the trivialization of $\pi:E \rightarrow M$  in such a way that
\bqn 
\psi(x,s)= \frac 12 \eklm{s, S_xs}, \qquad S_x\in \mathrm{Sym}(l,\R), \, \det S_x\not=0,
\eqn
where the symmetric matrix $S_x$ depends smoothly on $x\in U_j$.  Parseval's formula then yields
\begin{align*}
\pi_\ast(e^{i\psi/\mu} \alpha_j)
 &=\int_{\R^l} e^{i\psi(x,s)/\mu} f_j(x,s) ds \cdot \beta_j\\
 &=\frac{e^{i\pi \,\sgn S_x/4}\mu^{l/2}}{(2\pi)^{l/2}|\det S_x|^{1/2}} \int_{\R^l} e^{-i\mu  \eklm{S_x^{-1} \xi, \xi } /{2}} \mathcal{F}\big (  f_j(x,\cdot ) \big )(\xi) \d \xi   \cdot \beta_j\\
&= 
\frac{e^{i\pi \,\sgn S_x/4}\mu^{l/2}}{(2\pi)^{l/2}|\det S_x|^{1/2}} \left   [ (2\pi)^l \sum_{r=0}^{N-1} \frac{\mu^r}{r!} 
\left ( \eklm{D_s, \frac{S_x^{-1}}{2i} D_s}^r   f_j(x,\cdot ) \right )(x,0) \right. \\
&\left. + \int_{\R^l} 
\sum_{r=N}^\infty \frac{\mu^r}{r!} \Big (\frac{\eklm{S_x^{-1} \xi, \xi }} {2 i  }\Big )^r \mathcal{F}\big (  f_j(x,\cdot ) \big )(\xi) \d \xi  \right ]  \cdot \beta_j. 
\end{align*}
By integrating over $M$, the assertion of Theorem  \ref{thm:SP} follows. 

(2) In general, it is not possible to say anything about the convergence of the sum in \eqref{eq:8}  as $N \to \infty$, and consequently, about the limit $\lim_{N \to \infty} R_N(\mu)$, due to  the lack of control of the growth of the derivatives $\gd_s^\alpha   f_j(x,0)$ as $|\alpha| \to \infty$.
\end{rem_anh}

From Theorem \ref{thm:SP} we can now infer the classical generalized stationary phase theorem.

\begin{theorem_anh}[Generalized stationary phase theorem for manifolds]
\label{thm:SPM}
Let $M$ be a  $n$-dimensional, orientable Riemannian manifold with volume form $dM$,  $\psi \in \Cinft(M)$  a real valued phase function,  $\mu >0$, and set
\bqn
I({\mu})=\int_M e^{i\psi(m)/\mu} a(m) \, dM(m),
\eqn
where $a(m)\in \CT(M)$ denotes  a compactly supported function on $M$. Let
$$\mathcal C=\mklm{m \in M: \psi_\ast:T_mM \rightarrow T_{\psi(m)}\R \text{ is zero}}$$
 be the  critical set of the phase function $\psi$, and assume that $\mathcal C$ is clean in the sense that 
\begin{enumerate}
\item $\Ccal$ is a smooth submanifold of $M$  of dimension $p$ in a neighborhood of the support of $a$;
\item for all $m \in \Ccal$, the restriction $\psi''(m)_{|N_m\Ccal}$ of the Hessian of $\psi$ at the point $m$  to the normal space $N_m\Ccal$ is a non-degenerate quadratic form.
\end{enumerate}
\noindent
Then, for all $N \in \N$, there exists a constant $C_{N,\psi}>0$ such that
\bqn
|I(\mu) - e^{i\psi_0/\mu}  e^{ \frac{i \pi}4\sigma_{\psi}}(2\pi \mu)^{\frac {n-p}{2}}\sum_{j=0} ^{N-1} \mu^j Q_j (\psi;a)| \leq C_{N,\psi} \mu^N  \sup _{l\leq 2N} \norm{D^l a }_{\infty,M},
\eqn
where $D^l$ is a differential operator on $M$ of order $l$ and $\psi_0$ the constant value of $\psi$ on $\Ccal$, while $\sigma_{\psi}$   denotes the constant value of the signature of the transversal Hessian $\mathrm{Hess} \, \psi(m)_{|N_m\Ccal}$ on $\Ccal$. The coefficients $Q_j$ can be computed explicitly, and for each $j$ there exists a constant $\tilde C_{j,\psi}>0$ such that 
\bqn
|Q_j(\psi;a)|\leq \tilde C_{j,\psi}  \sup _{l\leq 2j} \norm{D^l a }_{\infty,\Ccal}.
\eqn
In particular,
\bqn
Q_0(\psi;a)= \int _{\Ccal} \frac {a(m)}{|\det \mathrm{Hess} \, \psi(m)_{|N_m\Ccal}|^{1/2}} d\sigma_{\Ccal}(m),
\eqn
where $d\sigma_{\mathcal C}$ is the induced volume form on $\mathcal C$.
\end{theorem_anh}
\begin{proof}
Due to the non-stationary phase principle, we can assume that $a \,dM$ is supported in a tubular neighborhood of $\Ccal$. Identifying the latter  with the total space $N\Ccal$ of the normal bundle of $ \Ccal$, the assertion follows with Theorem \ref{thm:SP}.
\end{proof}



\providecommand{\bysame}{\leavevmode\hbox to3em{\hrulefill}\thinspace}
\providecommand{\MR}{\relax\ifhmode\unskip\space\fi MR }
\providecommand{\MRhref}[2]{%
  \href{http://www.ams.org/mathscinet-getitem?mr=#1}{#2}
}
\providecommand{\href}[2]{#2}

\end{document}